%% file: main.tex
\documentclass[twoside]{article}
\usepackage{amsthm}

\usepackage[preprint]{aistats2027}
\usepackage{amssymb}
\usepackage[hidelinks]{hyperref}
\usepackage{bm}
\usepackage{booktabs}
\newtheorem{assumption}{Assumption}
\newtheorem{definition}{Definition}
\newtheorem{theorem}{Theorem}
\newtheorem{corollary}{Corollary}
\newtheorem{proposition}{Proposition}
\newtheorem{lemma}{Lemma}
\usepackage[round]{natbib}

\newcommand{\R}{\mathbb{R}}
\newcommand{\E}{\mathbb{E}}
\newcommand{\TV}{\mathrm{TV}}
\newcommand{\KL}{\mathrm{KL}}
\newcommand{\law}{\mathrm{Law}}

\newcommand{\diag}{\mathrm{diag}}

\newcommand{\N}{\mathcal{N}}

\newcommand{\norm}[1]{\left\|#1\right\|}

\newcommand{\bLambda}{\bm{\Lambda}}

\newcommand{\op}{\mathrm{op}}
\newcommand{\Cov}{\operatorname{Cov}}

\DeclareMathOperator{\tr}{tr}

\usepackage{times}
\usepackage{tikz}
\usetikzlibrary{positioning,arrows.meta,fit,backgrounds,calc}
\begin{document}

\runningtitle{How Many Directions Must a Truncated Diffusion Sampler Retain?}
\runningauthor{Radmehr Karimian, Ali Mohades, Johannes Lederer}

%

%

\twocolumn[

\aistatstitle{How Many Directions Must a Truncated Diffusion Sampler Retain?\\
Matching Bounds Under Power-Law Spectra}

\aistatsauthor{ Radmehr Karimian \And Ali Mohades \And Johannes Lederer }

\aistatsaddress{ Université de Genève \And Universität Hamburg \And Universität Hamburg } ]

\begin{abstract}
Diffusion samplers can reduce computation by generating selected spectral coordinates and filling the remaining directions with noise. How many directions must they retain? We study this question for data with power-law covariance spectra. For Gaussian data compared to a smoothed target, we prove matching bounds on the required number of retained directions, provided that the ambient dimension is sufficiently large. The truncation error depends on the combined Wiener gains of the omitted directions, regardless of the accuracy of the sampler on the retained coordinates. Keeping only directions whose signal exceeds the output noise level can therefore leave a non-vanishing error: many individually weak directions remain significant in aggregate. Combining this characterization with a diffusion convergence bound yields sufficient sampling-step complexity under exact scores. The upper bounds also extend to estimated principal components and, componentwise, to Gaussian mixtures. The practical prescription is to select the retained subspace using an aggregate spectral-tail error budget, then to choose the diffusion noise level accordingly.
\end{abstract}

\section{Introduction}
\label{sec:introduction}
Diffusion models generate data by learning to reverse a process that
gradually adds noise \citep{song2019generative,song2020score}. Their
computational cost typically grows with the dimension of the space in
which the diffusion runs, motivating approaches that exploit alternative
representations of the data. Examples include diffusion in wavelet
coordinates \citep{guth2022wavelet}, frequency-dependent processes such as
blurring diffusion and inverse-heat-dissipation models
\citep{hoogeboom2023blurring,rissanen2023generative}, and diffusion
restricted to lower-dimensional subspaces \citep{jing2022subspace}.
Related approaches operate in learned autoencoder representations,
including VAE-based latent spaces \citep{rombach2022},
which need not correspond to orthogonal transformations. Here, we focus
on orthogonal representations, such as principal-component, Fourier,
and wavelet bases, in which the distribution of signal across directions
can be described explicitly.

An orthogonal change of basis preserves the ambient dimension, but it
can reveal a structure that a sampler can exploit. For images, a central
example is spectral decay: natural-image spectra are often approximately
described by power laws
\citep{field1987relations,ruderman1994statistics,hyvarinen2009natural}.
In a corresponding coefficient model, the variances satisfy
$\sigma_i^2\asymp i^{-\beta}$, so a small number of directions carry much
of the variance, while many others contain progressively weaker signals.
This unequal distribution of signals suggests that the effort needed to
generate different coordinates should depend on their statistical
importance.

These observations motivate the question studied in this paper:
how does spectral structure affect the convergence of diffusion samplers
in an orthogonal representation and when can it reduce the sampling
complexity? More specifically, under power-law spectral decay, how many
directions must a sampler retain to achieve a prescribed accuracy at a
given resolution, and how many sampling steps are sufficient on those
directions? Answering these questions requires controlling both the
convergence of the sampler on the retained coordinates and the
approximation error from replacing the remaining coordinates with
noise. Our analysis connects these two costs, quantifying when sampling
complexity can be governed by the number of required directions rather
than by the ambient dimension.
\paragraph{Contributions.}
\begin{itemize}
\item \textbf{The exact cost of truncation.} For Gaussian data, replacing the
discarded directions by noise costs, in total variation, a quantity of the
order of the Euclidean norm of their Wiener gains
$g_i=\sigma_i^2/(\sigma_i^2+\eta^2)$, with matching upper and lower bounds
(Theorem~\ref{thm:truncation-cost}). The lower bound holds for every sampler
of this kind, however accurate it is on the retained directions.
 
\item \textbf{How many directions must be kept.} Under a power law, accuracy
$\varepsilon$ at resolution $\eta$ requires, and is achieved with 
$K\asymp_\beta\bigl(C/(\eta^2\varepsilon)\bigr)^{2/(2\beta-1)}$ retained
directions (Theorem~\ref{thm:num-directions}), strictly more than the number
of directions above the noise level.
 
\item \textbf{End-to-end sampling complexity.} For the reverse diffusion on
the retained coordinates, with direction-dependent rates, the number of steps
depends on the number of retained directions and not on the ambient dimension
(Theorem~\ref{thm:convergence}). Choosing the diffusion noise level from $K$
gives $\widetilde O\bigl(\varepsilon^{-4\beta/(2\beta-1)}\bigr)$ reverse steps
with an exact score. The upper bounds extend to the principal components
estimated from the data and to patch-based Gaussian mixtures
(Section~\ref{sec:extensions}).
\end{itemize}
 
\paragraph{Related work.}
Our lower bound is based on the two-sided total-variation bound for Gaussians
of \citet{devroye2018total}; what we add is the identification of the
relevant covariance ratios with the Wiener gains of the discarded directions,
and a projection argument that makes the bound hold for every truncated
sampler. The convergence analysis extends the Girsanov argument of
\citet{chen2022sampling} to anisotropic Ornstein-Uhlenbeck processes. In
the retained subspace, isotropic diffusion attains the same step count up to
constants (Proposition~\ref{prop:isotropic}), so the gain comes from choosing
which directions to keep, not from how they are diffused. Preconditioned
samplers \citep{ma2022accelerating} use anisotropy to speed up sampling
rather than to decide what to discard. A longer discussion is in
Appendix~\ref{app:related}.
 
\paragraph{Organisation.}
Section~\ref{sec:setup} introduces the data model and the truncated sampler.
Section~\ref{sec:sharp} proves the sharp bounds on truncation, and
Section~\ref{sec:convergence} the convergence of the sampler and its
end-to-end complexity. Section~\ref{sec:extensions} treats empirical bases
and Gaussian mixtures. Section ~\ref{sec:experiments} reports experiments, and Section ~\ref{sec:conclusion} discusses implications.
\section{Problem Setup and Truncated Sampler}
\label{sec:setup}
 
We first fix the data model, then the diffusion that generates the retained
coordinates, and finally the class of truncated samplers and the target
against which they are judged.
 
\subsection{Data model}
\label{sec:setup-data}
\label{sec:gauss_model}

The covariance spectrum describes how the signal is distributed in different directions.
We use a power-law envelope to allow unequal variances without imposing a
low-dimensional support; Gaussianity is an additional condition for the sharp
truncation results.

Let $X\in\R^d$ be a centered random vector with law $q$ and covariance
$C=V\diag(\sigma_1^2,\ldots,\sigma_d^2)V^\top$, where $V=[v_1,\ldots,v_d]$ is
orthogonal and $\sigma_1^2\ge\cdots\ge\sigma_d^2\ge0$. In our application
$X$ is an image patch of dimension $d=D_p$ and $V$ is its principal-component
basis. A full image split into $N_P$ non-overlapping patches is handled patch
by patch; generating all patches jointly replaces $K$ by $N_PK$ below, while
generating them independently adds a cross-patch error that we do not analyze
(Appendix~\ref{app:patches}).
 
\begin{assumption}[Power-law spectrum]
\label{ass:powerlaw}
There are constants $0<c_{\mathrm{sp}}\le C_{\mathrm{sp}}$ and $\beta>1$ such
that
\begin{equation}\label{eq:powerlaw}
 c_{\mathrm{sp}}\,i^{-\beta}\;\le\;\sigma_i^2\;\le\;C_{\mathrm{sp}}\,i^{-\beta},
 \qquad i=1,\ldots,d.
\end{equation}
\end{assumption}
 
The two-sided envelope allows for departures from an exact power law by fixed
multiplicative factors. Every direction may carry positive variance, so no
rank cut-off is imposed. The condition $\beta>1$ makes the total variance
bounded independently of $d$, which is the property used in the convergence
bound. It is not needed for the Gaussian comparison itself: the truncation
rates require only $\beta>1/2$ (Appendix~\ref{app:beta}). The sharp results of Section~\ref{sec:sharp}
further assume that the data are Gaussian with an exact power law,
\begin{equation}\label{eq:powerlaw2}
 q=\N(0,C),
 \qquad
 \sigma_i^2=C_{\mathrm{sp}}\,i^{-\beta},
\end{equation}
which is the case $c_{\mathrm{sp}}=C_{\mathrm{sp}}$ of
Assumption~\ref{ass:powerlaw}. Gaussian patches are also the components of
the patch Gaussian-mixture models commonly used for natural images
\citep{zoran2012natural}, which we treat in Section~\ref{sec:extensions}.
For natural images, the exponent should be estimated from patch covariances;
two-dimensional Fourier spectra suggest values of $\beta$ close to one
(Appendix~\ref{app:beta}).
 
\subsection{Diffusion on the retained coordinates}
\label{sec:setup-diffusion}

Once the retained subspace is chosen, sampling becomes a problem in the $K$
coordinates. We specify the forward process and its reverse discretisation
here so that the later convergence bound applies to an explicit sampler.

Fix a number $K\le d$ of the retained directions, let
$V_K=[v_1,\ldots,v_K]$, and write $U=V_K^\top X\in\R^K$ for the retained
coefficients, with the law $q_K$. (The convergence theory of
Section~\ref{sec:convergence} holds for any $V_K$ with orthonormal columns.)
The diffusion runs on these coefficients rather than in the data space,
where it would be degenerate (Appendix~\ref{app:pixel}).
The retained coefficients are diffused by the diagonal Ornstein-Uhlenbeck
process
\begin{equation}\label{eq:forward-OU-coeff}
 \mathrm dU_t=-\Lambda U_t\,\mathrm dt+\sigma\sqrt{2\Lambda}\,\mathrm dB_t,
 \qquad U_0\sim q_K,
\end{equation}
for $t\in[0,T]$, where $\Lambda=\diag(\lambda_1,\ldots,\lambda_K)$ with $\lambda_i>0$, and
$(B_t)$ is a standard Brownian motion in $\R^K$. Each coordinate relaxes at
its own rate towards the stationary law $\gamma_K=\N(0,\sigma^2I_K)$; we write
$q_{K,t}$ for the law of $U_t$. Equal rates $\lambda_i\equiv\lambda_0$ give
the standard Ornstein-Uhlenbeck process of score-based diffusion models
\citep{song2020score}, up to a time change and a rescaling.
 
\paragraph{Wiener gains and rates.}
For a noise level $s>0$, the Wiener gain of direction $i$ is
\begin{equation}\label{eq:wiener-gain}
 g_i(s):=\frac{\sigma_i^2}{\sigma_i^2+s^2}\in[0,1),
\end{equation}
the optimal linear shrinkage factor for recovering the $i$-th coefficient
from an observation corrupted by $\N(0,s^2)$ noise; for arbitrary data it
also bounds the data-dependent part of the score at noise level $s$
(Appendix~\ref{app:wiener}). It is close to one when
the signal dominates and close to the signal-to-noise ratio $\sigma_i^2/s^2$
when the noise dominates. The Wiener gain appears twice in this paper: at the
diffusion noise level $\sigma$ it sets the rates,
\begin{equation}\label{eq:wiener-rates}
 \lambda_i=\lambda_0\,g_i(\sigma),
 \qquad i=1,\ldots,K,
\end{equation}
and at the evaluation level $\eta$ of Section~\ref{sec:setup-truncated} it
measures the cost of discarding a direction (Section~\ref{sec:sharp}). The
number of directions whose variance exceeds the diffusion noise level is, up
to the envelope constants,
\begin{equation}\label{eq:deff}
 d_{\mathrm{eff}}^{\sigma}
 :=\min\Bigl\{d,\ \bigl\lfloor(C_{\mathrm{sp}}/\sigma^2)^{1/\beta}\bigr\rfloor\Bigr\}.
\end{equation}
If $K\le d_{\mathrm{eff}}^{\sigma}$, then
$\sigma_i^2\ge(c_{\mathrm{sp}}/C_{\mathrm{sp}})\sigma^2$ for every retained
direction, so all rates \eqref{eq:wiener-rates} lie in
$[\lambda_0c_{\mathrm{sp}}/(c_{\mathrm{sp}}+C_{\mathrm{sp}}),\lambda_0]$ and
their ratio $\kappa:=\lambda_{\max}/\lambda_{\min}$ is at most
$(c_{\mathrm{sp}}+C_{\mathrm{sp}})/c_{\mathrm{sp}}$, which equals $2$ for an
exact power law.
 
\paragraph{Reverse sampler.}
Let $t_k=kT/N$ for $k=0,\ldots,N$ and let $t_k^\leftarrow=T-t_{N-k}$ be the
reverse grid. Given score estimates
$\widehat s_{t_j}\approx\nabla\log q_{K,t_j}$ for $j=1,\ldots,N$, the sampler
starts from $\widehat U_0\sim\gamma_K$ and, at each interval
$[t_k^\leftarrow,t_{k+1}^\leftarrow)$, solves
\begin{equation}\label{eq:frozen-drift-coeff}
 \mathrm d\widehat U_t
 =\Bigl(\Lambda\widehat U_t
 +2\sigma^2\Lambda\,\widehat s_{t_{N-k}}\bigl(\widehat U_{t_k^\leftarrow}\bigr)\Bigr)\mathrm dt
 +\sigma\sqrt{2\Lambda}\,\mathrm d\overline B_t .
\end{equation}
This is the exact reverse-time dynamics of \eqref{eq:forward-OU-coeff} with
the score frozen at the start of each interval, an exponential-integrator
scheme that can be simulated exactly (explicit update in
Appendix~\ref{app:reverse}). We write $p:=\law(\widehat U_T)$ for
the law it produces on $\R^K$, and measure the step size by
$\delta:=\max_{k,i}\bigl(1-e^{-\lambda_i(t_k-t_{k-1})}\bigr)$.
 
\subsection{Truncated samplers and their target}
\label{sec:setup-truncated}
 
A truncated sampler generates the retained coefficients and fills every other
direction with noise. To compare it with the full-dimensional data, we add
Gaussian noise of a fixed level $\eta>0$ to all outputs: given
$\widehat U\sim p$, it returns
\[
 \widetilde X^\eta=V_K\bigl(\widehat U+\xi\bigr)+V_\perp Z,
\]
where $\xi\sim\N(0,\eta^2I_K)$, $Z\sim\N(0,\eta^2I_{d-K})$,
$V_\perp=[v_{K+1},\ldots,v_d]$, and $\widehat U,\xi,Z$ are independent.
 
\begin{definition}[Truncated samplers]
\label{def:truncated-class}
For $K\le d$, $\eta>0$ and a probability law $p$ on $\R^K$, the truncated
sampler is the law
\begin{equation}\label{eq:output-law}
 \widetilde p^\eta
 =\bigl(p*\N(0,\eta^2I_K)\bigr)\otimes\N(0,\eta^2I_{d-K}),
\end{equation}
written in the coordinates of the basis $V$. They are compared with the
smoothed data law, expressed in the same coordinates,
\begin{equation}\label{eq:smoothed-target}
 q^\eta:=q*\N(0,\eta^2I_d)=\law(X+\eta Z'),
\end{equation}
where $Z'\sim\N(0,I_d)$ is independent of $X$.
The ideal truncated sampler is the choice $p=q_K$; we denote its law by
$\widetilde q_K^{\,\eta}$.
\end{definition}
 
The level $\eta$ acts as a resolution: a sampler is judged on the structure
that remains visible after noise of level $\eta$ is added, and is not required
to reproduce finer detail. Some such smoothing is unavoidable, because
without it a truncated sampler produces laws supported on a $K$-dimensional
subspace, whose total-variation distance to a full-rank data law equals one
however accurate the sampler is. The same smoothed target arises for
early-stopped diffusion samplers, which return, up to a rescaling, the data
law corrupted by a small amount of Gaussian noise
\citep{chen2022sampling,benton2023nearly}; Appendix~\ref{app:early-stop}
makes this precise.
 
For every $p$, the output \eqref{eq:output-law} and the ideal truncated
sampler share their last $d-K$ factor, and convolving with a common Gaussian
does not increase total variation (Appendix~\ref{app:decomposition}). Hence
\begin{equation}\label{eq:error-decomposition}
 \TV\bigl(\widetilde p^\eta,q^\eta\bigr)
 \;\le\;
 \underbrace{\TV\bigl(p,q_K\bigr)}_{\text{sampling error}}
 \;+\;
 \underbrace{\TV\bigl(\widetilde q_K^{\,\eta},q^\eta\bigr)}_{\text{truncation error}} .
\end{equation}
The sampling error depends on the reverse sampler and is the subject of
Section~\ref{sec:convergence}. The truncation error depends only on $q$,
$K$ and $\eta$; Section~\ref{sec:sharp} computes it up to universal constants
for Gaussian data, and shows that it is then also a lower bound for every
truncated sampler. For non-Gaussian data the truncation error additionally
contains the dependence between retained and discarded coordinates
(Section~\ref{sec:extensions}).
 
\paragraph{Two noise levels.}
The diffusion noise level $\sigma$ belongs to the sampler: it sets the
stationary law, the rates \eqref{eq:wiener-rates} and the count
$d_{\mathrm{eff}}^{\sigma}$. The evaluation level $\eta$ belongs to the
target. The two are chosen separately: Section~\ref{sec:sharp} shows that
$\eta$ and the accuracy determine how many directions must be retained, and
Section~\ref{sec:convergence} then chooses $\sigma$ small enough that all of
them lie above the diffusion noise level. Together, the two sections answer
the question of this paper: \emph{given a resolution $\eta$ and an accuracy
$\varepsilon$, how many directions $K$ must a truncated sampler retain, and
how many reverse steps $N$ does the sampler \eqref{eq:frozen-drift-coeff}
need, to achieve
$\TV(\widetilde p^\eta,q^\eta)\le\varepsilon$}?
\section{Sharp Bounds on Spectral Truncation}
\label{sec:sharp}

This section computes the truncation error in \eqref{eq:error-decomposition}
for Gaussian data $q=\N(0,C)$. Nothing here depends on how the retained
coordinates are generated, so the results apply to every truncated sampler of
Definition~\ref{def:truncated-class}. Throughout, we write
\[
 x_i:=\frac{\sigma_i^2}{\eta^2}
 \qquad\text{and}\qquad
 g_i:=g_i(\eta)=\frac{x_i}{1+x_i}
\]
for the signal-to-noise ratio and the Wiener gain of direction $i$ at the
evaluation level $\eta$.

\subsection{The cost of truncation}

We first isolate the error that remains even when the retained coordinates
are sampled exactly. The same quantity will give a lower bound for every
sampler in our class, because all of them use the same discarded-coordinate
distribution.

\begin{theorem}[Cost of spectral truncation]
\label{thm:truncation-cost}
Let $q=\N(0,C)$, $1\le K<d$ and $\eta>0$.
\begin{enumerate}
\item[(a)] With $G_K:=\min\bigl\{1,\|(g_i)_{i>K}\|_2\bigr\}$, the ideal
truncated sampler satisfies
\[
 \tfrac1{100}\,G_K
 \;\le\;
 \TV\bigl(\widetilde q_K^{\,\eta},q^\eta\bigr)
 \;\le\;
 \tfrac32\,G_K,
\]
and, with explicit constants,
$\TV(\widetilde q_K^{\,\eta},q^\eta)\le\bigl(\frac18\sum_{i>K}x_i^2/(1+x_i)\bigr)^{1/2}$.
\item[(b)] For every probability law $p$ on $\R^K$, the truncated sampler
$\widetilde p^\eta$ of \eqref{eq:output-law} satisfies
\[
 \TV\bigl(\widetilde p^\eta,q^\eta\bigr)
 \;\ge\;
 \TV\bigl(\widetilde q_K^{\,\eta},q^\eta\bigr).
\]
\end{enumerate}
\end{theorem}

Part (a) determines the truncation error up to universal constants. The
second upper bound is the one we use in Theorem~\ref{thm:num-directions},
because its constants are explicit; since $x_i^2/(1+x_i)=g_i^2(1+x_i)\le2g_i^2$
when $x_i\le1$, it is of the same order when the discarded directions lie
below the evaluation level. The two-sided bound
follows from the total-variation theorem for Gaussians of
\citet{devroye2018total}: the two covariances differ only in the discarded
directions, where the eigenvalues of the covariance ratio minus the identity
are exactly $-g_i$. The explicit upper bound follows from the exact
divergence
$\KL(\widetilde q_K^{\,\eta}\|q^\eta)=\frac12\sum_{i>K}\bigl[-g_i-\log(1-g_i)\bigr]$
and Pinsker's inequality. Part (b) holds because every truncated sampler
outputs the same noise $\N(0,\eta^2)$ in the discarded directions, and total
variation cannot increase when we look only at those coordinates. Proofs are
in Appendix~\ref{app:proofs-sharp}.

Theorem~\ref{thm:truncation-cost} has three consequences. First, the cost of
discarding a direction is a denoising quantity: it depends on the direction
only through its Wiener gain at the evaluation level. A direction whose
signal dominates the noise ($g_i$ close to one) cannot be dropped without a
constant error. Second, the costs of weak directions add up in the
Euclidean norm, as $\sum_{i>K}x_i^2$, rather than through the total discarded
variance $\sum_{i>K}\sigma_i^2$: a weak direction costs much less than its
variance suggests, but many weak directions together are not free. Third, by
part (b), no accuracy on the retained coordinates can compensate for
discarded ones; the bound holds even for a perfect sampler. It concerns the
construction of Section~\ref{sec:setup-truncated}, in which the discarded
directions are filled with Gaussian noise, and is not a lower bound for models
that represent those directions in some other way.

\subsection{How many directions must be kept}

The preceding comparison turns a choice of subspace into a spectral-tail
error budget. We now convert that budget into a number of retained directions,
giving both a sufficient choice and a matching necessary order.
We specialise to the exact power law \eqref{eq:powerlaw2} and write
\[
 k_\eta:=\bigl\lfloor(C_{\mathrm{sp}}/\eta^2)^{1/\beta}\bigr\rfloor
\]
for the number of directions whose variance is at least $\eta^2$, that is,
whose signal-to-noise ratio is at least one.

\begin{theorem}[Number of retained directions]
\label{thm:num-directions}
Let $q=\N(0,C)$ with the exact power law \eqref{eq:powerlaw2}, and let
$\eta>0$.
\begin{enumerate}
\item[(a)] \emph{Sufficiency.} For every $1\le K<d$,
\begin{equation}\label{eq:sharp-tv}
 \TV\bigl(\widetilde q_K^{\,\eta},q^\eta\bigr)
 \le\frac{C_{\mathrm{sp}}}{\eta^2\sqrt{8(2\beta-1)}}\,K^{-(2\beta-1)/2}.
\end{equation}
In particular, the truncation error is at most $\varepsilon/2$ whenever
$K_\varepsilon(\eta)\le K<d$, where
\[
 K_\varepsilon(\eta)
 :=\biggl\lceil
 \Bigl(\frac{C_{\mathrm{sp}}}{\eta^2\,\varepsilon\sqrt{2(2\beta-1)}}\Bigr)^{2/(2\beta-1)}
 \biggr\rceil.
\]
\item[(b)] \emph{Necessity.} Let $\varepsilon<1/200$ and $2K\le d$. If some
truncated sampler that retains $K$ directions satisfies
$\TV(\widetilde p^\eta,q^\eta)\le\varepsilon$, then
\[
 K\ge k_\eta
 \qquad\text{and}\qquad
 K\ge\Bigl(\frac{C_{\mathrm{sp}}}{2^{\beta}\cdot200\,\eta^2\,\varepsilon}\Bigr)^{2/(2\beta-1)}.
\]
\item[(c)] \emph{Where to cut.} If $k_\eta\le K$ and $2K\le d$, then the
first discarded direction has $x_{K+1}<1$, and
\begin{align*}
 \TV\bigl(\widetilde q_K^{\,\eta},q^\eta\bigr)
 &\ge\tfrac1{100}\min\bigl\{1,\ 2^{-\beta-1}x_{K+1}\sqrt K\bigr\},\\
 \TV\bigl(\widetilde q_K^{\,\eta},q^\eta\bigr)
 &\le\tfrac32\min\bigl\{1,\ c_\beta\,x_{K+1}\sqrt K\bigr\},
\end{align*}
with $c_\beta:=\sqrt{(2\beta+1)/(2\beta-1)}$.
\end{enumerate}
\end{theorem}

Parts (a) and (b) show that, for accuracy $\varepsilon$ at resolution $\eta$,
the number of directions a truncated sampler must retain is
\[
 K=\Theta_\beta\Bigl(\bigl(C_{\mathrm{sp}}/(\eta^2\varepsilon)\bigr)^{2/(2\beta-1)}\Bigr),
\]
whenever this number is at most $d/2$. The implicit constants in (a) and (b)
differ by a large $\beta$-dependent factor, inherited mainly from the
constant $1/100$ of \citet{devroye2018total}; the dependence on $\eta$ and
$\varepsilon$ is what matches. This is larger than the number $k_\eta$
of directions above the evaluation level by a factor of order
$(k_\eta/\varepsilon^2)^{1/(2\beta-1)}$, so the gap grows both as the accuracy
improves and as the resolution becomes finer. Part (c) turns this
into a cut-off rule: accuracy $\varepsilon$ requires, and is achieved by,
discarding directions from the point where the signal-to-noise ratio falls to
order $\varepsilon/\sqrt K$, far below one. The natural rule of discarding
every direction below the evaluation level fails.

\begin{corollary}[Truncating at the noise level is not enough]
\label{cor:noise-floor}
Under the exact power law, let $K=k_\eta\ge1$ and $2k_\eta\le d$. Every
truncated sampler that retains exactly the directions with variance at least
$\eta^2$ satisfies
\[
 \TV\bigl(\widetilde p^\eta,q^\eta\bigr)
 \;\ge\;
 \frac1{100}\min\bigl\{1,\ 2^{-\beta-1}\sqrt{k_\eta}\bigr\},
\]
which is at least $1/100$ once $k_\eta\ge4^{\beta+1}$, however accurate the
sampler is in the retained directions.
\end{corollary}

The reason is that under a power law the directions just below the evaluation
level are as numerous as those above it, and each has a Wiener gain bounded
below by a constant. The proofs of Theorem~\ref{thm:num-directions} and
Corollary~\ref{cor:noise-floor} are in Appendix~\ref{app:proofs-sharp}; they
use the power law only through tail sums $\sum_{i>K}i^{-2\beta}$ and therefore
remain valid for $\beta>1/2$ (Appendix~\ref{app:beta}). Section~\ref{sec:convergence}
supplies the other half of \eqref{eq:error-decomposition}: a sampler that
reaches the sampling error $\varepsilon/2$ in the $K_\varepsilon(\eta)$ retained
directions.
\section{Diffusion Convergence and Sampling Complexity}
\label{sec:convergence}

This section bounds the sampling error $\TV(p,q_K)$ in
\eqref{eq:error-decomposition} for the reverse sampler
\eqref{eq:frozen-drift-coeff}, and then combines it with
Section~\ref{sec:sharp} into an end-to-end guarantee.

\subsection{Assumptions}

The assumptions separate properties of the target distribution from the
accuracy and growth of the score used by the sampler. They concern only the
 coefficients retained in $\R^K$ and follow the Girsanov framework of
\citet{chen2022sampling}. We state them for sampling down to time $0$.
For data without a smooth density, Appendix~\ref{app:proofs-convergence}
instead targets a positive-time marginal: regularity is then required of that
smoothed law, rather than of the raw data. Positive-time smoothing alone does
not assert the uniform Lipschitz bound below for every data law.

\begin{assumption}[Score regularity]
\label{ass:lipschitz}
For every $t\in[0,T]$, the law $q_{K,t}$ has a $C^2$ density whose score
$\nabla\log q_{K,t}$ is $L$-Lipschitz, and the tail and regularity conditions
of Appendix~\ref{app:general-rates} hold.
\end{assumption}

\begin{assumption}[Score accuracy]
\label{ass:score_error}
The score estimates used by the sampler satisfy
\[
 \frac1N\sum_{j=1}^N
 \E_{U\sim q_{K,t_j}}\bigl\|\widehat s_{t_j}(U)-\nabla\log q_{K,t_j}(U)\bigr\|_\Lambda^2
 \le\varepsilon_{\mathrm{score},\Lambda}^2,
\]
where $\|z\|_\Lambda^2:=z^\top\Lambda z$.
\end{assumption}

\begin{assumption}[Admissibility]
\label{ass:girsanov}
The change of measure between the exact reverse process and the sampler is
valid, i.e.\ the corresponding stochastic exponential is a true martingale.
This holds, for instance, when the score estimates grow at most linearly,
which can be enforced by clipping, and for exact Gaussian scores
(Appendix~\ref{app:general-rates}).
\end{assumption}

\begin{assumption}[Finite initial divergence]
\label{ass:finite-kl}
$\KL(q_K\|\gamma_K)<\infty$.
\end{assumption}

\paragraph{Why these conditions are reasonable.}
The Gaussian model provides an explicit example satisfying all four
conditions. Its forward variance in coordinate $i$ is
\[
 v_i(t)=e^{-2\lambda_i t}\sigma_i^2
   +(1-e^{-2\lambda_i t})\sigma^2,
\]
so the score is $-u_i/v_i(t)$ and is uniformly Lipschitz with
$L\le\max\{\sigma_K^{-2},\sigma^{-2}\}$. For the exact power law and
$K\le d_{\mathrm{eff}}^\sigma$, this gives $L\le\sigma^{-2}$.
Gaussian tails justify the integrations by parts; positive variances give
a finite initial KL; and the exact score is linear, hence admissible, with zero
score error. Thus, the end-to-end Gaussian result verifies these conditions
directly rather than imposing an additional regularity model.

For learned scores, Assumption~\ref{ass:score_error} specifies the accuracy
needed by the sampler. Its weighted mean-square norm is exactly the norm in
the Girsanov drift mismatch, so errors in slower directions receive less
weight. It is an accuracy requirement on the score, not a statement about how many
training samples are needed to achieve it. A linear-growth constraint or clipping supplies admissibility;
the accuracy condition must hold for the score actually used after any such
modification. Finally, the spectral assumption already supplies the moment
bound $\E\|U\|^2=\sum_{i\le K}\sigma_i^2\le C_{\mathrm{sp}}\zeta(\beta)$,
where $\zeta(\beta)=\sum_{i\ge1}i^{-\beta}$.

\subsection{Convergence on the retained coordinates}

With these conditions in place, we can bound the sampling part of the error.
The next theorem separates initialization, score estimation and time
discretisation, making explicit how each depends on the retained dimension.

Recall $\kappa_0:=(c_{\mathrm{sp}}+C_{\mathrm{sp}})/c_{\mathrm{sp}}$ from
Section~\ref{sec:setup-diffusion}, and let
$\lambda_\star:=\lambda_0c_{\mathrm{sp}}/(c_{\mathrm{sp}}+C_{\mathrm{sp}})$
be the resulting lower bound on the rates.

\begin{theorem}[Convergence on the retained coordinates]
\label{thm:convergence}
Let Assumption~\ref{ass:powerlaw} hold, let $K\le d_{\mathrm{eff}}^{\sigma}$,
use the Wiener-gain rates \eqref{eq:wiener-rates}, and let
Assumptions~\ref{ass:lipschitz}-\ref{ass:finite-kl} hold. If
\begin{equation}\label{eq:small-step-cor2}
 \delta\le\min\Bigl\{\frac18,\ \frac1{128\sqrt{\kappa_0}\,L\sigma^2}\Bigr\},
\end{equation}
then, for a constant $C'$ that depends only on $\beta$ and $\kappa_0$,
\begin{align}
 \TV(p,q_K)
 &\le e^{-\lambda_\star T}\sqrt{\tfrac12\KL(q_K\|\gamma_K)}
 \label{eq:tv-cor2}\\
 &\quad+C'\sigma\sqrt T\Bigl(
  \varepsilon_{\mathrm{score},\Lambda}
  +L\sigma\sqrt{\lambda_0\delta K}
 \nonumber\\
 &\qquad
  +L\delta\sqrt{\lambda_0C_{\mathrm{sp}}\zeta(\beta)+\lambda_0\sigma^2K}
 \nonumber\\
 &\qquad
  +(1+L\sigma^2)\sqrt L\,\delta\sqrt{\lambda_0K}
 \Bigr).
 \nonumber
\end{align}
In particular, if
$\varepsilon_{\mathrm{score},\Lambda}\le c_1\varepsilon/(\sigma\sqrt T)$ for a
small universal constant $c_1$ and
\[
 T\ge\lambda_\star^{-1}\log\bigl(2\sqrt{\KL(q_K\|\gamma_K)/2}\,/\varepsilon\bigr),
\]
then $N=\widetilde O(K/\varepsilon^2)$ steps on a uniform grid give
$\TV(p,q_K)\le\varepsilon$. Here $\widetilde O$ hides logarithmic factors and
the constants $\lambda_0,\sigma,L,\beta,\kappa_0$, but not $K$ or $d$.
\end{theorem}

The ambient dimension $d$ does not appear: the bound depends on the data only
through the number $K$ of retained directions and the summable spectrum.
Keeping $K\le d_{\mathrm{eff}}^{\sigma}$ matters because it keeps every rate
above $\lambda_\star$, so the mixing term decays at a rate independent of
$K$; a direction with a very small rate would slow down the mixing of the
whole process. The proof, in Appendix~\ref{app:proofs-convergence}, applies a
general bound for arbitrary diagonal rates (Theorem~\ref{thm:tv-trace}), in
which dimension enters only through $\tr(\Lambda)$ and
$\E\|U\|_\Lambda^2$. That bound extends the Girsanov argument of
\citet{chen2022sampling}: the drift mismatch is controlled in the
$\Lambda$-weighted norm by a one-step perturbation estimate for the score of a
diagonal Ornstein-Uhlenbeck step. On the retained subspace, equal rates
$\lambda_i\equiv\lambda_0$ satisfy the same bound up to constants
(Proposition~\ref{prop:isotropic}); what reduces the cost from $d$ to $K$ is
the choice of the retained directions, not the choice of rates.

\subsection{End-to-end complexity}

We now allocate the total error budget between truncation and sampling.
Theorems~\ref{thm:num-directions} and~\ref{thm:convergence} combine
through \eqref{eq:error-decomposition}: first choose $K$ to control the
discarded tail, then choose the reverse horizon and grid to sample those
$K$ coordinates accurately. The remaining choice is the diffusion
noise level $\sigma$, which also enters the Lipschitz constant and the
initial divergence.

\begin{theorem}[End-to-end complexity]
\label{thm:end-to-end}
Let $q=\N(0,C)$ with the exact power law \eqref{eq:powerlaw2}, let $\eta>0$
and $\varepsilon\in(0,1)$, and let the sampler use the exact scores
$\widehat s_{t_j}=\nabla\log q_{K,t_j}$. Set
\[
 K:=K_\varepsilon(\eta),
 \qquad
 \sigma^2:=C_{\mathrm{sp}}K^{-\beta},
\]
\[
 T:=\frac2{\lambda_0}\log\frac{4\sqrt{C_{\mathrm{sp}}\zeta(\beta)}}{\sigma\varepsilon},
\]
assume $K<d$, and use a uniform grid with
\[
 N\ \ge\ c\,\max\Bigl\{
 \frac{\lambda_0^2T^2K}{\varepsilon^2},\
 \frac{\lambda_0^{3/2}T^{3/2}\sqrt{C_{\mathrm{sp}}\zeta(\beta)}}{\sigma\,\varepsilon},\
 \lambda_0T
 \Bigr\}
\]
for an absolute constant $c$. Then the truncated sampler satisfies
$\TV(\widetilde p^\eta,q^\eta)\le\varepsilon$, with
\[
 N=\widetilde O\bigl(\varepsilon^{-4\beta/(2\beta-1)}\bigr),
\]
where $\widetilde O$ hides logarithmic factors and constants depending only
on $\lambda_0$, $C_{\mathrm{sp}}$, $\beta$ and $\eta$.
\end{theorem}

The theorem makes the roles of the two noise levels concrete. The evaluation
level $\eta$ and the accuracy $\varepsilon$ fix how many directions are
retained (Theorem~\ref{thm:num-directions}); this number in turn fixes the
diffusion noise level at the variance of the last retained direction, so that
every retained direction lies above it and all rates are at least
$\lambda_0/2$. Since $\sigma^2/\eta^2=(C_{\mathrm{sp}}/\eta^2)K^{-\beta}
\approx(k_\eta/K)^\beta$, and $K$ is typically much larger than $k_\eta$
(Section~\ref{sec:sharp}), the diffusion runs at a much finer noise level than
the one at which its output is judged. The dominant term
in $N$ is $K/\varepsilon^2$; the term in $1/\sigma$, which grows with $K$
through the choice of $\sigma$, is of lower order. By
Theorem~\ref{thm:num-directions}(b), when $\varepsilon<1/200$ and $2K\le d$
no truncated sampler can retain a number of directions of smaller order in
$\eta$ and $\varepsilon$, so any further improvement would have to come from
the step count, not from the truncation. The theorem assumes exact scores and counts discretisation steps
only; with learned scores, Theorem~\ref{thm:convergence} adds the score error
$\varepsilon_{\mathrm{score},\Lambda}$. The proof is in
Appendix~\ref{app:proofs-convergence}.
\section{Extensions Beyond the Population Gaussian Model}
\label{sec:extensions}

The sharp truncation and end-to-end results use Gaussian data and a known
principal-component basis; the retained-coordinate convergence theorem also
applies to non-Gaussian laws. We now explain how the error decomposition
accommodates general data, estimated bases and Gaussian mixtures. Precise
statements and proofs are in
Appendix~\ref{app:extensions}.

\paragraph{Non-Gaussian data.}
Theorem~\ref{thm:convergence} does not use Gaussianity: for any data law
with a power-law spectrum that satisfies its assumptions, the sampling error on the retained coordinates is
at most $\varepsilon$ after $\widetilde O(K/\varepsilon^2)$ steps. What
changes is the truncation error. Let a truncated sampler output
$V_K\widehat U+Z$, with $\widehat U\sim p$ and the discarded coordinates $Z$
drawn independently from a law $\nu_\perp$.

\begin{proposition}[Truncation error for general data]
\label{prop:full-bound}
For any data law $q$ and any law $\nu_\perp$ on the discarded coordinates,
the output law $\widetilde p=p\otimes\nu_\perp$ satisfies
\[
 \TV(\widetilde p,q)\le\TV(p,q_K)+\varepsilon_{\mathrm{PCA}},
\]
where $\varepsilon_{\mathrm{PCA}}:=\TV(q_K\otimes\nu_\perp,\,q)$.
If $\nu_\perp$ is the true law $q_\perp$ of the discarded coordinates, then
$\varepsilon_{\mathrm{PCA}}\le\sqrt{I/2}$, where $I$ is the mutual information
between the retained and the discarded coordinates.
\end{proposition}

The truncation error $\varepsilon_{\mathrm{PCA}}$ is small when two things
hold: the retained and discarded coordinates are nearly independent, and the
residual law $\nu_\perp$ is close to $q_\perp$. For Gaussian data the first
holds exactly in the principal-component basis, and the second is what
Section~\ref{sec:sharp} quantifies for the residual $\N(0,\eta^2)$. Extending the matching lower bound of
Theorem~\ref{thm:truncation-cost} to non-Gaussian laws is a direction for
further work. The Wiener gain nevertheless retains a distribution-free
interpretation for each coordinate: it bounds the data-dependent part of the score at the corresponding noise
level (Proposition~\ref{prop:learnable}).

\paragraph{Estimated principal components.}
In practice the basis and the rates are estimated from $n_{\mathrm{patch}}$
training patches.
Suppose the sampler uses the exact score of the process defined by the
estimated basis and rates (an oracle assumption, since this score depends on
the unknown covariance). Then, with probability at least $1-2e^{-d}$ and
under a small-perturbation condition, Theorem~\ref{thm:empirical} gives
\begin{align*}
 \TV\bigl(\widetilde p^{\,\eta}_{\mathrm{emp}},q^\eta\bigr)
 &\le
 \widetilde O\Bigl(\sqrt{\tfrac KN}\Bigr)
 +C_1\,\frac{\sigma_1^2\sqrt K}{\eta^2\,\delta_{\mathrm{gap}}}\,\Delta_n\\
 &\quad+\TV\bigl(\widetilde q_K^{\,\eta},q^\eta\bigr),
\end{align*}
where $\Delta_n=C_\Delta\sigma_1^2\sqrt{d/n_{\mathrm{patch}}}$,
$\delta_{\mathrm{gap}}=\sigma_K^2-\sigma_{K+1}^2$, and $C_1$ does not
depend on $d$. Compared with the population case, the only new term is the
error from estimating the basis. It is below $\varepsilon/4$ once
$n_{\mathrm{patch}}$ exceeds
a threshold that depends on the spectrum, $\sigma$, $\eta$, $\varepsilon$ and
$d$ but not on the number of reverse steps (Theorem~\ref{thm:decouple}), so for a
fixed diffusion noise level the choices of $K$ and $N$ from the population
analysis remain valid. Under the exact power law,
$\delta_{\mathrm{gap}}\asymp K^{-\beta-1}$, so resolving larger subspaces
requires more training data. The result gives a conservative consistency
guarantee at fixed $\sigma$. The joint choice of $\sigma$ and $K$ in
Theorem~\ref{thm:end-to-end} is established for the population basis;
extending it to estimated bases requires tracking the noise dependence of
the empirical sampling term.

\paragraph{Patch Gaussian mixtures.}
Natural-image patches are well described by Gaussian mixtures
\citep{zoran2012natural}, $q=\sum_{m}\pi_m\N(0,C_m)$. Running, with
probability $\pi_m$, a truncated sampler adapted to component $m$ in its own
principal-component basis gives, by convexity of total variation,
\[
 \TV\bigl(\widetilde p^{\,\eta},q^\eta\bigr)\le\sum_m\pi_m\varepsilon_m,
\]
where $\varepsilon_m$ is the error for component $m$
(Proposition~\ref{prop:gmm-lift}). Theorems~\ref{thm:num-directions}
and~\ref{thm:end-to-end} therefore apply component by component, for
components whose spectra satisfy their assumptions, with a number of retained
directions chosen for each component. This is an upper
bound only: the lower bound of Theorem~\ref{thm:truncation-cost} does not
transfer to the mixture, because the error of a mixture can be smaller than
the average error of its components. A basis shared across components adds
an error measuring how far the component covariances are from being jointly
diagonal (Appendix~\ref{app:extensions}).

\paragraph{Whole images.}
All results apply patch by patch. Generating all patches jointly replaces
$K$ by $N_PK$; generating them independently adds a cross-patch error that we
do not analyze (Appendix~\ref{app:patches}).
\section{Experimental Results}
\label{sec:experiments}

We test the three modeling ingredients of the paper, a power-law spectrum,
the Gaussian truncation bounds and Gaussian mixtures, in three settings
(Appendix~\ref{app:exp}). (i)~Covariance spectra of natural-image
patches from CIFAR-10/100 \citep{krizhevsky2009learning}, DIV2K
\citep{Agustsson_2017_CVPR_Workshops} and BSDS500 \citep{MartinFTM01}, with
patch sizes from $8\times8$ to $32\times32$, to check Assumption~\ref{ass:powerlaw}
out of sample; audio frames from FMA \citep{fma_dataset} are included for comparison(Appendix~\ref{app:exp}). (ii)~Synthetic Gaussian data with an exact power law, on which every result of Sections~3-4 is checked
(Appendix~\ref{app:exp-gaussian}). (iii)~$8\times8\times3$ patches of
CIFAR-10 and CIFAR-100 modeled by Gaussian mixtures with up to $32$ components
\citep{zoran2012natural}. In (ii) and (iii) all densities are in closed form,
so truncation errors are estimated by Monte Carlo without a learned
discriminator. The code is available
\textcolor{magenta}{\href{https://anonymous.4open.science/r/Truncated-Diffusion-Sampler-Under-Power-Law-67BB/README.md}{here}}.




\paragraph{Experiment 1: Power-law spectra.}
On all ten image cases the power law fits the patch spectra $2.4$-$20$
times better than an exponential law, so the spectra are not
low-rank-like. The fitted exponents lie in $\hat\beta\in[1.24,1.85]$,
decreasing with patch size towards the Fourier value $\beta\approx1$ of
Appendix~B.5. On data not used to fit $\hat\beta$, the envelope ratio
$C/c$ is at most $2.42$ on the fitting window and $4.11$ on the spectral
bulk. The spectra depart from a single power law in two ways. They bend
gradually, which a broken power law captures and which satisfies
Assumption~1 by construction (Proposition~\ref{prop:broken-envelope}).
Beyond the bulk they also roll off below the envelope, which affects only
the lower side, and the analysis needs the lower side only on the bulk
(Appendix~\ref{app:e1}). The components of the CIFAR mixtures used in
Experiment~3 also satisfy Assumption~1, with
$\hat\beta_m\in[1.56,2.57]$ and envelope ratios at most $1.89$
(Appendix~\ref{app:exp-data}).

\paragraph{Experiment 2: Synthetic Gaussian data.}
For $\beta\in\{1.25,1.5,2\}$ and $d=2^{18}$, every bound of
Sections~3-4 holds at every tested point. The truncation error is
$0.276$-$0.302$ times $\|(g_i)_{i>K}\|_2$. This is close to the constant
$1/(2\sqrt\pi)\approx0.282$ predicted by a central-limit heuristic
(Appendix~\ref{app:exp-constant}) and far inside the constants $1/100$
and $3/2$ of Theorem~1(a). The fitted exponents of $K^\star$ in
$1/\varepsilon$ and in $C_{\mathrm{sp}}/\eta^2$ match $2/(2\beta-1)$ to
within $1\%$. Whenever $2K^\star\le d$, $K^\star$ changes by at most
$6.7\%$ over $d\in\{2^{12},\dots,2^{18}\}$. The sufficient choice of
Theorem~2(a) exceeds $K^\star$ by a factor $1.15$-$1.51$. The necessary
bound of Theorem~2(b) lies a factor $26$-$528$ below $K^\star$, which
reflects its constants. Truncating at the noise level leaves an error of
$0.57$ at $k_\eta=16$ ($\beta=1.5$), and of at least $0.99$ for every
$\beta$ at $k_\eta=256$. With exact scores in the settings of Theorem~4,
the total error reaches $\varepsilon$ after $32$-$128$ reverse steps,
far fewer than the theorem requires. The number of steps needed for
sampling accuracy grows like $K^{0.58}$, not linearly, and equal rates
show the same exponent (Appendix~\ref{app:e2}).

\paragraph{Experiment 3: CIFAR patch mixtures.}
The mixture error matches the Gaussian prediction at coarse resolution
(median ratio $1.03$-$1.05$ at $k_\eta=2$). At fine resolution it exceeds
the prediction (median $1.80$ at $k_\eta=32$), tracking how non-Gaussian
the smoothed data are, as Proposition~1 allows. Because the error decays
steeply in $K$, the required number of directions is nevertheless within
$19\%$ of the Gaussian prediction. At coarse resolution it scales as
$\varepsilon^{-0.79}$ to $\varepsilon^{-0.80}$, against
$\varepsilon^{-0.75}$ to $\varepsilon^{-0.76}$ from Theorem~2 with the
fitted exponent. Retaining only the $k_\eta$ directions above the noise
level leaves errors of $0.16$-$0.56$, consistent with Corollary~1 although
its hypotheses are not met here. Among rules that choose $K$ before
training, the explicit bound of Theorem~1(a), evaluated on the measured
covariance, is the most economical: it uses a median of $1.11K^\star$
directions and misses $13$ of $60$ targets, by at most a factor $1.78$.
Theorem~2(a) with a fitted power law always meets the target, with a
median of $1.5K^\star$ directions. Explained-variance thresholds miss
the target at fine resolution and keep far too many directions at coarse
resolution (Appendices~\ref{app:exp-gaussian}-\ref{app:exp-rules}). The
bound of Proposition~4 holds in every setting, with the mixture error
between $0.58$ and $0.88$ times the bound. Truncating each component in
its own basis uses $8$-$31\%$ fewer directions than a shared basis at
$\varepsilon=0.01$ (Appendix~\ref{app:exp-components}).
\section{Discussion and Conclusion}
\label{sec:conclusion}

Truncating a diffusion model to a few spectral directions is attractive
because natural data concentrate their variance, but it has a cost that no
amount of training can remove: every truncated sampler pays at least the
error of its discarded directions (Theorem~\ref{thm:truncation-cost}). We
quantify that cost and hence where to cut. Under a power-law spectrum the
directions just below the noise level are as numerous as those above it, so
their errors add up, and keeping only what is visible above the noise leaves
a non-vanishing error (Corollary~\ref{cor:noise-floor}).

\paragraph{What this means in practice.}
The results suggest four rules for designing a truncated sampler.
\begin{itemize}
\item \emph{Set the truncation level by the target accuracy, not by the noise
level.} Retain directions until their signal-to-noise ratio at the
evaluation level falls to order $\varepsilon/\sqrt K$
(Theorem~\ref{thm:num-directions}(c)). The resulting number of directions,
of order $(C_{\mathrm{sp}}/(\eta^2\varepsilon))^{2/(2\beta-1)}$, can be
estimated before training from the eigenvalues of the patch covariance, up to
constants that depend on $\beta$.
\item \emph{Run the diffusion below the evaluation level.} Once the number of
retained directions is fixed, the diffusion noise level should be set near the
variance of the last retained direction, so that all retained directions stay
above it and mix at comparable rates (Theorem~\ref{thm:end-to-end}).
\item \emph{Spend effort on the choice of directions rather than on the
rates.} On a fixed retained subspace, direction-dependent rates improve the
sampling bound only by constants over equal rates
(Proposition~\ref{prop:isotropic}); what reduces the cost from the ambient
dimension to the number of retained directions is which directions are kept.
\item \emph{Diagnose which error dominates.} The error splits into a
sampling part and a truncation part \eqref{eq:error-decomposition}. A better
score network reduces only the first; if the second dominates, the remedy is
to retain more directions or to model the discarded ones.
\end{itemize}

\paragraph{Scope and limitations.}
The Gaussian replacement model isolates the cost of spectral truncation and
admits matching bounds. For general data and mixtures,
Section~\ref{sec:extensions} provides upper bounds with explicit residual
modeling terms. The end-to-end theorem uses exact scores to separate
sampling and truncation from score learning; Theorem~\ref{thm:convergence}
quantifies the additional error for learned scores. The condition $\beta>1$
ensures a dimension-independent moment bound, while the truncation analysis
extends to $\beta>1/2$. Image spectra near $\beta=1$ therefore call for
retaining the finite spectral sums described in Appendix~\ref{app:beta}.
The empirical PCA result establishes consistency under an oracle score and a
small-perturbation condition. Patch-level guarantees extend to jointly
modeled patches; independent generation incurs the cross-patch term in
Appendix~\ref{app:patches}.

\paragraph{Further directions.}
The error decomposition suggests several extensions. For non-Gaussian data
or learned residuals, useful lower bounds would account for both dependence
and residual modeling accuracy; such results would complement the present
Gaussian replacement benchmark. A coordinate-wise contraction analysis
could also clarify when direction-dependent rates improve on equal rates,
especially if weak directions are retained at small rates. Finally, wavelets
\citep{guth2022wavelet} and projection-pursuit directions
\citep{friedman1974projection,huber1985projection} offer alternative bases
in which to study the same resolution-dependent trade-off.

\section*{AI Statement}
We used Claude Opus 5.5, an AI assistant developed by Anthropic, in the following ways.

\paragraph{Manuscript review.} Claude reviewed a draft of the paper in the
style of a conference review and pointed out weaknesses
in positioning, presentation and constants. All resulting changes were decided on and made by the authors.

\paragraph{Experiments: figures, tables and text.} The authors designed and ran
all experiments. Claude analyzed their output files, generated the figures and
tables of Section~6 and Appendix~\ref{app:exp} directly from those files, and wrote the
text describing the results. Every number in the text was checked against the
experimental outputs, and the authors reviewed and edited all text.

\paragraph{Theory and other writing.} Opus 5.5 was used to improve the grammar and
clarity of author-written text in Sections~1-5; all such edits were reviewed
by the authors.
No AI tool was used to generate data. The authors take full responsibility
for the content of the paper, including the correctness of its results and the
accuracy of its citations.

\bibliography{bib} 

\clearpage
\appendix
\thispagestyle{empty}

\onecolumn
\aistatstitle{Appendix}

\paragraph{Organisation of the appendix.}
Appendix~\ref{app:related} discusses related work.
Appendix~\ref{app:setup} gives details of the setup of Section~\ref{sec:setup}.
Appendices~\ref{app:proofs-sharp}, \ref{app:proofs-convergence}
and~\ref{app:extensions} contain the proofs and full statements for
Sections~\ref{sec:sharp}, \ref{sec:convergence} and~\ref{sec:extensions}.
The remaining appendices contain the technical lemmas: the estimates used in
the Girsanov argument (Appendix~\ref{sec:app:lemmas}), the proofs of the
general convergence theorems (Appendix~\ref{sec:app:C}), the covariance and
principal-component perturbation lemmas (Appendix~\ref{app:lemmas}), and the
proofs of the empirical and decoupling results
(Appendix~\ref{app:proofs-empirical}). Experimental details are available in Appendix~\ref{app:exp}.

\input{roadmap_figure}

\section{Extended Related Work and Positioning}
\label{app:related}

\subsection{Convergence theory for score-based diffusion models}

Score-based diffusion models estimate the scores of noised data marginals and
use them in a reverse-time stochastic differential equation
\citep{anderson1982reverse,song2019generative,song2020score}. Polynomial
convergence guarantees were obtained by
\citet{lee2022convergence,chen2022sampling,chen2023improved}, and
\citet{benton2023nearly} proved bounds that are nearly linear in the ambient
dimension under finite second moments only. These results measure the cost of
sampling in the dimension of the space in which the diffusion runs. Our
convergence analysis (Theorem~\ref{thm:convergence}) extends the Girsanov
argument of \citet{chen2022sampling} to diagonal anisotropic
Ornstein-Uhlenbeck processes, keeping the rate matrix throughout, so that
the dimension enters through $\tr(\Lambda)$ and the weighted second moment
(Theorem~\ref{thm:tv-trace}).

\paragraph{Relation to truncation baselines.}
In the regime where our rates are used, all retained rates lie within a
constant factor of $\lambda_0$, so $\tr(\Lambda)=\Theta(K)$. The step count
of Theorem~\ref{thm:convergence} therefore coincides, up to constants, with
what an existing dimension-linear bound gives when applied to the
$K$-dimensional coefficient law, for instance \citet{benton2023nearly} or
Theorem~\ref{thm:tv-dim} with $r=K$; Proposition~\ref{prop:isotropic} states
this formally for equal rates on the same retained subspace. Our
contribution is therefore not a faster sampler for a fixed $K$. It is the
analysis of the truncation itself: the cost of discarding directions
(Theorem~\ref{thm:truncation-cost}), the number of directions that must be
retained (Theorem~\ref{thm:num-directions}), and the resulting end-to-end
prescription (Theorem~\ref{thm:end-to-end}). Whether direction-dependent rates
can strictly improve on equal rates, for example by retaining weak directions
with small rates, is left open (Section~\ref{sec:conclusion}).

\subsection{Low-dimensional and adaptive diffusion theory}

Several works show that diffusion models adapt to low-dimensional structure.
\citet{chen2023score} analyze score approximation and distribution recovery
for data supported on a low-dimensional subspace, and \citet{liang2025low}
obtain total-variation convergence governed by an intrinsic dimension.
analyzes under the manifold hypothesis include
\citet{debortoli2022convergence,potaptchik2024linear,azangulov2024convergence}.
A power-law covariance is different: it is generally full rank, so there is
no intrinsic dimension to adapt to. Instead it defines a resolution-dependent
count of relevant directions, which grows as the target resolution becomes
finer. Our results quantify this count and show that it is larger than the
number of directions visible above the noise level.

\subsection{Spectral, subspace and multiscale diffusion}

Diffusion processes that treat spectral directions differently have been
proposed and studied empirically. Subspace diffusion restricts later stages
of the diffusion to lower-dimensional subspaces, such as principal or
low-frequency components \citep{jing2022subspace}. Blurring and
inverse-heat-dissipation models damp high frequencies faster than low ones
\citep{hoogeboom2023blurring,rissanen2023generative}. Wavelet-based score
models generate coefficients scale by scale \citep{guth2022wavelet}. Matrix
preconditioning has been used to accelerate score-based sampling
\citep{ma2022accelerating}; there anisotropy serves speed, whereas here it
enters the forward and reverse processes together and the question is which
directions can be discarded. Our patch principal-component process with
Wiener-gain rates belongs to this family. We do not propose a new
architecture, and we do not claim that the Wiener-gain profile is an optimal
schedule; our contribution is a theory that says how many directions such
processes must retain and what discarding the rest costs.

\subsection{Gaussian total-variation bounds and denoising}

Our lower bound relies on the two-sided total-variation bound for Gaussians of
\citet{devroye2018total}, and the interpretation of the Wiener gain as the
size of the data-dependent part of the score
(Proposition~\ref{prop:learnable}) is an elementary consequence of Tweedie's
formula; the Wiener gain itself is the classical optimal linear shrinkage
factor \citep{torralba2024foundations}. What we add is the combination of
these tools into statements about truncated diffusion samplers.
\begin{enumerate}
\item[(i)] For samplers that fill the discarded directions with noise, the
eigenvalues that govern the Gaussian comparison are exactly the Wiener gains
of the discarded directions, so the truncation error is a denoising quantity
(Theorem~\ref{thm:truncation-cost}(a)).
\item[(ii)] A projection argument turns the comparison into a lower bound for
\emph{every} sampler in this class: no accuracy on the retained directions
compensates for discarded ones (Theorem~\ref{thm:truncation-cost}(b)).
\item[(iii)] Under a power law this yields matching necessary and sufficient
numbers of retained directions and an interpretable cut-off rule: the first
discarded direction should have signal-to-noise ratio of order
$\varepsilon/\sqrt K$, not of order one, so cutting at the noise level is
provably insufficient (Theorem~\ref{thm:num-directions},
Corollary~\ref{cor:noise-floor}).
\item[(iv)] Combined with the convergence analysis, the result gives an
end-to-end prescription in which the resolution $\eta$ and accuracy
$\varepsilon$ determine $K$, and $K$ in turn determines a diffusion noise
level well below $\eta$ (Theorem~\ref{thm:end-to-end}).
\end{enumerate}
The sharp truncation results assume Gaussian data and apply regardless of
the active sampler; exact scores are used for the end-to-end step count.
The mixture extension gives an upper bound. Extending the matching lower
bounds beyond Gaussian replacement is a direction for further work.

\subsection{Natural-image statistics and patch models}

Natural images have approximately power-law spectra
\citep{field1987relations,ruderman1994statistics,hyvarinen2009natural}, and
patch Gaussian mixtures are a classical model of local image statistics
\citep{zoran2012natural}. We use these facts to motivate the coefficient
model of Section~\ref{sec:setup-data}; they do not imply that a complete image
is Gaussian or that its patch law factorises, and the corresponding residual
and cross-patch errors remain explicit in our statements
(Appendix~\ref{app:patches}). The exponent that enters our bounds is the
decay rate of the sorted patch-covariance eigenvalues, which differs from the
Fourier exponent (Appendix~\ref{app:beta}). Projection pursuit
\citep{friedman1974projection,huber1985projection} offers an alternative to
principal components for choosing directions, which we leave for future work.

\section{Details of the Setup}
\label{app:setup}

This appendix collects the material behind Section~\ref{sec:setup}: notation,
the reverse-time dynamics and an explicit form of the sampler, the
pixel-space and patch-level views, the exponent $\beta$ for images, the link
between the smoothed target and early stopping, the proof of the error
decomposition, and the relation between Wiener gains and scores.

\subsection{Notation}
\label{app:notation}

\begin{center}
\small
\begin{tabular}{@{}lp{0.62\linewidth}l@{}}
\hline
Symbol & Meaning & Defined in\\
\hline
$d$ & ambient dimension (patch dimension $D_p$ in the application) & \S\ref{sec:setup-data}\\
$C,\ \sigma_i^2,\ v_i$ & data covariance, its eigenvalues and eigenvectors & \S\ref{sec:setup-data}\\
$\beta,\ c_{\mathrm{sp}},\ C_{\mathrm{sp}}$ & power-law exponent and envelope constants & \eqref{eq:powerlaw}\\
$K,\ V_K,\ q_K$ & number of retained directions, their basis, law of their coefficients & \S\ref{sec:setup-diffusion}\\
$\sigma$ & diffusion noise level (stationary law $\gamma_K=\N(0,\sigma^2I_K)$) & \eqref{eq:forward-OU-coeff}\\
$\eta$ & evaluation level (resolution of the target) & \eqref{eq:smoothed-target}\\
$g_i(s)$ & Wiener gain $\sigma_i^2/(\sigma_i^2+s^2)$ at noise level $s$ & \eqref{eq:wiener-gain}\\
$\lambda_0,\ \lambda_i,\ \Lambda,\ \kappa$ & base rate, rates, rate matrix, rate ratio $\lambda_{\max}/\lambda_{\min}$ & \eqref{eq:wiener-rates}\\
$d_{\mathrm{eff}}^{\sigma}$ & number of directions above the diffusion noise level & \eqref{eq:deff}\\
$T,\ N,\ \delta$ & time horizon, number of reverse steps, step-size parameter & \eqref{eq:frozen-drift-coeff}\\
$\widehat s_{t_j}$ & score estimates, $j=1,\ldots,N$ & \eqref{eq:frozen-drift-coeff}\\
$p,\ \widetilde p^\eta$ & law of the sampler on $\R^K$, and of the truncated output & \eqref{eq:output-law}\\
$q^\eta,\ \widetilde q_K^{\,\eta}$ & smoothed data law, ideal truncated sampler & Def.~\ref{def:truncated-class}\\
$\varepsilon$ & target accuracy in total variation & \S\ref{sec:setup-truncated}\\
$N_P$ & number of patches in an image & App.~\ref{app:patches}\\
\hline
\end{tabular}
\end{center}

\subsection{Reverse-time dynamics and the explicit sampler update}
\label{app:reverse}

For a forward diffusion $\mathrm dX_t=b_t(X_t)\,\mathrm dt+G\,\mathrm dB_t$
with constant invertible $G$, the time-reversed process
$\overline X_t:=X_{T-t}$ solves, under mild regularity conditions
\citep{anderson1982reverse},
\[
 \mathrm d\overline X_t
 =\bigl(-b_{T-t}(\overline X_t)+GG^\top\nabla\log q_{T-t}(\overline X_t)\bigr)\mathrm dt
 +G\,\mathrm d\overline B_t,
\]
where $q_t$ is the law of $X_t$. For the forward process
\eqref{eq:forward-OU-coeff}, $b_t(u)=-\Lambda u$ and $G=\sigma\sqrt{2\Lambda}$,
so the exact reverse drift is
\[
 b(t,u)=\Lambda u+2\sigma^2\Lambda\nabla\log q_{K,T-t}(u),
\]
and the sampler \eqref{eq:frozen-drift-coeff} replaces
$\nabla\log q_{K,T-t}(u)$ by the estimate $\widehat s_{t_{N-k}}$ evaluated at
the start $\widehat U_{t_k^\leftarrow}$ of the current interval. The estimates
used are therefore exactly $\widehat s_{t_N},\ldots,\widehat s_{t_1}$; this is
the index set in the score-error assumption of
Section~\ref{sec:convergence}. The sampler is started at
$\gamma_K=\N(0,\sigma^2I_K)$ instead of the unknown $q_{K,T}$; the resulting
error is the mixing term of the convergence theorems.

On an interval of length $h$ the frozen score $s:=\widehat
s_{t_{N-k}}(\widehat U_{t_k^\leftarrow})$ is constant, and each coordinate of
\eqref{eq:frozen-drift-coeff} is a linear stochastic differential equation.
Writing $W=\widehat U+2\sigma^2s$, which solves
$\mathrm dW_i=\lambda_iW_i\,\mathrm dt+\sigma\sqrt{2\lambda_i}\,\mathrm d\overline
B_{t,i}$, gives the exact update
\begin{equation}\label{eq:sampler-update}
 \widehat U_{t_{k+1}^\leftarrow,i}
 =e^{\lambda_ih}\,\widehat U_{t_k^\leftarrow,i}
 +2\sigma^2\bigl(e^{\lambda_ih}-1\bigr)s_i
 +\sigma\sqrt{e^{2\lambda_ih}-1}\;\zeta_i,
 \qquad \zeta_i\sim\N(0,1)\ \text{i.i.d.},
\end{equation}
so the sampler can be run without further discretisation error. On the
uniform grid $h=T/N$, the step-size parameter is
$\delta=1-e^{-\lambda_{\max}T/N}\le\lambda_{\max}T/N$.

\subsection{Pixel-space view}
\label{app:pixel}

The retained coefficients map back to the data space through
$\widehat X_t:=V_KU_t$. With $\widetilde\Lambda:=V_K\Lambda V_K^\top$ one may
formally write
$\mathrm d\widehat X_t=-\widetilde\Lambda\widehat X_t\,\mathrm dt
+\sigma\sqrt{2\widetilde\Lambda}\,\mathrm dW_t$, but for $K<d$ this equation
is degenerate: $\widetilde\Lambda$ has rank $K$ and the process stays in the
span of $V_K$. The reverse-time formula and the Girsanov argument of
Section~\ref{sec:convergence} require an invertible diffusion coefficient, so
all of the analysis is carried out for the non-degenerate coefficient process
\eqref{eq:forward-OU-coeff}, and data-space statements follow by the push-forward
$u\mapsto V_Ku$, which is injective and hence preserves total variation.

\subsection{Images made of patches}
\label{app:patches}

Let an image $x\in\R^{N_PD_p}$ consist of $N_P$ non-overlapping, centered
patches $x^{(j)}\in\R^{D_p}$ (in practice the empirical patch mean is removed
before training and restored after generation). The basis $V$ is the
eigenbasis of the pooled patch covariance
$\Sigma_p=\E\,x^{(j)}(x^{(j)})^\top$, and each patch is represented by its
retained coefficients $u^{(j)}=V_K^\top x^{(j)}$.

\emph{Joint generation.} Stacking the coefficients of all patches gives a
vector in $\R^{N_PK}$ with basis $I_{N_P}\otimes V_K$ and rate matrix
$I_{N_P}\otimes\Lambda$. All results apply with $K$ replaced by $N_PK$; the
trace quantities acquire a factor $N_P$, e.g.
$\tr(I_{N_P}\otimes\Lambda)=N_P\tr(\Lambda)$. The number of reverse steps
therefore still grows linearly with the number of patches, hence with the
number of pixels; what the spectral analysis improves is the factor per
patch, from $D_p$ to the number of retained directions.

\emph{Independent generation.} If patches are trained or generated
independently, the sampler targets the product of the patch marginals, and a
full-image guarantee must add
\[
 \varepsilon_{\mathrm{cross}}
 :=\TV\Bigl(\law\bigl(u^{(1)},\ldots,u^{(N_P)}\bigr),\
 \textstyle\bigotimes_{j=1}^{N_P}\law\bigl(u^{(j)}\bigr)\Bigr).
\]
We do not analyze $\varepsilon_{\mathrm{cross}}$. Non-overlapping patch bases
also ignore correlations between patches, which under a power-law image
spectrum carry much of the low-frequency energy; a global basis would capture
them.

\subsection{The exponent \texorpdfstring{$\beta$}{beta} for image patches}
\label{app:beta}

The exponent in Assumption~\ref{ass:powerlaw} is the decay exponent of the
\emph{sorted} covariance eigenvalues, not of the radial Fourier spectrum. For
a stationary two-dimensional field with power spectrum
$S(f)\propto|f|^{-\alpha}$, the number of frequencies with $|f|\le F$ grows
like $F^2$, so the $i$-th largest eigenvalue behaves like $i^{-\alpha/2}$.
The classical estimate $\alpha\approx2$ for natural images
\citep{field1987relations,ruderman1994statistics} therefore suggests
$\beta\approx1$ for image patches, at the boundary of
Assumption~\ref{ass:powerlaw}, and $\beta$ should be estimated from patch
covariances before the bounds are read quantitatively.

The two halves of the analysis depend on $\beta$ differently. The truncation
bounds of Section~\ref{sec:sharp} use the power law only through tail sums
$\sum_{i>K}i^{-2\beta}$, which are finite whenever $\beta>1/2$. The
convergence bounds of Section~\ref{sec:convergence} use $\beta>1$ through
$\zeta(\beta)=\sum_{i\ge1}i^{-\beta}$, which bounds the weighted second moment
and the initial divergence; for $\beta\le1$ these quantities are instead
bounded by the partial sum $\sum_{i\le K}i^{-\beta}$, which grows like $\log K$
for $\beta=1$ and like $K^{1-\beta}$ for $\beta<1$.

\subsection{The smoothed target and early stopping}
\label{app:early-stop}

Run the isotropic process $\mathrm dX_t=-\lambda_0X_t\,\mathrm dt
+\sigma\sqrt{2\lambda_0}\,\mathrm dB_t$ on all $d$ coordinates from
$X_0\sim q$ up to a time $\tau_0>0$. Since
$X_{\tau_0}=e^{-\lambda_0\tau_0}X_0+\sigma\sqrt{1-e^{-2\lambda_0\tau_0}}\,Z$
with $Z\sim\N(0,I_d)$ independent of $X_0$,
\begin{equation}\label{eq:early-stop}
 e^{\lambda_0\tau_0}X_{\tau_0}
 =X_0+\sigma\sqrt{e^{2\lambda_0\tau_0}-1}\;Z
 \ \sim\ q^\eta,
 \qquad
 \eta^2=\sigma^2\bigl(e^{2\lambda_0\tau_0}-1\bigr).
\end{equation}
Total variation is invariant under the bijection $x\mapsto e^{\lambda_0\tau_0}x$,
so judging a sampler against $q^\eta$ is the same as judging its rescaled
output against the law at which an early-stopped diffusion sampler stops.
With direction-dependent rates, the effective smoothing level differs across
directions; our target uses a single level $\eta$ in all directions.

\subsection{Proof of the error decomposition}
\label{app:decomposition}

Let $\mu,\mu'$ be probability laws on $\R^K$ and $G=\N(0,\eta^2I_K)$. The
convolution $\mu*G$ is the push-forward of $\mu\otimes G$ under
$(u,\xi)\mapsto u+\xi$, so Lemma~\ref{lem:tv-dp} and the product identity
\eqref{eq:app-product-tv} give
\[
 \TV(\mu*G,\mu'*G)\le\TV(\mu\otimes G,\mu'\otimes G)=\TV(\mu,\mu').
\]
Applying \eqref{eq:app-product-tv} once more, with the common factor
$\N(0,\eta^2I_{d-K})$, and then the case $\mu=p$, $\mu'=q_K$,
\[
 \TV\bigl(\widetilde p^\eta,\widetilde q_K^{\,\eta}\bigr)
 =\TV\bigl(p*G,\ q_K*G\bigr)
 \le\TV(p,q_K).
\]
The triangle inequality
$\TV(\widetilde p^\eta,q^\eta)\le\TV(\widetilde p^\eta,\widetilde q_K^{\,\eta})
+\TV(\widetilde q_K^{\,\eta},q^\eta)$ gives \eqref{eq:error-decomposition}.
\hfill$\square$

\subsection{Wiener gains and the score}
\label{app:wiener}

The Wiener gain of Section~\ref{sec:setup-diffusion} also measures, for
arbitrary data, how far the score at noise level $s$ is from the score of a
Gaussian reference. We state it at the evaluation level $\eta$.

\begin{proposition}[Correction to the Gaussian reference score]
\label{prop:learnable}
Let $X\in\R^{d}$ satisfy $\E X=0$ and $\E\|X\|^2<\infty$, and write
$\sigma_i^2:=\E X_i^2$ for its coordinate variances in a fixed orthonormal
basis. For a noise level $\eta>0$ let $Y=X+\eta Z$ with
$Z\sim\N(0,I_d)$ independent of $X$, and let $p_\eta:=\law(Y)$ (the smoothed
law $q^\eta$ of \eqref{eq:smoothed-target} when $X\sim q$). Then, for every
coordinate $i$,
\[
 \partial_i\log p_\eta(y)
 =\underbrace{-\frac{y_i}{\eta^2}}_{\text{Gaussian reference}}
 +\underbrace{\frac{m_i(y)}{\eta^2}}_{\text{data-dependent correction}},
 \qquad
 m_i(y):=\E[X_i\mid Y=y],
\]
and, writing $g_i(\eta):=\sigma_i^2/(\sigma_i^2+\eta^2)$ for the Wiener gain,
\begin{equation}\label{eq:learnable-fraction}
 g_i(\eta)^2
 \;\le\;
 \frac{\E\,m_i(Y)^2}{\E\,Y_i^2}
 \;\le\;
 g_i(\eta).
\end{equation}
The lower bound is attained when $X$ is Gaussian with covariance diagonal in
the chosen basis.
\end{proposition}

\begin{proof}
The decomposition is Tweedie's formula
$\nabla\log p_\eta(y)=(\E[X\mid Y=y]-y)/\eta^2$, and
$\E Y_i^2=\sigma_i^2+\eta^2$ because $X$ and $Z$ are independent and centered.
Upper bound: by Jensen's inequality for conditional expectations,
$\E\,m_i(Y)^2\le\E X_i^2=\sigma_i^2$, so the ratio is at most
$\sigma_i^2/(\sigma_i^2+\eta^2)=g_i$.
Lower bound: by orthogonality of the conditional mean,
$\E\,m_i(Y)^2=\sigma_i^2-\E\bigl(X_i-m_i(Y)\bigr)^2$, and since the
conditional mean minimises the mean-squared error over all measurable
functions of $Y$,
\[
 \E\bigl(X_i-m_i(Y)\bigr)^2
 \le\E\bigl(X_i-g_iY_i\bigr)^2
 =\sigma_i^2-2g_i\sigma_i^2+g_i^2(\sigma_i^2+\eta^2)
 =\sigma_i^2(1-g_i).
\]
Hence $\E\,m_i(Y)^2\ge g_i\sigma_i^2$, and the ratio is at least
$g_i\sigma_i^2/(\sigma_i^2+\eta^2)=g_i^2$. If $X$ is Gaussian with diagonal
covariance, then $m_i(y)=g_iy_i$ and $\E\,m_i(Y)^2=g_i^2(\sigma_i^2+\eta^2)$,
which is the lower bound.
\end{proof}

Since $\E\,Y_i^2/\eta^4$ is the energy of the Gaussian reference score
$-y_i/\eta^2$, the proposition compares the correction with the
\emph{reference} score, not with the actual score: for $X\sim\N(0,1)$ and
$\eta=1$ the correction $y/2$ and the score $-y/2$ have equal energy although
$g=1/2$. The statement is that, at noise level $\eta$, a direction with
$\sigma_i^2\ll\eta^2$ differs from the known reference only by a correction
of relative energy at most $g_i\approx\sigma_i^2/\eta^2$; it says nothing
about a particular learning procedure. Non-Gaussian data carry a larger
correction than the Gaussian lower bound. The Wiener-gain rates
\eqref{eq:wiener-rates} scale the diffusion rate of each direction by this
upper bound, taken at the diffusion level $\sigma$. How the corrections of
many discarded directions combine is the subject of Section~\ref{sec:sharp}.

\section{Proofs for Section~\ref{sec:sharp}}
\label{app:proofs-sharp}

Throughout, $q=\N(0,C)$, $x_i=\sigma_i^2/\eta^2$ and $g_i=x_i/(1+x_i)$. In
the eigenbasis of $C$, both $q^\eta$ and $\widetilde q_K^{\,\eta}$ are centered
Gaussians with diagonal covariance: the entries are $\sigma_i^2+\eta^2$ for
$q^\eta$, and $\sigma_i^2+\eta^2$ for $i\le K$ and $\eta^2$ for $i>K$ for
$\widetilde q_K^{\,\eta}$. The two laws therefore agree on the first $K$
coordinates and are products over coordinates.

\subsection{Exact divergence and self-contained bounds}

The following lemma records the exact Kullback-Leibler cost of truncation and
bounds that do not rely on \citet{devroye2018total}.

\begin{lemma}[Exact divergence of truncation]
\label{lem:sharp-residual}
\label{prop:residual-kl-sharpness}
For every $1\le K<d$ and $\eta>0$,
\begin{equation}\label{eq:sharp-kl}
 \KL\bigl(\widetilde q_K^{\,\eta}\,\|\,q^\eta\bigr)
 =\frac12\sum_{i>K}\Bigl[\log(1+x_i)-\frac{x_i}{1+x_i}\Bigr]
 =\frac12\sum_{i>K}\bigl[-g_i-\log(1-g_i)\bigr]
 \le\frac14\sum_{i>K}\frac{x_i^2}{1+x_i}.
\end{equation}
If $x_i\le1$ for all $i>K$, then also
$\KL(\widetilde q_K^{\,\eta}\|q^\eta)\ge\frac1{16}\sum_{i>K}x_i^2$.
Moreover, with $S:=\sum_{i>K}\min\{x_i^2,1\}$,
\begin{equation}\label{eq:tv-lower}
 \frac12\min\Bigl\{1,\ \frac{S}{35}\Bigr\}
 \;\le\;
 \TV\bigl(\widetilde q_K^{\,\eta},q^\eta\bigr)
 \;\le\;
 \Bigl(\frac18\sum_{i>K}\frac{x_i^2}{1+x_i}\Bigr)^{1/2}.
\end{equation}
\end{lemma}

\begin{proof}
\emph{Identity.} Since the two laws agree on the first $K$ coordinates and
are products, the divergence is the sum over $i>K$ of the divergences between
$\N(0,\eta^2)$ and $\N(0,\eta^2(1+x_i))$. The one-dimensional formula
$\KL(\N(0,a)\|\N(0,b))=\frac12(a/b-1-\log(a/b))$ with $a/b=1/(1+x_i)=1-g_i$
gives $\frac12[\log(1+x_i)-x_i/(1+x_i)]=\frac12[-g_i-\log(1-g_i)]$.

\emph{Upper bound.} For $g\in[0,1)$,
\[
 -g-\log(1-g)=\sum_{k\ge2}\frac{g^k}{k}
 \le\frac{g^2}{2}\sum_{j\ge0}g^j
 =\frac{g^2}{2(1-g)}
 =\frac{x^2}{2(1+x)} .
\]

\emph{Lower bound in divergence.} Let $\phi(x)=\log(1+x)-x/(1+x)$. On
$[0,1]$ the function $\phi(x)-x^2/8$ vanishes at $0$ and has derivative
$x\bigl((1+x)^{-2}-\frac14\bigr)\ge0$, so $\phi(x)\ge x^2/8$ there; the outer
factor $\frac12$ gives the claim.

\emph{Upper bound in total variation.} Pinsker's inequality
$\TV\le\sqrt{\KL/2}$ and \eqref{eq:sharp-kl}.

\emph{Lower bound in total variation.} Let
$\mathrm{BC}(\mu,\nu)=\int\sqrt{\mathrm d\mu\,\mathrm d\nu}$ be the
Bhattacharyya coefficient. It is multiplicative over independent
coordinates, and $1-\mathrm{BC}(\mu,\nu)\le\TV(\mu,\nu)$ (Le Cam's inequality
$H^2\le\TV$ with $H^2=1-\mathrm{BC}$). For centered one-dimensional Gaussians
with variances $a,b$, $\mathrm{BC}=\bigl(2\sqrt{ab}/(a+b)\bigr)^{1/2}$; with
$a=\eta^2$ and $b=\eta^2(1+x_i)$ this equals $1-\psi(x_i)$, where
\[
 \psi(x):=1-\Bigl(\frac{2\sqrt{1+x}}{2+x}\Bigr)^{1/2}.
\]
The function $\psi$ is increasing on $[0,\infty)$, because the ratio of the
geometric to the arithmetic mean of $1$ and $1+x$ decreases in $x$. The ratio
$\psi(x)/x^2$ decreases on $(0,1]$ from its limit $1/16$ at $0$ to
$\psi(1)\approx0.0290>1/35$. Hence $\psi(x)\ge\min\{x^2,1\}/35$ for all
$x\ge0$, and
\[
 \TV\bigl(\widetilde q_K^{\,\eta},q^\eta\bigr)
 \ge1-\prod_{i>K}\bigl(1-\psi(x_i)\bigr)
 \ge1-\exp\Bigl(-\sum_{i>K}\psi(x_i)\Bigr)
 \ge1-e^{-S/35}
 \ge\frac12\min\Bigl\{1,\frac{S}{35}\Bigr\},
\]
using $1-e^{-u}\ge\frac12\min\{u,1\}$ for $u\ge0$.
\end{proof}

\subsection{Proof of Theorem~\ref{thm:truncation-cost}}

We identify the covariance ratios in the Gaussian comparison with Wiener gains, then project onto the discarded coordinates to obtain a bound independent of the active sampler.

\emph{Part (a).} Let $\Sigma_1=\mathrm{Cov}(q^\eta)$ and
$\Sigma_2=\mathrm{Cov}(\widetilde q_K^{\,\eta})$; both are positive definite.
In the eigenbasis of $C$, the matrix $\Sigma_1^{-1}\Sigma_2-I$ is diagonal
with entries $0$ for $i\le K$ and $\eta^2/(\sigma_i^2+\eta^2)-1=-g_i$ for
$i>K$. Theorem~1.1 of \citet{devroye2018total} states that, for centered
Gaussians with positive definite covariances,
\[
 \frac1{100}
 \le
 \frac{\TV\bigl(\N(0,\Sigma_1),\N(0,\Sigma_2)\bigr)}
  {\min\bigl\{1,\bigl(\sum_i\lambda_i^2\bigr)^{1/2}\bigr\}}
 \le\frac32,
\]
where $\lambda_i$ are the eigenvalues of $\Sigma_1^{-1}\Sigma_2-I$. With
$\sum_i\lambda_i^2=\sum_{i>K}g_i^2$ this is the two-sided bound. The bound
with explicit constants is the upper bound in \eqref{eq:tv-lower}.

\emph{Part (b).} Let $\Pi_\perp$ map a point to its coordinates $i>K$ in the
eigenbasis of $C$. Under every truncated sampler these coordinates have law
$\N(0,\eta^2I_{d-K})$, whatever $p$ is, and under $q^\eta$ they have law
$\N\bigl(0,\diag(\sigma_i^2+\eta^2)_{i>K}\bigr)$. By data processing
(Lemma~\ref{lem:tv-dp}),
\[
 \TV\bigl(\widetilde p^\eta,q^\eta\bigr)
 \ge\TV\Bigl(\N(0,\eta^2I_{d-K}),\ \N\bigl(0,\diag(\sigma_i^2+\eta^2)_{i>K}\bigr)\Bigr).
\]
Since $\widetilde q_K^{\,\eta}$ and $q^\eta$ are products with the same first
factor, the product identity \eqref{eq:app-product-tv} shows that the
right-hand side equals $\TV(\widetilde q_K^{\,\eta},q^\eta)$.
\hfill$\square$

\subsection{Proof of Theorem~\ref{thm:num-directions}}

The upper bound follows from a tail integral; the lower bound uses a block of discarded directions of comparable variance. Together they determine the required order of $K$.

Under the exact power law, $x_i=(C_{\mathrm{sp}}/\eta^2)\,i^{-\beta}$.

\emph{Part (a).} By \eqref{eq:tv-lower} and $x_i^2/(1+x_i)\le x_i^2$,
\[
 \TV\bigl(\widetilde q_K^{\,\eta},q^\eta\bigr)^2
 \le\frac18\sum_{i>K}x_i^2
 \le\frac{C_{\mathrm{sp}}^2}{8\eta^4}\int_K^\infty t^{-2\beta}\,\mathrm dt
 =\frac{C_{\mathrm{sp}}^2}{8\eta^4}\,\frac{K^{1-2\beta}}{2\beta-1},
\]
which is \eqref{eq:sharp-tv}. The right-hand side of \eqref{eq:sharp-tv} is
at most $\varepsilon/2$ if and only if
\[
 K^{(2\beta-1)/2}
 \ge\frac{2C_{\mathrm{sp}}}{\eta^2\varepsilon\sqrt{8(2\beta-1)}}
 =\frac{C_{\mathrm{sp}}}{\eta^2\varepsilon\sqrt{2(2\beta-1)}},
\]
which holds for $K\ge K_\varepsilon(\eta)$.

\emph{Part (b).} Suppose $\TV(\widetilde p^\eta,q^\eta)\le\varepsilon<1/200$.
If $K<k_\eta$, then $K+1\le k_\eta\le(C_{\mathrm{sp}}/\eta^2)^{1/\beta}$, so
$x_{K+1}\ge1$ and $g_{K+1}\ge1/2$; Theorem~\ref{thm:truncation-cost} then
gives $\TV(\widetilde p^\eta,q^\eta)\ge1/200$, a contradiction. Hence
$K\ge k_\eta$, so every discarded direction has $x_i<1$ and therefore
$g_i\ge x_i/2$. If $\|(g_i)_{i>K}\|_2\ge1$, Theorem~\ref{thm:truncation-cost}
gives $\TV\ge1/100$, again a contradiction; hence
$\|(g_i)_{i>K}\|_2\le100\,\varepsilon$ and
$\sum_{i>K}x_i^2\le4\sum_{i>K}g_i^2\le4\cdot10^4\varepsilon^2$. On the other
hand, since $2K\le d$,
\[
 \sum_{i>K}x_i^2
 \ge\sum_{i=K+1}^{2K}\frac{C_{\mathrm{sp}}^2}{\eta^4}\,i^{-2\beta}
 \ge K\,\frac{C_{\mathrm{sp}}^2}{\eta^4}\,(2K)^{-2\beta}.
\]
Combining the two displays gives
$K^{2\beta-1}\ge C_{\mathrm{sp}}^2/(2^{2\beta}\cdot4\cdot10^4\,\eta^4\varepsilon^2)$,
which is the stated bound.

\emph{Part (c).} Since $K\ge k_\eta$, we have
$K+1>(C_{\mathrm{sp}}/\eta^2)^{1/\beta}$ and hence $x_{K+1}<1$; all discarded
directions then satisfy $x_i<1$ and $x_i/2\le g_i\le x_i$. For $i>K$,
$x_i=x_{K+1}\bigl((K+1)/i\bigr)^{\beta}$. Using the $K$ terms with
$K<i\le2K$, each at least $x_{K+1}2^{-\beta}$, and an integral comparison for
the tail,
\[
 2^{-2\beta}K\,x_{K+1}^2
 \le\sum_{i>K}x_i^2
 \le x_{K+1}^2\Bigl(1+\frac{K+1}{2\beta-1}\Bigr)
 \le\frac{2\beta+1}{2\beta-1}\,K\,x_{K+1}^2,
\]
where the last step uses $K\ge1$. Hence
$2^{-\beta-1}x_{K+1}\sqrt K\le\|(g_i)_{i>K}\|_2
\le\sqrt{(2\beta+1)/(2\beta-1)}\,x_{K+1}\sqrt K$, and
Theorem~\ref{thm:truncation-cost}(a) gives the claim.
\hfill$\square$

\subsection{Proof of Corollary~\ref{cor:noise-floor}}

At the noise-level cut-off, the next block of coordinates still has gains bounded away from zero. Summing their contributions gives the claimed obstruction.

With $K=k_\eta$, every discarded direction has $x_i<1$, and for
$k_\eta<i\le2k_\eta$,
\[
 x_i\ge\frac{C_{\mathrm{sp}}}{\eta^2}(2k_\eta)^{-\beta}
 \ge2^{-\beta},
\]
because $C_{\mathrm{sp}}k_\eta^{-\beta}\ge\eta^2$. Hence
$g_i\ge x_i/2\ge2^{-\beta-1}$ for these $k_\eta$ directions, so
$\|(g_i)_{i>k_\eta}\|_2\ge2^{-\beta-1}\sqrt{k_\eta}$, and
Theorem~\ref{thm:truncation-cost} gives the bound for every truncated sampler.
The final claim follows since $2^{-\beta-1}\sqrt{k_\eta}\ge1$ exactly when
$k_\eta\ge4^{\beta+1}$.
\hfill$\square$

\section{General Results and Proofs for Section~\ref{sec:convergence}}
\label{app:proofs-convergence}

\subsection{Convergence for arbitrary diagonal rates}
\label{app:general-rates}

We work with a generic coefficient dimension $r$ and the diagonal
Ornstein-Uhlenbeck process
\begin{equation}\label{eq:generic-forward-OU}
 \mathrm dU_t=-\Lambda U_t\,\mathrm dt+\sigma\sqrt{2\Lambda}\,\mathrm dB_t,
 \qquad U_0\sim q_r,
 \qquad \Lambda=\diag(\lambda_1,\ldots,\lambda_r),\ \lambda_i>0,
\end{equation}
on $\R^r$, with stationary law $\gamma_r=\N(0,\sigma^2I_r)$ and marginals
$q_{r,t}$. The case $r=K$, $q_r=q_K$ is the retained-coordinate process of
Section~\ref{sec:setup-diffusion}; the case $r=d$ is a full-dimensional
diffusion. Assumptions~\ref{ass:lipschitz}-\ref{ass:finite-kl} are
understood with $K$ replaced by $r$.

\emph{Tail and regularity conditions in Assumption~\ref{ass:lipschitz}.}
For every $t$ and every coordinate $i$, the integration-by-parts boundary term
vanishes,
\[
 \lim_{R\to\infty}
 \int_{\partial[-R,R]^r}q_{r,t}(u)\,\partial_i\log q_{r,t}(u)\,n_i(u)\,\mathrm dS(u)=0,
\]
where $n_i$ is the $i$-th component of the outward unit normal; this holds
under sub-Gaussian or polynomial tails and gives the Fisher bound
$\E\|\nabla\log q_{r,t}\|_\Lambda^2\le L\tr(\Lambda)$
(Lemma~\ref{lem:fisher}). In addition, the differentiation under the
integral, integration by parts and conditional Brascamp-Lieb steps in the
proofs of the one-step perturbation lemmas (Lemmas~\ref{lem:score-perturb}
and~\ref{lem:trace_aware_weighted}) are valid for the marginals $q_{r,t}$;
this holds under standard smoothness and tail conditions and in particular for
Gaussian data.

\emph{Early stopping.} If the data law has no smooth density, all results
below hold on $[0,T-\tau_0]$ for any $\tau_0>0$ after replacing $q_r$ by
$q_{r,\tau_0}$, $T$ by $T-\tau_0$ and $\KL(q_r\|\gamma_r)$ by
$\KL(q_{r,\tau_0}\|\gamma_r)$, with the assumptions required only for
$t\ge\tau_0$. The sampler is then compared with $q_{r,\tau_0}$ rather than
$q_r$.

\emph{Admissibility.} Assumption~\ref{ass:girsanov} requires that, for the
exact reverse drift $b$ and the sampler drift $\widetilde b$, with
$G=\sigma\sqrt{2\Lambda}$, the stochastic exponential of
$G^{-1}\bigl(b(t,U_t)-\widetilde b(t,U_t)\bigr)$ is a true martingale on
$[0,T]$ under the law $Q_T^\leftarrow$ of the exact reverse process.
Two sufficient conditions are standard. (i) Novikov's condition,
\[
\mathbb E_{Q_T^\leftarrow}
\exp\left(
\frac12\int_0^T
\left\|
G^{-1}\bigl(b(t,U_t)-\widetilde b(t,U_t)\bigr)
\right\|^2\,dt
\right)
<\infty .
\]
(ii) Linear growth of the learned scores,
$\|\widehat s_{t_j}(u)\|\le A(1+\|u\|)$ for all $j$ and $u$: together with the
Lipschitz exact score of Assumption~\ref{ass:lipschitz}, both drifts then grow
at most linearly in $\sup_{s\le t}\|U_s\|$, and Bene\v{s}' criterion (see
e.g.\ \citealp[Section~3.5]{karatzas1991brownian}) makes the stochastic
exponential a true martingale. Condition (ii) can be enforced by score
clipping, and it holds automatically for exact Gaussian scores, which are
linear. Assumptions~\ref{ass:lipschitz}
and~\ref{ass:second_moment} alone do not imply
Assumption~\ref{ass:girsanov}, since they do not constrain the learned
score.

\begin{assumption}[Finite second moment]
\label{ass:second_moment}
$m_{2,r}^2:=\E_{U\sim q_r}\|U\|^2<\infty$. Under
Assumption~\ref{ass:powerlaw} with $r=K$, $m_{2,K}^2\le C_{\mathrm{sp}}\zeta(\beta)$.
\end{assumption}

The following two theorems are proved in Appendix~\ref{sec:app:C}.

\begin{theorem}[TV convergence, dimension-dependent warm-up]
\label{thm:tv-dim}
Assume Assumptions~\ref{ass:lipschitz}, \ref{ass:second_moment}, \ref{ass:score_error},
\ref{ass:girsanov} and~\ref{ass:finite-kl} in $\R^r$, with the unweighted Fisher bound $\mathbb
E_{U\sim q_{r,t}}\|\nabla\log q_{r,t}(U)\|^2\le C_F^{(0)}r$ from
Lemma~\ref{lem:fisher} ($C_F^{(0)}=L$). We additionally assume that
the one-step score-perturbation bound of Lemma~\ref{lem:score-perturb}
applies uniformly to the coefficient-space marginals $q_{r,t}$,
$t\in[0,T]$ - that is, the differentiation-under-the-integral,
integration-by-parts, and conditional Brascamp-Lieb steps used in its
proof are valid for these marginals. This holds under standard
smoothness and tail-decay conditions and is the only regularity used
beyond Assumption~\ref{ass:lipschitz}. If $\delta\le\min\{\tfrac18,
\tfrac{1}{128L\sigma^2}\}$
then, writing
\[
 \varepsilon_{\mathrm{score}}^2 := \frac1N\sum_{j=1}^{N}
 \mathbb E_{U\sim q_{r,t_j}}\left\|\widehat s_{t_j}(U)-\nabla\log q_{r,t_j}(U)\right\|^2,
\]
there is a universal constant $\widetilde C>0$ with
\begin{equation*}
\TV(p_{\mathrm{out}},q_r) \le e^{-\lambda_{\min}T}\sqrt{\tfrac12\KL(q_r\|\gamma_r)}
+ \widetilde C\sigma\sqrt{\lambda_{\max}T}\Big(\varepsilon_{\mathrm{score}}
+ Lm_{2,r}\delta + (1+L\sigma^2)\sqrt{C_F^{(0)}r}\,\delta + L\sigma\sqrt{r\delta}\Big)
\end{equation*}
where $\lambda_{\min},\lambda_{\max}$ are the extreme rates of
$\Lambda$ and $m_{2,r}^2:=\mathbb E_{U\sim q_r}\|U\|^2$.
\end{theorem}

\begin{theorem}[Trace-dependent TV convergence]
\label{thm:tv-trace}
Assume Assumptions~\ref{ass:lipschitz}, \ref{ass:second_moment}, \ref{ass:score_error},
\ref{ass:girsanov} and~\ref{ass:finite-kl} in $\R^r$, with the weighted Fisher bound $\mathbb
E_{U\sim q_{r,t}}\|\nabla\log q_{r,t}(U)\|_\Lambda^2\le C_F\tr(\Lambda)$
from Lemma~\ref{lem:fisher} ($C_F=L$). As in
Theorem~\ref{thm:tv-dim}, we additionally assume that the weighted
one-step score-perturbation bound of
Lemma~\ref{lem:trace_aware_weighted} applies uniformly to the
marginals $q_{r,t}$, $t\in[0,T]$; this is the only regularity used
beyond Assumption~\ref{ass:lipschitz}. Let
$\kappa:=\lambda_{\max}/\lambda_{\min}$; suppose $\lambda_{\max}
\Delta\tau_{\max}\le1$ and $\delta\le\min\{\tfrac18,
\tfrac{1}{128\sqrt\kappa\,L\sigma^2}\}$. Then there is a constant $C(\kappa)=O(\sqrt\kappa)$
with
\begin{align*}
&\TV(p_{\mathrm{out}},q_r) \le e^{-\lambda_{\min}T}\sqrt{\tfrac12\KL(q_r\|\gamma_r)}
\nonumber\\
&\quad+ C(\kappa)\sigma\sqrt T\Big(\varepsilon_{\mathrm{score},\Lambda}
+ L\delta\sqrt{m_{2,\Lambda}^2+\sigma^2\tr(\Lambda)}
+ (1+L\sigma^2)\sqrt{C_F}\,\delta\sqrt{\tr(\Lambda)}
+ L\sigma\sqrt{\delta\tr(\Lambda)}\Big),
\end{align*}
where $m_{2,\Lambda}^2:=\mathbb E_{U\sim q_r}\|U\|_\Lambda^2$ and
$\varepsilon_{\mathrm{score},\Lambda}^2$ is defined as in
Assumption~\ref{ass:score_error}.
\end{theorem}

Theorem~\ref{thm:tv-dim} with $r=K$ already gives
$N=\widetilde O(K/\varepsilon^2)$ on the retained coordinates.
Theorem~\ref{thm:tv-trace} replaces $r$ by the rate-weighted quantities
$\tr(\Lambda)$ and $m_{2,\Lambda}^2$ and reduces to the same dependence for
equal rates. Specialised to the retained coordinates, it gives the
following bound, which holds for any $K$ and any rates.

\begin{corollary}[Top-$K$ TV bound]
\label{cor:topK}
Let \(V_K\in\R^{d\times K}\) be any matrix with orthonormal columns, and
let \(\lambda_1\ge\cdots\ge\lambda_K>0\) be any rates. Let
\(\kappa:=\lambda_{\max}/\lambda_{\min}=\lambda_1/\lambda_K\) denote the
condition number. Under Assumptions~\ref{ass:lipschitz}, \ref{ass:second_moment}, \ref{ass:score_error},
\ref{ass:girsanov} and~\ref{ass:finite-kl}, with the weighted Fisher bound
$\mathbb E\|\nabla\log q_{K,t}(U)\|_\Lambda^2\le C_F\tr(\Lambda)$
furnished by Lemma~\ref{lem:fisher} ($C_F=L$), under the one-step
score-perturbation hypothesis of Theorem~\ref{thm:tv-trace}, and under the small-step
condition
\begin{equation}
\delta\le\min\!\Bigl(
 \tfrac{1}{8},\,
 \tfrac{1}{128\sqrt{\kappa}\,L\sigma^2}
\Bigr),
\label{eq:small-step-cor1}
\end{equation}
there exists a constant \(C_0(\kappa)>0\), depending only on \(\kappa\) and
universal constants, such that
\begin{align}
\TV(p,q_K)
&\le
e^{-\lambda_K T}
\sqrt{\frac12\,\KL(q_K\|\gamma_K)}
\nonumber\\
&\quad+C_0(\kappa)\,\sigma\sqrt{T}\Bigl(
 \varepsilon_{\mathrm{score},\bLambda}
 +L\delta\sqrt{\mathfrak{m}_{2,\bLambda}^2
 +\sigma^2\tr(\bLambda)}
 \nonumber\\
&\qquad\qquad
 +(1+L\sigma^2)\sqrt{C_F}\,\delta\sqrt{\tr(\bLambda)}
 +L\sigma\sqrt{\delta\,\tr(\bLambda)}
\Bigr),
\label{eq:tv-cor1}
\end{align}
where \(\lambda_K=\min_i\lambda_i\) is the smallest active rate,
\[
\mathfrak{m}_{2,\bLambda}^2
=\sum_{i=1}^K\lambda_i\sigma_i^2,
\qquad
\tr(\bLambda)=\sum_{i=1}^K\lambda_i,
\qquad
\tr(\bLambda^2)=\sum_{i=1}^K\lambda_i^2.
\]
For non-PCA bases, the expression for \(\mathfrak m_{2,\bLambda}^2\) should be
read as \(\mathbb E\|U_0\|_\Lambda^2\). The bound depends on \(K\) and the
rates \(\{\lambda_i\}_{i=1}^K\) only through these weighted trace quantities,
and is independent of the ambient dimension \(d\).
\end{corollary}

\begin{proof}[Proof of Corollary~\ref{cor:topK}]
Apply Theorem~\ref{thm:tv-trace} with $r=K$, $q_r=q_K$ and
$\lambda_{\min}=\lambda_K$.
\end{proof}

\begin{proposition}[Isotropic rates on the same retained subspace]
\label{prop:isotropic}
Let $1\le K\le d_{\mathrm{eff}}^{\sigma}$ and replace the Wiener-gain rates by the
constant rates $\lambda_i\equiv\lambda_0$, $i\le K$. Assume the hypotheses
of Corollary~\ref{cor:topK} for the resulting isotropic process, with
Lipschitz constant $L_{\mathrm{iso}}$ in Assumption~\ref{ass:lipschitz} and
score error $\varepsilon_{\mathrm{score},\lambda_0I}$ in
Assumption~\ref{ass:score_error}. Then the bound \eqref{eq:tv-cor2} of
Theorem~\ref{thm:convergence} holds for the isotropic sampler with $L$
replaced by $L_{\mathrm{iso}}$, $\kappa_0$ replaced by $1$, and the mixing
factor $e^{-\lambda_\star T}$ replaced by $e^{-\lambda_0T}$.
Consequently the step count $N=\widetilde O(K/\varepsilon^2)$ of
Theorem~\ref{thm:convergence} also holds for isotropic diffusion on the same
subspace. Moreover, the weighted norms of the
two schemes are equivalent on the retained subspace:
$\lambda_\star\|z\|^2\le\|z\|_{\Lambda}^2\le\lambda_0\|z\|^2$ for the
Wiener-gain matrix $\Lambda$; for a fixed error vector and a fixed
evaluation distribution the two error norms therefore differ by at most the
factor $\kappa_0$. This does not compare how hard the two score
functions are to learn: the two processes have different forward marginals and
different target scores. For Gaussian data with the exact
power law, $L_{\mathrm{iso}}=L=\sigma^{-2}$.
\end{proposition}

\begin{proof}
With $\lambda_i\equiv\lambda_0$ we have $\kappa=1$, $\lambda_K=\lambda_0$,
$\tr(\Lambda)=\lambda_0K$, and
$m_{2,\Lambda}^2=\lambda_0\sum_{i\le K}\sigma_i^2\le\lambda_0C_{\mathrm{sp}}\zeta(\beta)$
by the upper power-law envelope. These are the bounds of
Lemma~\ref{lem:traces}, so substituting them into Corollary~\ref{cor:topK}
gives \eqref{eq:tv-cor2} with the stated replacements, and the step-count
calculation in the proof of Theorem~\ref{thm:convergence} is unchanged. The norm
equivalence follows from $\lambda_\star\le\lambda_i\le\lambda_0$ for the
Wiener-gain rates on $i\le K\le d_{\mathrm{eff}}^{\sigma}$. For Gaussian data the
isotropic marginal covariance has diagonal entries
$e^{-2\lambda_0t}\sigma_i^2+\sigma^2(1-e^{-2\lambda_0t})\ge\sigma^2$, since
$\sigma_i^2\ge\sigma^2$ for $i\le d_{\mathrm{eff}}^{\sigma}$, so the score is
$\sigma^{-2}$-Lipschitz as in the Wiener-gain case.
\end{proof}

\subsection{Proof of Theorem~\ref{thm:convergence}}

\begin{lemma}[Trace bounds under Wiener-gain rates and power-law decay]
\label{lem:traces}
Under the upper side of Assumption~\ref{ass:powerlaw}, namely
$\sigma_i^2\le C_{\mathrm{sp}}i^{-\beta}$ with $\beta>1$, use the Wiener-gain
rates~\eqref{eq:wiener-rates} and let
$1\le K\le d_{\mathrm{eff}}^{\sigma}$. Then
\begin{align}
 \tr(\Lambda)
 &\le \lambda_0K,
 \label{eq:trL}\\[3pt]
 \tr(\Lambda^2)
 &\le \lambda_0^2K,
 \label{eq:trL2}\\[3pt]
 m_{2,\Lambda}^2
 &\le \lambda_0C_{\mathrm{sp}}\zeta(\beta),
 \label{eq:m2L}
\end{align}
where $\zeta(\beta)=\sum_{i=1}^{\infty}i^{-\beta}<\infty$. In particular,
the bounds do not depend on the ambient dimension $d$.
\end{lemma}

\begin{proof}
Throughout, write
\[
 g_i := \frac{\sigma_i^2}{\sigma_i^2+\sigma^2}\in[0,1)
\]
for the Wiener gain.

\medskip\noindent\textbf{Bounding \(\tr(\Lambda)\).}

Since \(\lambda_i=\lambda_0g_i\) and \(0\le g_i\le1\),
\[
 \tr(\Lambda)
 =
 \lambda_0\sum_{i=1}^K g_i
 \le
 \lambda_0 K.
\]

\medskip\noindent\textbf{Bounding \(\tr(\Lambda^2)\).}

Since \(0\le g_i\le1\), we have \(g_i^2\le g_i\le1\). Hence
\[
 \tr(\Lambda^2)
 =
 \lambda_0^2\sum_{i=1}^K g_i^2
 \le
 \lambda_0^2 K.
\]

\medskip\noindent\textbf{Bounding \(m_{2,\Lambda}^2\).}

By definition,
\[
 m_{2,\Lambda}^2
 =
 \sum_{i=1}^K \lambda_i\sigma_i^2
 =
 \lambda_0\sum_{i=1}^K g_i\sigma_i^2.
\]
Since \(g_i\le1\),
\[
 m_{2,\Lambda}^2
 \le
 \lambda_0\sum_{i=1}^K\sigma_i^2.
\]
Using the upper power-law bound \(\sigma_i^2\le C_{\mathrm{sp}} i^{-\beta}\),
\[
 m_{2,\Lambda}^2
 \le
 \lambda_0 C_{\mathrm{sp}}\sum_{i=1}^{K}i^{-\beta}
 \le
 \lambda_0 C_{\mathrm{sp}}\sum_{i=1}^{\infty}i^{-\beta}
 =
 \lambda_0 C_{\mathrm{sp}}\zeta(\beta).
\]
Since \(\beta>1\), the zeta sum is finite. Therefore
\[
 m_{2,\Lambda}^2
 \le
 \lambda_0 C_{\mathrm{sp}}\zeta(\beta),
\]
independently of \(K\) and \(d\).
\end{proof}

\begin{proof}[Proof of Theorem~\ref{thm:convergence}]
\emph{Rates.} Since $K\le d_{\mathrm{eff}}^{\sigma}\le(C_{\mathrm{sp}}/\sigma^2)^{1/\beta}$,
the lower envelope gives, for every $i\le K$,
$\sigma_i^2\ge c_{\mathrm{sp}}i^{-\beta}\ge c_{\mathrm{sp}}K^{-\beta}
\ge(c_{\mathrm{sp}}/C_{\mathrm{sp}})\sigma^2$, and therefore
\[
 \lambda_i\ge\lambda_0\,\frac{c_{\mathrm{sp}}}{c_{\mathrm{sp}}+C_{\mathrm{sp}}}
 =\lambda_\star,
 \qquad
 \kappa\le\frac{\lambda_0}{\lambda_\star}=\kappa_0 .
\]

\emph{Bound.} Apply Corollary~\ref{cor:topK}. Its mixing term is at most
$e^{-\lambda_\star T}\sqrt{\frac12\KL(q_K\|\gamma_K)}$ because
$\lambda_K\ge\lambda_\star$. By Lemma~\ref{lem:fisher}, $C_F=L$. Substituting
the trace bounds of Lemma~\ref{lem:traces} and absorbing $C_0(\kappa)\le
C_0(\kappa_0)$, which uses that $C_0$ is nondecreasing in $\kappa$ (one can
take $C_0(\kappa)=c\sqrt\kappa$), into a constant $C'$ depending only on
$\beta$ and $\kappa_0$ gives \eqref{eq:tv-cor2}.

\emph{Step count.} The choice of $T$ makes the mixing term at most
$\varepsilon/2$, and the assumption on $\varepsilon_{\mathrm{score},\Lambda}$
makes the score term at most a small constant times $\varepsilon$. On a
uniform grid, $\delta\le\lambda_{\max}T/N\le\lambda_0T/N$. The largest of the
three discretisation terms for small $\delta$ is
$C'\sigma\sqrt T\,L\sigma\sqrt{\lambda_0\delta K}$, which is at most a small
constant times $\varepsilon$ once
$\delta\lesssim\varepsilon^2/(\sigma^4TL^2\lambda_0K)$, i.e.\ once
\[
 N\gtrsim\frac{(L\sigma^2)^2\lambda_0^2T^2K}{\varepsilon^2}.
\]
The two remaining discretisation terms are linear in $\delta$ and are smaller
under the same choice for small $\varepsilon$, and the small-step condition
\eqref{eq:small-step-cor2} holds once
$N\ge\lambda_0T\max\{8,\,128\sqrt{\kappa_0}\,L\sigma^2\}$.
Since $T$ is logarithmic in $1/\varepsilon$, this is
$N=\widetilde O(K/\varepsilon^2)$.
\end{proof}

\subsection{Proof of Theorem~\ref{thm:end-to-end}}

We first keep the dependence on $\sigma$ explicit in the active sampling bound. This is necessary because the theorem chooses $\sigma$ from $K$, after which we combine sampling and truncation errors.

\begin{lemma}[$\sigma$-explicit active error, Gaussian case]
\label{lem:active-sigma}
Let $q=\N(0,C)$ with the exact power law \eqref{eq:powerlaw2}, let
$1\le K\le d_{\mathrm{eff}}^{\sigma}$, use the Wiener-gain rates, exact scores and a uniform grid, and
assume $N\ge128\sqrt2\,\lambda_0T$. Then, with $C_0=C_0(2)$ the constant of
Corollary~\ref{cor:topK},
\[
 \TV(p,q_K)
 \le e^{-\lambda_0T/2}\frac{\sqrt{C_{\mathrm{sp}}\zeta(\beta)}}{2\sigma}
 +C_0\Bigl(\lambda_0T\sqrt{\tfrac KN}
 +3\lambda_0^{3/2}T^{3/2}\frac{\sqrt K}{N}
 +\lambda_0^{3/2}T^{3/2}\frac{\sqrt{C_{\mathrm{sp}}\zeta(\beta)}}{\sigma N}\Bigr).
\]
\end{lemma}

\begin{proof}
We first collect the constants. For $i\le K\le d_{\mathrm{eff}}^{\sigma}$ the
exact power law gives $\sigma_i^2\ge\sigma^2$, hence $g_i(\sigma)\ge1/2$,
$\lambda_i\in[\lambda_0/2,\lambda_0]$ and $\kappa\le2$. The marginal
$q_{K,t}$ is $\N(0,D_t)$ with diagonal entries
$e^{-2\lambda_it}\sigma_i^2+\sigma^2(1-e^{-2\lambda_it})\ge\sigma^2$, so its
score $-D_t^{-1}u$ is $\sigma^{-2}$-Lipschitz: $L=C_F=\sigma^{-2}$ and
$L\sigma^2=1$. Exact scores give $\varepsilon_{\mathrm{score},\Lambda}=0$;
they are linear, so Assumption~\ref{ass:girsanov} holds by the linear-growth
criterion of Appendix~\ref{app:general-rates}, and Gaussian marginals satisfy
the tail and regularity conditions. Lemma~\ref{lem:traces} gives
$\tr(\Lambda)\le\lambda_0K$ and
$m_{2,\Lambda}^2\le\lambda_0C_{\mathrm{sp}}\zeta(\beta)$. Finally, with
$\mu_i=\sigma_i^2/\sigma^2\ge1$,
$\KL(q_K\|\gamma_K)=\frac12\sum_{i\le K}(\mu_i-1-\log\mu_i)
\le\frac12\sum_{i\le K}\mu_i\le C_{\mathrm{sp}}\zeta(\beta)/(2\sigma^2)$. On a uniform
grid $\delta\le\lambda_{\max}T/N\le\lambda_0T/N\le1/(128\sqrt2)$, so the
small-step condition of Corollary~\ref{cor:topK} (with $\kappa\le2$ and
$L\sigma^2=1$) and $\lambda_{\max}\Delta\tau_{\max}\le1$ hold. Substituting
into \eqref{eq:tv-cor1}, the mixing term is at most
$e^{-\lambda_0T/2}\sqrt{C_{\mathrm{sp}}\zeta(\beta)}/(2\sigma)$, and the
three discretisation terms are bounded as follows:
\begin{align*}
 \sigma\sqrt T\,L\delta\sqrt{m_{2,\Lambda}^2+\sigma^2\tr(\Lambda)}
 &\le\lambda_0^{3/2}T^{3/2}\Bigl(\frac{\sqrt{C_{\mathrm{sp}}\zeta(\beta)}}
  {\sigma N}+\frac{\sqrt K}{N}\Bigr),\\
 \sigma\sqrt T\,(1+L\sigma^2)\sqrt{C_F}\,\delta\sqrt{\tr(\Lambda)}
 &\le2\lambda_0^{3/2}T^{3/2}\frac{\sqrt K}{N},\\
 \sigma\sqrt T\,L\sigma\sqrt{\delta\tr(\Lambda)}
 &\le\lambda_0T\sqrt{\frac KN},
\end{align*}
using $\sqrt{a+b}\le\sqrt a+\sqrt b$. Collecting terms gives the claim.
\end{proof}

\begin{proof}[Proof of Theorem~\ref{thm:end-to-end}]
By \eqref{eq:error-decomposition}, it suffices to make the truncation error
and the sampling error each at most $\varepsilon/2$.

\emph{Truncation error.} With $K=K_\varepsilon(\eta)<d$,
Theorem~\ref{thm:num-directions}(a) gives
$\TV(\widetilde q_K^{\,\eta},q^\eta)\le\varepsilon/2$.

\emph{Admissibility of $K$.} With $\sigma^2=C_{\mathrm{sp}}K^{-\beta}$ we have
$(C_{\mathrm{sp}}/\sigma^2)^{1/\beta}=K$, so $d_{\mathrm{eff}}^{\sigma}=K$
(using $K<d$) and Lemma~\ref{lem:active-sigma} applies.

\emph{Sampling error.} We require each of the four terms of
Lemma~\ref{lem:active-sigma} to be at most $\varepsilon/8$. The mixing term
is at most $\varepsilon/8$ by the choice of $T$, since
$e^{-\lambda_0T/2}=\sigma\varepsilon/(4\sqrt{C_{\mathrm{sp}}\zeta(\beta)})$. The
three discretisation terms are at most $\varepsilon/8$ once
$N\ge64C_0^2\lambda_0^2T^2K/\varepsilon^2$,
$N\ge24C_0\lambda_0^{3/2}T^{3/2}\sqrt K/\varepsilon$ and
$N\ge8C_0\lambda_0^{3/2}T^{3/2}\sqrt{C_{\mathrm{sp}}\zeta(\beta)}/(\sigma\varepsilon)$,
and the step-size condition of the lemma holds once
$N\ge128\sqrt2\,\lambda_0T$. Because $\sigma^2\le C_{\mathrm{sp}}\le
C_{\mathrm{sp}}\zeta(\beta)$ and $\varepsilon<1$, we have $\lambda_0T\ge1$, and
then the second requirement is implied by the first. This gives the
condition on $N$ in the theorem.

\emph{Order of $N$.} We have $1/\sigma=K^{\beta/2}/\sqrt{C_{\mathrm{sp}}}$ and
$K=K_\varepsilon(\eta)\asymp\varepsilon^{-2/(2\beta-1)}$, so
\[
 \frac{K}{\varepsilon^2}\asymp\varepsilon^{-\frac{4\beta}{2\beta-1}},
 \qquad
 \frac{K^{\beta/2}}{\varepsilon}\asymp\varepsilon^{-\frac{3\beta-1}{2\beta-1}},
\]
and $3\beta-1<4\beta$, so the first term dominates; $T=O(\lambda_0^{-1}
\log(1/\varepsilon))$ contributes only logarithmic factors.
\end{proof}

\section{Details and Proofs for Section~\ref{sec:extensions}}
\label{app:extensions}

\subsection{General data}

\begin{proof}[Proof of Proposition~\ref{prop:full-bound}]

The sampler outputs
$\tilde{X} = V_K\hat{X} + Z$ with $\hat{X} \sim p$ and
$Z \sim \nu_\perp$ independent, supported on the orthogonal
subspaces $S_K$ and $S_K^\perp$ respectively. The law of
$\tilde{X}$ is $\tilde{p} = p \otimes \nu_\perp$ (product on
$S_K \times S_K^\perp \cong \R^d$).

\begin{align}
 \TV(\tilde{p},\,q)
 &= \TV(p\otimes\nu_\perp,\;q) \notag\\
 &\le \TV(p\otimes\nu_\perp,\;q_K\otimes\nu_\perp)
 + \TV(q_K\otimes\nu_\perp,\;q).
 \label{eq:app-triangle}
\end{align}

The distributions $p\otimes\nu_\perp$ and $q_K\otimes\nu_\perp$ share
the factor $\nu_\perp$. We claim:
\begin{equation}\label{eq:app-product-tv}
 \TV(\mu_1\otimes\nu,\;\mu_2\otimes\nu)
 = \TV(\mu_1,\,\mu_2)
\end{equation}
for any probability measures $\mu_1,\mu_2$ and any probability measure $\nu$.
Indeed, let $\Delta:=\mu_1-\mu_2$ be the signed measure difference. Since
$\nu$ is a positive probability measure, the variation measure factorizes:
\[
 |\Delta\otimes\nu|=|\Delta|\otimes\nu .
\]
Therefore
\[
\begin{aligned}
 \TV(\mu_1\otimes\nu,\mu_2\otimes\nu)
 &=\frac12 |(\mu_1-\mu_2)\otimes\nu|(\mathcal X\times\mathcal Y)\\
 &=\frac12 |\mu_1-\mu_2|(\mathcal X)\,\nu(\mathcal Y)\\
 &=\TV(\mu_1,\mu_2),
\end{aligned}
\]
which proves~\eqref{eq:app-product-tv}.

Applying~\eqref{eq:app-product-tv} with $\mu_1 = p$,
$\mu_2 = q_K$, $\nu = \nu_\perp$ to the first term
of~\eqref{eq:app-triangle}:
\[
 \TV(p\otimes\nu_\perp,\;q_K\otimes\nu_\perp)
 = \TV(p,\,q_K).
\]
Substituting into~\eqref{eq:app-triangle} and using the
definition of $\varepsilon_{\mathrm{PCA}}$:
\[
 \TV(\tilde{p},\,q)
 \le \TV(p,\,q_K) + \varepsilon_{\mathrm{PCA}}.
\]
\end{proof}

\emph{Mutual-information bound.} If $\nu_\perp=q_\perp$, then
$\KL(q\,\|\,q_K\otimes q_\perp)=I$, the mutual information between the
retained and discarded coordinates, and Pinsker's inequality gives
$\varepsilon_{\mathrm{PCA}}=\TV(q_K\otimes q_\perp,q)\le\sqrt{I/2}$.

\paragraph{Filling with stationary noise.}
A sampler that runs the diffusion on all $d$ coordinates but uses, for the
discarded ones, the score $-u_i/\sigma^2$ of the stationary law has reverse
drift $\lambda_iu_i+2\sigma^2\lambda_i(-u_i/\sigma^2)=-\lambda_iu_i$ in those
coordinates, which leaves $\N(0,\sigma^2)$ invariant. Its output is therefore
$p\otimes\N(0,\sigma^2I_{d-K})$, the construction above with
$\nu_\perp=\N(0,\sigma^2I)$. This residual is not small against the
unsmoothed target: a discarded Gaussian coordinate has law
$\N(0,\sigma_i^2)$, and $\TV(\N(0,\sigma_i^2),\N(0,\sigma^2))\to1$ as
$\sigma_i/\sigma\to0$. Adding $\N(0,\eta^2)$ noise afterwards gives variance
$\sigma^2+\eta^2$ instead of $\sigma_i^2+\eta^2$, which differs from the
residual $\N(0,\eta^2)$ analyzed in Section~\ref{sec:sharp}.

\paragraph{Retaining directions below the diffusion noise level.}
To reduce $\varepsilon_{\mathrm{PCA}}$ one may retain $K>d_{\mathrm{eff}}^{\sigma}$
directions. Extending the Wiener-gain rates would let $\lambda_{\min}\to0$,
so the extra directions should get a fixed rate $\lambda_\star$, which keeps
$\kappa\le\kappa_0$ and gives $\tr(\Lambda)\le\lambda_0K$. The Lipschitz
constant is, however, no longer uniform in $K$: for Gaussian data the
time-zero score is $-\diag(\sigma_i^{-2})u$, so $L\ge\sigma_K^{-2}$ and
$L\sigma^2$ grows like $(K/d_{\mathrm{eff}}^{\sigma})^{\beta}$. The step count
of Theorem~\ref{thm:convergence} then becomes
$\widetilde O\bigl((L\sigma^2)^2K/\varepsilon^2\bigr)$ with $L=L(K)$, and the
trade-off between sampling and truncation must track this growth. Moreover,
$\varepsilon_{\mathrm{PCA}}$ need not decrease monotonically in $K$: moving the
cut can separate dependent coordinates that previously lay together in the
residual.

\subsection{Estimated principal components}

In the patch application $d = D_p$ and $C = \Sigma$, the population patch
covariance; $\widehat{\Sigma}$ denotes its empirical estimate.
\begin{assumption}[Independent patch sampling]
\label{ass:iid-patches}
The \(n_{\mathrm{patch}}\) patches
\(x^{(1)},\ldots,x^{(n_{\mathrm{patch}})}\) used to form the empirical covariance
\(\widehat\Sigma\) are drawn independently from the population patch marginal
\(\mathcal N(0,\Sigma)\).
\end{assumption}

Extracting all $N_P$ non-overlapping patches from each of $n$ training
images gives $n_{\mathrm{patch}}=nN_P$ pooled patches, which are not
necessarily independent since patches from the same image may share
global scene statistics. The i.i.d.\ concentration bound of
Lemma~\ref{lem:bernstein} applies directly under
Assumption~\ref{ass:iid-patches}; under a conservative
block-dependence view, one instead treats each image as one
independent unit, so the effective sample size is $n$ and
$\Delta_n=C_\Delta\sigma_1^2\sqrt{D_p/n}$. Any sufficient training-data
condition stated in terms of $n_{\mathrm{patch}}$ under i.i.d.\ patch
sampling should then be read as a condition on the number of
independent images $n$.

Here $d=D_p$ is the patch dimension. On the concentration event of
Lemma~\ref{lem:bernstein}, which has probability at least $1-2e^{-D_p}$ when
$n_{\mathrm{patch}}\ge D_p$, the empirical covariance satisfies
$\|\widehat\Sigma-\Sigma\|_{\mathrm{op}}\le\Delta_n$ with
$\Delta_n=C_\Delta\sigma_1^2\sqrt{D_p/n_{\mathrm{patch}}}$. The perturbation lemmas of
Appendix~\ref{app:lemmas} (Weyl, Davis-Kahan, stability of the Wiener gains,
trace perturbation, minimum empirical rate, off-diagonal covariance) propagate
this bound to the empirical basis, eigenvalues and rates.

\begin{theorem}[Smoothed Gaussian reconstruction, empirical basis and rates, oracle score]
\label{thm:empirical}
Assume Assumption~\ref{ass:iid-patches}, use Wiener-gain rates, and assume
\[
 q=\mathcal N(0,\mathcal C),
 \qquad
 \mathcal C=V\diag(\sigma_1^2,\ldots,\sigma_d^2)V^\top,
 \qquad
 \sigma_i^2=C_{\mathrm{sp}}i^{-\beta},
\]
with \(C_{\mathrm{sp}}>0\) and \(\beta>1\). Let
\[
 V_K=[v_1,\ldots,v_K],
 \qquad
 P_K:=V_KV_K^\top
\]
be the population top-\(K\) eigenspace and projector. Let \(\widehat V_K\) be
the empirical top-\(K\) PCA basis estimated from \(n_{\mathrm{patch}}\) training
patches, and define
\[
 \widehat P_K:=\widehat V_K\widehat V_K^\top.
\]
Let
\[
 \widehat\lambda_i
 =
 \lambda_0
 \frac{\widehat\sigma_i^2}{\widehat\sigma_i^2+\sigma^2},
 \qquad
 \widehat\Lambda:=\diag(\widehat\lambda_1,\ldots,\widehat\lambda_K).
\]
Assume $1\le K\le d_{\mathrm{eff}}^{\sigma}$ and that the covariance
concentration event
\[
 \|\widehat\Sigma-\Sigma\|_{\mathrm{op}}\le \Delta_n
\]
holds, where \(\Sigma\) is the population patch covariance and
\(\widehat\Sigma\) is the empirical patch covariance. Suppose further that
\[
 \Delta_n
 \le
 c_0
 \min\left\{
  \sigma^2,\,
  \sigma_K^2,\,
  \delta_{\mathrm{gap}},\,
  \frac{\eta^2\delta_{\mathrm{gap}}}{\sigma_1^2\sqrt K}
 \right\},
\]
for a sufficiently small universal constant \(c_0>0\), where
\[
 \delta_{\mathrm{gap}}:=\sigma_K^2-\sigma_{K+1}^2.
\]

In the empirical active coordinates, the true projected covariance is generally
non-diagonal:
\[
 \widehat C_K:=\widehat V_K^\top\mathcal C\,\widehat V_K.
\]
Let \(\widehat q_{K,t}\) denote the Gaussian marginal of the empirical
active-space OU process with initial covariance \(\widehat C_K\) and rate matrix
\(\widehat\Lambda\), namely
\[
 \widehat q_{K,t}
 =
 \mathcal N(0,\widehat D_t),
 \qquad
 \widehat D_t
 =
 e^{-\widehat\Lambda t}\widehat C_K e^{-\widehat\Lambda t}
 +
 \sigma^2(I_K-e^{-2\widehat\Lambda t}).
\]
Assume that the empirical active-space score is oracle-exact (note that
$\widehat D_t$ involves the unknown population covariance $\mathcal C$, so this
is an oracle assumption, not a data-derived score):
\[
 \widehat s_t(u)=\nabla\log \widehat q_{K,t}(u)
 =
 -\widehat D_t^{-1}u.
\]

Then, with the same probability as the covariance concentration event,
\[
 \TV(\widetilde p_{\mathrm{emp}}^\eta,\,q^\eta)
 \le
 \underbrace{
 \widetilde O\!\left(\sqrt{\frac K N}\right)
 }_{\mathrm{(I)}\ \mathrm{sampling}}
 +
 \underbrace{
 C_1
 \frac{\sigma_1^2\sqrt K}{\eta^2\,\delta_{\mathrm{gap}}}
 \Delta_n
 }_{\mathrm{(II)}\ \mathrm{basis/covariance\ mismatch}}
 +
 \underbrace{
 \TV\bigl(\widetilde q_K^{\,\eta},q^\eta\bigr)
 }_{\mathrm{(III)}\ \mathrm{truncation}},
\]
where \(C_1>0\) is independent of the ambient dimension \(d\), and
(III) is the truncation error of Section~\ref{sec:sharp}, bounded by
Theorem~\ref{thm:num-directions}(a). The hidden
constants in the \(\widetilde O(\cdot)\) term may depend on fixed model
parameters such as \((\lambda_0,\sigma,\eta,\beta,T,m_2)\), but not on \(d\).
The small-perturbation conditions imply that the empirical Gaussian marginals
also satisfy the coefficient-space Lipschitz-score and weighted Fisher
assumptions with constants independent of \(d\). The sampling term is applied
to the empirical OU process itself; no Girsanov comparison between the
population rate matrix \(\Lambda\) and the empirical rate matrix
\(\widehat\Lambda\) is used.
\end{theorem}

\paragraph{No separate rate-mismatch term.}
The oracle-exact Gaussian assumption eliminates the need for a
rate-mismatch term: the empirical rates $\widehat\lambda_i$ define the
empirical process itself, and $\widehat s_t=\nabla\log\widehat
q_{K,t}$ by assumption, so once
$\widehat\lambda_{\min}\ge\lambda_0/3$ and $\widehat\kappa\le3$ are
established (Lemma~\ref{lem:minrate}), the empirical sampling error
remains $\widetilde O(\sqrt{K/N})$. A separate Girsanov comparison
between $\Lambda$ and $\widehat\Lambda$ would be inappropriate here,
since changing the rates changes the diffusion coefficient as well as
the drift.

\paragraph{Dependence on the smoothing level.}
The basis/covariance mismatch term contains $\eta^{-2}$ because the
empirical and population PCA subspaces differ: without output
smoothing, the two Gaussian laws are supported on different
$K$-dimensional subspaces and are generally mutually singular. The
added covariance $\eta^2I_d$ supplies the uniform spectral lower bound
needed to compare them in total variation.

\paragraph{Explicit dependence on the empirical eigengap.}
Under the exact power-law spectrum, $\delta_{\mathrm{gap}}(K) =
C_{\mathrm{sp}}(K^{-\beta}-(K+1)^{-\beta}) \asymp C_{\mathrm{sp}}\beta
K^{-\beta-1}$, so the basis/covariance mismatch term in
Theorem~\ref{thm:empirical} scales as
$O(\sigma_1^2 K^{\beta+3/2}\Delta_n/(\eta^2 C_{\mathrm{sp}}\beta))$ -
independent of $d$, but imposing a potentially strong sample-size
requirement as $K$ grows. The empirical PCA guarantee should
accordingly be read mainly as a consistency result, since Davis-Kahan
perturbation scales with the inverse eigengap.

\paragraph{Optional diagonal empirical approximation.}\label{rem:nondiag-tv}
Replacing the true empirical-coordinate covariance $\hat
C_K:=\hat V_K^\top C\hat V_K$ by the diagonal $\hat\Sigma_K$ introduces
an additional error: under $\Delta_n\leq \hat\sigma_K^2/2$, the
standard Gaussian KL bound and Pinsker's inequality give
\[
 \TV\!\left(\mathcal{N}(0,\hat{V}_K^\top C\hat{V}_K),\;
 \mathcal{N}(0,\hat{\Sigma}_K)\right)
 \leq \frac{\sqrt{K}\,\Delta_n}{2\hat{\sigma}_K^2}.
\]
This term is not part of Theorem~\ref{thm:empirical}, which uses the
oracle score of the correct empirical-coordinate covariance $\hat C_K$.

\paragraph{Sample size.}
Let $M_K:=\min\{\sigma^2,\sigma_K^2,\delta_{\mathrm{gap}}(K),
\eta^2\delta_{\mathrm{gap}}(K)/(\sigma_1^2\sqrt K)\}$, so that
$\Delta_n\le c_0M_K$ is the small-perturbation requirement of
Theorem~\ref{thm:empirical}. Under the exact power law, $\sigma_K^2$ and
$\delta_{\mathrm{gap}}(K)$ decrease in $K$, so the most restrictive condition
over $K\le K_{\max}$ is attained at $K_{\max}$.

\begin{assumption}[Sufficient training data for empirical Gaussian reconstruction]
\label{ass:sufficient_data}
Fix \(K_{\max}\le d_{\mathrm{eff}}^{\sigma}\) and \(\varepsilon>0\). Assume that
the number of independent training patches satisfies
\[
 n_{\mathrm{patch}}
 \ge
 \widetilde\Omega\left(
 \max\left\{
  \frac{\sigma_1^4D_p}{M_{K_{\max}}^2},
  \frac{\sigma_1^8K_{\max}D_p}
  {\varepsilon^2\eta^4\delta_{\mathrm{gap}}(K_{\max})^2}
 \right\}
 \right),
\]
where
\[
 M_{K_{\max}}
 :=
 \min\left\{
  \sigma^2,\,
  \sigma_{K_{\max}}^2,\,
  \delta_{\mathrm{gap}}(K_{\max}),\,
  \frac{\eta^2\delta_{\mathrm{gap}}(K_{\max})}
  {\sigma_1^2\sqrt{K_{\max}}}
 \right\}.
\]
\end{assumption}

\begin{theorem}[Decoupling of data collection and sampling]
\label{thm:decouple}
Fix $\varepsilon>0$ and $K_{\max}=d_{\mathrm{eff}}^{\sigma}$. Under
Assumption~\ref{ass:sufficient_data}, the basis-estimation term of
Theorem~\ref{thm:empirical} is at most $\varepsilon/4$ uniformly over
$K\le K_{\max}$, so that
\begin{equation}\label{eq:reduced}
 \TV\bigl(\widetilde p_{\mathrm{emp}}^{\,\eta},q^\eta\bigr)
 \le
 \widetilde O\Bigl(\sqrt{\tfrac KN}\Bigr)
 +\TV\bigl(\widetilde q_K^{\,\eta},q^\eta\bigr)
 +\frac{\varepsilon}{4}.
\end{equation}
For a fixed diffusion noise level $\sigma$, any choice of $K\le K_{\max}$ and
$N$ that makes the population bound at most $3\varepsilon/4$ therefore gives
error at most $\varepsilon$ with the estimated basis and rates. The required
number of patches depends on $d$, the spectrum, $\sigma$, $\eta$ and
$\varepsilon$, but not on $N$.
\end{theorem}

\paragraph{Proofs.}
Theorems~\ref{thm:empirical} and~\ref{thm:decouple} are proved in Appendix~\ref{app:proofs-empirical}.

\subsection{Patch Gaussian mixtures}
Let $q=\sum_{m=1}^M\pi_m\N(0,C_m)$ with components eigenbases $V_m$, and let
$p_m^\eta$ be a truncated sampler for component $m$ in its basis $V_m$. The
mixture sampler draws a component $m$ with probability $\pi_m$ and then runs
$p_m^\eta$; its law is $p^\eta=\sum_m\pi_mp_m^\eta$, and the target is
$q^\eta=\sum_m\pi_m\N(0,C_m+\eta^2I)$.

\begin{proposition}[Mixture lifting of the Gaussian bound]
\label{prop:gmm-lift}
Let $q_{\mathrm{GMM}}^\eta=\sum_{m=1}^M\pi_mq_m^\eta$ and
$p_{\mathrm{GMM}}^\eta=\sum_{m=1}^M\pi_mp_m^\eta$. If
$\TV(p_m^\eta,q_m^\eta)\le\varepsilon_m$ for all components, then
\[
 \TV(p_{\mathrm{GMM}}^\eta,q_{\mathrm{GMM}}^\eta)
 \le\sum_{m=1}^M\pi_m\varepsilon_m.
\]
Here $p_m^\eta$ is a truncated sampler adapted to component $m$, run with
probability $\pi_m$.
\end{proposition}

\begin{proof}
The convexity of the total variation gives
\[
 \TV\!\left(\sum_{m=1}^M\pi_mp_m^\eta,
    \sum_{m=1}^M\pi_mq_m^\eta\right)
 \le\sum_{m=1}^M\pi_m\TV(p_m^\eta,q_m^\eta)
 \le\sum_{m=1}^M\pi_m\varepsilon_m.
\]
\end{proof}

\paragraph{Shared bases.}
Using a pooled basis $V$ for all components is exact only when the
component covariances are diagonal in $V$. Otherwise, an
additional error
\[
 \varepsilon_{\mathrm{shared}}
 :=\sum_{m=1}^M\pi_m\,
 \TV\bigl(\N(0,V^\top C_mV),\,\N(0,\diag(V^\top C_mV))\bigr)
\]
must be added. The mixture result is a componentwise upper bound; it is not
a claim that one global basis diagonalises natural-image patch statistics,
and it does not give a lower bound for the mixture.

\section{Auxiliary lemmas for the Girsanov argument}\label{sec:app:lemmas}
\begin{lemma}[Data processing for TV]\label{lem:tv-dp}
Let $\mu,\nu$ be probability measures on a measurable space
$(\Omega,\mathcal{F})$ and let $\Phi:\Omega\to S$ be measurable.
Then
\[
\TV(\Phi_\#\mu,\Phi_\#\nu)\le \TV(\mu,\nu).
\]
\end{lemma}

\begin{proof}
By definition,
\begin{align*}
\TV(\Phi_\# \mu, \Phi_\# \nu)
&= \sup_{A\subseteq S} \bigl| (\Phi_\# \mu)(A) - (\Phi_\# \nu)(A) \bigr| \\
&= \sup_{A\subseteq S} \bigl| \mu(\Phi^{-1}(A)) - \nu(\Phi^{-1}(A)) \bigr| \\
&\le \sup_{B\subseteq \Omega} \bigl| \mu(B) - \nu(B) \bigr| \\
&= \TV(\mu,\nu).
\end{align*}
\end{proof}

\begin{lemma}[Forward second moment bound]\label{lem:moment}
For the forward OU \eqref{eq:generic-forward-OU} on $\R^r$,
\[
\sup_{\tau\in[0,T]}\E\norm{U_\tau}^2 \le m_{2,r}^2 + r\sigma^2.
\]
\end{lemma}

\begin{proof}
Solve \eqref{eq:generic-forward-OU} explicitly:
\[
U_\tau = e^{-\bLambda\tau}U_0
 + \int_0^\tau e^{-\bLambda(\tau-s)}\sqrt{2\bLambda}\,\sigma\,dB_s.
\]
Then:
\[
\E\left[\norm{U_\tau}^2\right]
= \E\left[\norm{e^{-\bLambda\tau}U_0}^2\right]
+ \E\left[\Big\|\int_0^\tau e^{-\bLambda(\tau-s)}\sqrt{2\bLambda}\,\sigma\,dB_s\Big\|^2\right].
\]
The cross term vanishes because the stochastic integral has zero
mean and is independent of $U_0$.
For the first term, $\|e^{-\bLambda\tau}\|_{\mathrm{op}}\le 1$,
so $\E\norm{e^{-\bLambda\tau}U_0}^2\le \E\norm{U_0}^2$.
For the second term, by It\^{o}'s isometry:
\begin{align*}
\E\left[\Big\|\int_0^\tau e^{-\bLambda(\tau-s)}\sqrt{2\bLambda}\,\sigma\,dB_s\Big\|^2\right]
&= \int_0^\tau \tr\!\Bigl(e^{-\bLambda(\tau-s)}(2\bLambda\sigma^2)
 e^{-\bLambda(\tau-s)}\Bigr)\,ds.
\end{align*}
\begin{align*}
\int_0^\tau \tr\!\Bigl(e^{-\bLambda(\tau-s)}(2\bLambda\sigma^2)
e^{-\bLambda(\tau-s)}\Bigr)\,ds
&= \sum_{i=1}^r \int_0^\tau 2\lambda_i\sigma^2 e^{-2\lambda_i(\tau-s)}\,ds \\
&= \sum_{i=1}^r \sigma^2\!\left(1 - e^{-2\lambda_i\tau}\right) \\
&\le \sum_{i=1}^r \sigma^2 = r\sigma^2.
\end{align*}
Finally, $\E\norm{U_0}^2 = m_{2,r}^2$ by definition.
Combining gives $\sup_{\tau}\E\norm{U_\tau}^2\le m_{2,r}^2+r\sigma^2$.
\end{proof}

\begin{lemma}[Fisher bound from score regularity]\label{lem:fisher}
Under Assumption~\ref{ass:lipschitz} (applied to the coefficient-space
marginals $q_{r,t}$ on $\mathbb{R}^r$), for every $t \in [\tau_0, T]$ and every
diagonal positive semidefinite matrix
$M = \mathrm{diag}(m_1, \dots, m_r)$, $m_i \ge 0$,
\[
  \mathbb{E}_{U \sim q_{r,t}}\bigl[\|\nabla \log q_{r,t}(U)\|_M^2\bigr]
  \le L \,\mathrm{tr}(M).
\]
In particular, for $M = I_r$,
$\mathbb{E}_{U \sim q_{r,t}}\bigl[\|\nabla \log q_{r,t}(U)\|^2\bigr] \le L r$.
\end{lemma}
\begin{proof}
Write $q_{r,t}=e^{-V_t}$ for some $V_t\in C^2(\mathbb R^r)$, as
guaranteed by Assumption~\ref{ass:lipschitz}. Fix a coordinate
\(i\in\{1,\ldots,r\}\). Since \(q_{r,t}=e^{-V_t}\), we have
\[
 \partial_i q_{r,t}(x)
 =
 -q_{r,t}(x)\,\partial_i V_t(x).
\]
Therefore
\[
 \partial_i\!\left(q_{r,t}\,\partial_i V_t\right)
 =
 (\partial_i q_{r,t})\,\partial_i V_t
 +
 q_{r,t}\,\partial_{ii}^2 V_t
 =
 -q_{r,t}(\partial_i V_t)^2
 +
 q_{r,t}\,\partial_{ii}^2 V_t.
\]
Rearranging gives the pointwise identity
\[
 q_{r,t}(\partial_i V_t)^2
 =
 q_{r,t}\,\partial_{ii}^2 V_t
 -
 \partial_i\!\left(q_{r,t}\,\partial_i V_t\right).
\]

Integrating over the cube \(Q_R=[-R,R]^r\) and applying the divergence theorem,
\[
\begin{aligned}
 \int_{Q_R} q_{r,t}(x)(\partial_i V_t(x))^2\,dx
 &=
 \int_{Q_R} q_{r,t}(x)\,\partial_{ii}^2 V_t(x)\,dx \\
 &\quad
 -
 \int_{\partial Q_R}
 q_{r,t}(x)\,\partial_i V_t(x)\,n_i(x)\,dS(x).
\end{aligned}
\]
By the boundary condition in Assumption~\ref{ass:lipschitz}, the
boundary term converges to zero as \(R\to\infty\).

Since the score is \(L\)-Lipschitz and \(V_t\in C^2\), we have
\[
 \|\nabla^2 V_t(x)\|_{\mathrm{op}}\le L,
\]
and hence, for every \(i\),
\[
 \partial_{ii}^2 V_t(x)\le L.
\]
Moreover \(|\partial_{ii}^2 V_t(x)|\le L\), so dominated convergence applies to
the term
\[
 \int_{Q_R} q_{r,t}(x)\,\partial_{ii}^2V_t(x)\,dx,
\]
while the left-hand side converges by monotone convergence, its integrand
being nonnegative.
Letting \(R\to\infty\), we obtain
\[
 \mathbb E_{U\sim q_{r,t}}
 \left[
  (\partial_i V_t(U))^2
 \right]
 =
 \mathbb E_{U\sim q_{r,t}}
 \left[
  \partial_{ii}^2 V_t(U)
 \right]
 \le
 L.
\]

Since
\[
 \nabla\log q_{r,t} = -\nabla V_t,
\]
we have
\[
 (\partial_i\log q_{r,t})^2
 =
 (\partial_i V_t)^2.
\]
Therefore
\[
 \mathbb E_{U\sim q_{r,t}}
 \left[
  (\partial_i\log q_{r,t}(U))^2
 \right]
 \le
 L.
\]

Multiplying by \(m_i\ge0\) and summing over \(i=1,\ldots,r\) gives
\[
\begin{aligned}
 \mathbb E_{U\sim q_{r,t}}
 \left[
  \|\nabla\log q_{r,t}(U)\|_M^2
 \right]
 &=
 \sum_{i=1}^r
 m_i\,
 \mathbb E_{U\sim q_{r,t}}
 \left[
  (\partial_i\log q_{r,t}(U))^2
 \right] \\
 &\le
 \sum_{i=1}^r
 m_i L
 =
 L\,\tr(M).
\end{aligned}
\]
Taking \(M=I_r\) gives
\[
 \mathbb E_{U\sim q_{r,t}}
 \left[
  \|\nabla\log q_{r,t}(U)\|^2
 \right]
 \le
 Lr.
\]
This completes the proof.
\end{proof}
\begin{lemma}[Forward movement bound]\label{lem:move-forward}
Let $0\le\tau<\tau'\le T$ and set $\Delta:=\tau'-\tau$. Then
\[
\E\left[\norm{U_{\tau'}-U_\tau}^2\right]
\le 2\lambda_{\max}^2\Delta^2\!\left(\E[\|U_\tau\|^2]+r\sigma^2\right)
+ 4\sigma^2\tr(\bLambda)\Delta.
\]
In particular, using Lemma~\ref{lem:moment} and
$\tr(\bLambda)\le r\lambda_{\max}$,
\[
\E\left[\norm{U_{\tau'}-U_\tau}^2\right]
\le 4\Bigl(\lambda_{\max}^2\Delta^2(m_{2,r}^2+2r\sigma^2)
 +\sigma^2 r\lambda_{\max}\Delta\Bigr).
\]
\end{lemma}

\begin{proof}
\[
U_{\tau'}-U_\tau
= \int_\tau^{\tau'}(-\bLambda U_s)\,ds
+ \int_\tau^{\tau'}\sqrt{2\bLambda}\,\sigma\,dB_s.
\]
Since $(a+b)^2\le 2a^2+2b^2$,
\[
\E\left[\norm{U_{\tau'}-U_\tau}^2\right]
\le 2\E\left[\Big\|\int_\tau^{\tau'}\bLambda U_s\,ds\Big\|^2\right]
+2\E\left[\Big\|\int_\tau^{\tau'}\sqrt{2\bLambda}\,\sigma\,dB_s\Big\|^2\right].
\]
By Cauchy-Schwarz and Jensen:
\begin{align*}
\E\left[\Big\|\int_\tau^{\tau'}\bLambda U_s\,ds\Big\|^2\right]
&= \E\left[\left\langle\int_\tau^{\tau'}\bLambda U_s\,ds,
 \int_\tau^{\tau'}\bLambda U_s\,ds\right\rangle\right] \\
&\le \E\left[\left(\int_\tau^{\tau'}\norm{\bLambda U_s}\,ds\right)^2\right] \\
&\le \E\left[\left(\int_\tau^{\tau'}1\,ds\right)
 \left(\int_\tau^{\tau'}\norm{\bLambda U_s}^2\,ds\right)\right] \\
&= \Delta\int_\tau^{\tau'}\E\!\left[\norm{\bLambda U_s}^2\right]ds.
\end{align*}
By It\^{o}'s isometry:
\[
\E\left[\Big\|\int_\tau^{\tau'}\sqrt{2\bLambda}\,\sigma\,dB_s\Big\|^2\right]
= \int_\tau^{\tau'}\tr(2\bLambda\sigma^2)\,ds
= 2\sigma^2\tr(\bLambda)\Delta.
\]
Therefore:
\begin{align*}
\E\left[\norm{U_{\tau'}-U_\tau}^2\right]
&\le 2\Delta\int_\tau^{\tau'}\lambda_{\max}^2\E\!\left[\norm{U_s}^2\right]ds
+4\sigma^2\tr(\bLambda)\Delta \\
&= 2\lambda_{\max}^2\Delta\int_\tau^{\tau'}\E\!\left[\norm{U_s}^2\right]ds
+4\sigma^2\tr(\bLambda)\Delta.
\end{align*}
From the OU explicit solution:
\[
U_s = e^{-\bLambda(s-\tau)}U_\tau
+\int_\tau^s e^{-\bLambda(s-r)}\sqrt{2\bLambda}\,\sigma\,dB_r,
\]
so
\begin{align*}
\E[\|U_s\|^2]
&\le \E[\|e^{-\bLambda(s-\tau)}U_\tau\|^2]
 +\sigma^2\sum_{i=1}^r(1-e^{-2\lambda_i(s-\tau)}) \\
&\le \E[\|U_\tau\|^2]+r\sigma^2.
\end{align*}
Substituting:
\begin{align*}
\E\left[\norm{U_{\tau'}-U_\tau}^2\right]
&\le 2\lambda_{\max}^2\Delta\int_\tau^{\tau'}
 \!\left(\E[\|U_\tau\|^2]+r\sigma^2\right)ds
+4\sigma^2\tr(\bLambda)\Delta \\
&= 2\lambda_{\max}^2\Delta^2\!\left(\E[\|U_\tau\|^2]+r\sigma^2\right)
+4\sigma^2\tr(\bLambda)\Delta.
\end{align*}
From Lemma~\ref{lem:moment}:
\[
\E\left[\norm{U_{\tau'}-U_\tau}^2\right]
\le 2\lambda_{\max}^2\Delta^2\!\left(m_{2,r}^2+2r\sigma^2\right)
+4\sigma^2\tr(\bLambda)\Delta.
\]
Since $\tr(\bLambda)\le r\lambda_{\max}$,
\[
\E\left[\norm{U_{\tau'}-U_\tau}^2\right]
\le 2\lambda_{\max}^2\Delta^2\!\left(m_{2,r}^2+2r\sigma^2\right)
+4\sigma^2 r\lambda_{\max}\Delta.
\]
\end{proof}
\begin{lemma}[$\bLambda$-weighted forward movement bound]
\label{lem:move-forward-weighted}
Let \(0 \leq \tau < \tau' \leq T\) and set \(\Delta := \tau' - \tau\).
For the forward OU process in \(\mathbb{R}^K\) with rate matrix
\(\Lambda = \mathrm{diag}(\lambda_1,\ldots,\lambda_K)\), all \(\lambda_i > 0\):
\begin{equation}\label{eq:weighted-movement}
 \mathbb{E}\bigl[\|U_{\tau'} - U_\tau\|^2_\Lambda\bigr]
 \;\leq\;
 4\Bigl[\lambda_{\max}^2\Delta^2
 \Bigl(m_{2,\Lambda}^2 + \sigma^2\,\mathrm{tr}(\Lambda)\Bigr)
 + \sigma^2\Delta\,\mathrm{tr}(\Lambda^2)\Bigr].
\end{equation}
\end{lemma}

\begin{proof}
From the SDE
$dU_t=-\bLambda U_t\,dt+\sigma\sqrt{2\bLambda}\,dB_t$,
the increment satisfies
\[
U_{\tau'}-U_\tau
=\int_\tau^{\tau'}(-\bLambda U_s)\,ds
+\int_\tau^{\tau'}\sigma\sqrt{2\bLambda}\,dB_s.
\]
By $(a+b)^2\le 2(a^2+b^2)$ applied in the
$\bLambda$-norm:
\begin{equation}
\E\|U_{\tau'}-U_\tau\|^2_{\bLambda}
\le 2\,\E\!\left\|
 \int_\tau^{\tau'}\bLambda U_s\,ds
 \right\|^2_{\bLambda}
+2\,\E\!\left\|
 \int_\tau^{\tau'}\sigma\sqrt{2\bLambda}\,dB_s
 \right\|^2_{\bLambda}.
\label{eq:split-weighted}
\end{equation}

\medskip\noindent
\textbf{Drift term.}
By the Cauchy-Schwarz inequality:
\[
\left\|\int_\tau^{\tau'}\bLambda U_s\,ds
\right\|^2_{\bLambda}
\le\Delta\int_\tau^{\tau'}
 \|\bLambda U_s\|^2_{\bLambda}\,ds.
\]
Now:
\[
\|\bLambda U_s\|^2_{\bLambda}
=\sum_i\lambda_i(\lambda_i U_{s,i})^2
=\sum_i\lambda_i^3 U_{s,i}^2
\le\lambda_{\max}\sum_i\lambda_i^2 U_{s,i}^2
=\lambda_{\max}\|\bLambda^{1/2}U_s\|^2_{\bLambda}.
\]
Since $\|\bLambda^{1/2}U_s\|^2_{\bLambda}
=\sum_i\lambda_i\cdot\lambda_i U_{s,i}^2
=\|U_s\|^2_{\bLambda^2}
\le\lambda_{\max}\|U_s\|^2_{\bLambda}$,
we have:
\[
\|\bLambda U_s\|^2_{\bLambda}
\le\lambda_{\max}^2\|U_s\|^2_{\bLambda}.
\]
Therefore:
\begin{equation}
\mathbb{E}\left\|\int_{\tau}^{\tau'} \Lambda U_s\,ds\right\|^2_\Lambda
\;\leq\;
\lambda_{\max}^2 \Delta^2
\Bigl(m_{2,\Lambda}^2 + \sigma^2\,\mathrm{tr}(\Lambda)\Bigr),
\label{eq:drift-weighted}
\end{equation}
here we used the pointwise bound
\begin{equation*}
 \mathbb{E}\|U_s\|^2_\Lambda
 \;=\; \sum_i \lambda_i\,\mathbb{E}[U_{s,i}^2]
 \;=\; \sum_i \lambda_i
   \Bigl[e^{-2\lambda_i s}\mathbb{E}[U_{0,i}^2]
    + \sigma^2(1 - e^{-2\lambda_i s})\Bigr]
 \;\leq\; \sum_i \lambda_i\,\mathbb{E}[U_{0,i}^2]
   + \sigma^2\sum_i \lambda_i
\end{equation*}
\begin{equation}\label{eq:ou-second-moment-lambda}
 \;=\; m_{2,\Lambda}^2 + \sigma^2\,\mathrm{tr}(\Lambda),
\end{equation}
which follows from the explicit OU marginal
$\mathbb{E}[U_{s,i}^2] = e^{-2\lambda_i s}\mathbb{E}[U_{0,i}^2]
+ \sigma^2(1-e^{-2\lambda_i s})$ by multiplying by $\lambda_i$ and summing;
the bound holds for all $s \geq 0$ and requires no knowledge of $T$.

\medskip\noindent
\textbf{Noise term.}
By It\^o's isometry in the $\bLambda$-norm:
\begin{align}
\E\!\left\|\int_\tau^{\tau'}
 \sigma\sqrt{2\bLambda}\,dB_s
 \right\|^2_{\bLambda}
&=\int_\tau^{\tau'}
 \tr\!\bigl(\bLambda\cdot
 \sigma\sqrt{2\bLambda}\cdot
 (\sigma\sqrt{2\bLambda})^\top\bigr)\,ds
\nonumber\\
&=\int_\tau^{\tau'}
 \tr(2\sigma^2\bLambda^2)\,ds
\nonumber\\
&=2\sigma^2\Delta\,\tr(\bLambda^2).
\label{eq:noise-weighted}
\end{align}
Here we used that for a diagonal diffusion
coefficient $G=\sigma\sqrt{2\bLambda}$,
the It\^o isometry gives
$\E\|\int_\tau^{\tau'}G\,dB_s\|^2_{\bLambda}
=\int_\tau^{\tau'}\tr(\bLambda GG^\top)\,ds
=\int_\tau^{\tau'}\tr(2\sigma^2\bLambda^2)\,ds$,
since $\bLambda GG^\top
=\bLambda\cdot 2\sigma^2\bLambda=2\sigma^2\bLambda^2$.

\medskip\noindent
\textbf{Combining.}
Substituting~\eqref{eq:drift-weighted}
and~\eqref{eq:noise-weighted}
into~\eqref{eq:split-weighted} gives
\[
 \E\|U_{\tau'}-U_\tau\|^2_{\bLambda}
 \le
 2\lambda_{\max}^2\Delta^2
 \bigl(m_{2,\bLambda}^2+\sigma^2\tr(\bLambda)\bigr)
 +
 4\sigma^2\Delta\,\tr(\bLambda^2).
\]
Loosening constants yields the stated bound
\[
 \E\|U_{\tau'}-U_\tau\|^2_{\bLambda}
 \le
 4\Bigl[
  \lambda_{\max}^2\Delta^2
  \bigl(m_{2,\bLambda}^2+\sigma^2\tr(\bLambda)\bigr)
  +
  \sigma^2\Delta\,\tr(\bLambda^2)
 \Bigr].
\]
\end{proof}

The next two lemmas and one corollary control how the score changes
over one diagonal OU step. They share the same proof skeleton (change
of variables, denoising identity, conditional Brascamp-Lieb). They
differ only in the norm and in how the Gaussian covariance
$\Sigma=\sigma^2(I-D^2)$ enters the remainder:
\begin{itemize}
\item Lemma~\ref{lem:score-perturb} is the unweighted Euclidean bound, with remainder
 $L\sigma\sqrt{d\zeta}$;
\item Corollary~\ref{lem:trace_aware} is the same estimate written in trace form, with
 $L\sqrt{\mathrm{tr}(\Sigma)}$ in place of $L\sigma\sqrt{d\zeta}$;
\item Lemma~\ref{lem:trace_aware_weighted} is the $\Lambda$-weighted, $\kappa$-dependent version
 needed for the anisotropic analysis.
\end{itemize}

Throughout Lemmas~\ref{lem:score-perturb}-\ref{lem:trace_aware_weighted} and
Corollary~\ref{lem:trace_aware}, $d$ denotes a generic dimension; in
Appendix~\ref{sec:app:C} they are applied with $d = r$, the dimension of
the retained coefficients, and not the ambient dimension.
\begin{lemma}[Score perturbation for a diagonal OU step]
\label{lem:score-perturb}
Let $Y\sim q$ on $\R^d$ with a strictly positive $C^2$ density and
$L$-Lipschitz score $s(y):=\nabla\log q(y)$. Assume that the differentiations
under the integral sign, integrations by parts, and conditional Brascamp-Lieb
steps used below are justified; for instance, this holds under standard
smoothness and tail-decay conditions ensuring that the required boundary terms
vanish.
Let $D=\diag(\rho_1,\dots,\rho_d)$ with
\[
\rho_i\in[1-\zeta,1),\qquad \zeta\in(0,1/2].
\]
Define
\[
\Sigma:=\sigma^2(I-D^2),
\qquad
Y':=DY,
\qquad
X:=Y'+\varepsilon,
\qquad
\varepsilon\sim\mathcal N(0,\Sigma)
\ \text{independent of }Y.
\]
Let $\tilde q:=\law(X)$ and $\tilde s(u):=\nabla\log \tilde q(u)$.
If
\[
16L\sigma^2\zeta\le 1,
\]
then for every $u\in\R^d$,
\[
\|\tilde s(u)-s(u)\|
\;\le\;
C\Bigl(
L\zeta\|u\|
+
(1+L\sigma^2)\zeta\,\|s(u)\|
+
L\sigma\sqrt{d\zeta}
\Bigr),
\]
for a universal constant $C>0$. If moreover
$\zeta\le\min\{1/8,\,1/(128L\sigma^2)\}$, then, with a larger universal
constant, the same perturbation bound may be written with the later-time score
\(\tilde s\) in the middle term:
\[
\|\tilde s(u)-s(u)\|
\;\le\;
C\Bigl(
L\zeta\|u\|
+
(1+L\sigma^2)\zeta\,\|\tilde s(u)\|
+
L\sigma\sqrt{d\zeta}
\Bigr).
\]
\end{lemma}

\begin{proof}
Since $\rho_i<1$ for all $i$, the covariance matrix
$\Sigma=\sigma^2(I-D^2)$ is positive definite.

\medskip
\noindent
Let $q_{Y'}$ be the density of $Y'=DY$. By the change-of-variables formula,
\[
q_{Y'}(u)=|\det D|^{-1}q(D^{-1}u),
\]
hence
\[
s_{Y'}(u):=\nabla\log q_{Y'}(u)=D^{-1}s(D^{-1}u).
\]
Since $\rho_i\in[1-\zeta,1)$ and $\zeta\le 1/2$, we have
$\rho_i\ge 1/2$, so
\[
\|D^{-1}\|_{\op}\le 2,
\qquad
\|D^{-1}-I\|_{\op}\le 2\zeta.
\]
Therefore
\begin{align}
\|s_{Y'}(u)-s(u)\|
&=\|D^{-1}s(D^{-1}u)-s(u)\| \notag\\
&\le \|D^{-1}\|_{\op}\,\|s(D^{-1}u)-s(u)\|
 +\|D^{-1}-I\|_{\op}\,\|s(u)\| \notag\\
&\le 2L\|D^{-1}u-u\|+2\zeta\|s(u)\| \notag\\
&\le 4L\zeta\|u\|+2\zeta\|s(u)\|.
\label{eq:syprime-minus-s}
\end{align}
In particular,
\begin{equation}
\|s_{Y'}(u)\|
\le \|s(u)\|+\|s_{Y'}(u)-s(u)\|
\le 4L\zeta\|u\|+(1+2\zeta)\|s(u)\|.
\label{eq:syprime-bound}
\end{equation}

\medskip
\noindent
Since $\Sigma\succ 0$, the density of $X=Y'+\varepsilon$ is
\[
\tilde q(u)=\int_{\R^d} q_{Y'}(y)\,\varphi_\Sigma(u-y)\,dy,
\]
where $\varphi_\Sigma$ is the Gaussian density with covariance $\Sigma$.

Differentiating under the integral sign and using
$\nabla_u\varphi_\Sigma(u-y)=-\Sigma^{-1}(u-y)\varphi_\Sigma(u-y)$,
we obtain
\[
\tilde s(u)
=
-\Sigma^{-1}
\frac{\int_{\R^d}(u-y)\,q_{Y'}(y)\varphi_\Sigma(u-y)\,dy}
{\tilde q(u)}
=
-\Sigma^{-1}\bigl(u-\E[Y'\mid X=u]\bigr).
\]
Equivalently, if we set
\[
\bar\varepsilon:=\E[u-Y'\mid X=u],
\]
then
\begin{equation}
\tilde s(u)=-\Sigma^{-1}\bar\varepsilon.
\label{eq:posterior-score}
\end{equation}

Also, integrating by parts in $y$,
\begin{align*}
\nabla \tilde q(u)
&=
\int q_{Y'}(y)\,\nabla_u\varphi_\Sigma(u-y)\,dy \\
&=
-\int q_{Y'}(y)\,\nabla_y\varphi_\Sigma(u-y)\,dy \\
&=
\int \nabla_y q_{Y'}(y)\,\varphi_\Sigma(u-y)\,dy \\
&=
\int q_{Y'}(y)\,s_{Y'}(y)\,\varphi_\Sigma(u-y)\,dy.
\end{align*}
Dividing by $\tilde q(u)$ yields
\begin{equation}
\tilde s(u)=\E[s_{Y'}(Y')\mid X=u].
\label{eq:denoising-identity}
\end{equation}

For all $u,v\in\R^d$,
\[
\|s_{Y'}(u)-s_{Y'}(v)\|
=
\|D^{-1}(s(D^{-1}u)-s(D^{-1}v))\|
\le
\|D^{-1}\|_{\op}^2 L \|u-v\|.
\]
Hence $s_{Y'}$ is $L'$-Lipschitz with
\begin{equation}
L'\le 4L.
\label{eq:Lprime}
\end{equation}

Let
\[
\varepsilon:=u-Y'.
\]
Conditioned on $X=u$, the variable $\varepsilon$ has density
\[
\pi_u(\varepsilon)\propto
q_{Y'}(u-\varepsilon)
\exp\!\Bigl(-\frac12\varepsilon^\top\Sigma^{-1}\varepsilon\Bigr).
\]
Write
\[
V_u(\varepsilon)
:=
-\log q_{Y'}(u-\varepsilon)
+\frac12\varepsilon^\top\Sigma^{-1}\varepsilon.
\]
Then
\[
\nabla^2 V_u(\varepsilon)
=
-\nabla^2 \log q_{Y'}(u-\varepsilon)+\Sigma^{-1}.
\]
Because $s_{Y'}=\nabla\log q_{Y'}$ is $L'$-Lipschitz,
\[
-\nabla^2\log q_{Y'}(z)\succeq -L' I
\qquad\text{for all }z\in\R^d,
\]
so
\[
\nabla^2V_u(\varepsilon)\succeq \Sigma^{-1}-L'I.
\]
Now
\[
\|\Sigma\|_{\op}
=
\sigma^2\max_i(1-\rho_i^2)
\le 2\sigma^2\zeta.
\]
Using \eqref{eq:Lprime} and the hypothesis $16L\sigma^2\zeta\le 1$,
\[
L'\|\Sigma\|_{\op}
\le 4L\cdot 2\sigma^2\zeta
\le \frac12.
\]
Therefore
\[
\nabla^2V_u(\varepsilon)\succeq \frac12\,\Sigma^{-1}
\qquad\text{for all }\varepsilon.
\]
By the Brascamp-Lieb inequality,
\[
\Cov_{\pi_u}(\varepsilon)\preceq 2\Sigma.
\]
Hence
\begin{equation}
\E_{\pi_u}\|\varepsilon-\bar\varepsilon\|^2
\le 2\,\tr(\Sigma),
\qquad
\bar\varepsilon:=\E_{\pi_u}[\varepsilon].
\label{eq:posterior-var}
\end{equation}
By Jensen,
\begin{equation}
\E_{\pi_u}\|\varepsilon\|
\le
\|\bar\varepsilon\|+\sqrt{2\,\tr(\Sigma)}.
\label{eq:Eeps-basic}
\end{equation}

From \eqref{eq:posterior-score} and \eqref{eq:denoising-identity},
\[
\bar\varepsilon
=
-\Sigma \tilde s(u)
=
-\Sigma \E_{\pi_u}[s_{Y'}(u-\varepsilon)].
\]
Therefore
\begin{align}
\|\bar\varepsilon\|
&\le
\|\Sigma\|_{\op}\,
\E_{\pi_u}\|s_{Y'}(u-\varepsilon)\| \notag\\
&\le
\|\Sigma\|_{\op}
\Bigl(
\|s_{Y'}(u)\| + L' \E_{\pi_u}\|\varepsilon\|
\Bigr).
\label{eq:bareps-1}
\end{align}
Combining \eqref{eq:Eeps-basic} with \eqref{eq:bareps-1},
\[
\|\bar\varepsilon\|
\le
\|\Sigma\|_{\op}\|s_{Y'}(u)\|
+
\|\Sigma\|_{\op}L'
\Bigl(\|\bar\varepsilon\|+\sqrt{2\,\tr(\Sigma)}\Bigr).
\]
Since $L'\|\Sigma\|_{\op}\le 1/2$,
\[
\|\bar\varepsilon\|
\le
2\|\Sigma\|_{\op}\|s_{Y'}(u)\|
+
2L'\|\Sigma\|_{\op}\sqrt{2\,\tr(\Sigma)}.
\]
Substituting into \eqref{eq:Eeps-basic} and using again
$L'\|\Sigma\|_{\op}\le 1/2$,
\begin{equation}
\E_{\pi_u}\|\varepsilon\|
\le
2\|\Sigma\|_{\op}\|s_{Y'}(u)\|
+
2\sqrt{2\,\tr(\Sigma)}.
\label{eq:Eeps-final}
\end{equation}

Using \eqref{eq:denoising-identity}, the Lipschitz property of $s_{Y'}$,
and \eqref{eq:Eeps-final},
\begin{align*}
\|\tilde s(u)-s_{Y'}(u)\|
&=
\Bigl\|
\E_{\pi_u}\bigl[s_{Y'}(u-\varepsilon)-s_{Y'}(u)\bigr]
\Bigr\| \\
&\le
L' \E_{\pi_u}\|\varepsilon\| \\
&\le
2L'\|\Sigma\|_{\op}\|s_{Y'}(u)\|
+
2L'\sqrt{2\,\tr(\Sigma)}.
\end{align*}
Using $L'\le 4L$, $\|\Sigma\|_{\op}\le 2\sigma^2\zeta$, and
\[
\tr(\Sigma)
=
\sigma^2\sum_{i=1}^d(1-\rho_i^2)
\le 2\sigma^2\zeta d,
\]
we get
\begin{equation}
\|\tilde s(u)-s_{Y'}(u)\|
\le
16L\sigma^2\zeta\,\|s_{Y'}(u)\|
+
C_1 L\sigma\sqrt{d\zeta},
\label{eq:tilde-minus-syprime}
\end{equation}
for a universal constant $C_1>0$.

Now substitute \eqref{eq:syprime-bound} into \eqref{eq:tilde-minus-syprime}:
\begin{align*}
\|\tilde s(u)-s_{Y'}(u)\|
&\le
16L\sigma^2\zeta
\Bigl(
4L\zeta\|u\|+(1+2\zeta)\|s(u)\|
\Bigr)
+
C_1L\sigma\sqrt{d\zeta}.
\end{align*}
Since $16L\sigma^2\zeta\le 1$ and $\zeta\le 1/2$,
\[
16L\sigma^2\zeta\cdot 4L\zeta\|u\|
\le 4L\zeta\|u\|,
\qquad
16L\sigma^2\zeta(1+2\zeta)\|s(u)\|
\le 32L\sigma^2\zeta\,\|s(u)\|.
\]
Hence
\begin{equation}
\|\tilde s(u)-s_{Y'}(u)\|
\le
4L\zeta\|u\|
+
32L\sigma^2\zeta\,\|s(u)\|
+
C_1L\sigma\sqrt{d\zeta}.
\label{eq:tilde-minus-syprime-2}
\end{equation}

By \eqref{eq:syprime-minus-s} and \eqref{eq:tilde-minus-syprime-2},
\begin{align*}
\|\tilde s(u)-s(u)\|
&\le
\|\tilde s(u)-s_{Y'}(u)\|+\|s_{Y'}(u)-s(u)\| \\
&\le
\Bigl(4L\zeta\|u\|
+
32L\sigma^2\zeta\,\|s(u)\|
+
C_1L\sigma\sqrt{d\zeta}\Bigr)
+
\Bigl(4L\zeta\|u\|+2\zeta\|s(u)\|\Bigr) \\
&\le
C\Bigl(
L\zeta\|u\|
+
(1+L\sigma^2)\zeta\,\|s(u)\|
+
L\sigma\sqrt{d\zeta}
\Bigr),
\end{align*}
for a universal constant $C>0$.

For the alternative form, the display before the last one gives
\[
 \|\tilde s(u)-s(u)\|
 \le8L\zeta\|u\|+b\,\|s(u)\|+C_1L\sigma\sqrt{d\zeta},
 \qquad b:=(32L\sigma^2+2)\zeta.
\]
Substituting $\|s(u)\|\le\|\tilde s(u)\|+\|\tilde s(u)-s(u)\|$ yields
$(1-b)\|\tilde s(u)-s(u)\|\le8L\zeta\|u\|+b\|\tilde s(u)\|
+C_1L\sigma\sqrt{d\zeta}$. Under
$\zeta\le\min\{1/8,1/(128L\sigma^2)\}$ we have $32L\sigma^2\zeta\le1/4$
and $2\zeta\le1/4$, so $b\le1/2$, and dividing by $1-b\ge1/2$ gives the
alternative form with the constant doubled. (Without the extra condition the
coefficient $b$ can exceed one: $L\sigma^2=1$, $\zeta=1/32$ gives
$b=34/32$.)
\end{proof}

\begin{corollary}[Trace-aware one-step score perturbation]
\label{lem:trace_aware}
Under the hypotheses of Lemma~\ref{lem:score-perturb},
\[
\bigl\|\tilde s(u)-s(u)\bigr\|
\le C\Bigl(
 L\zeta\|u\|
 +(1+L\sigma^2)\zeta\|s(u)\|
 +L\sqrt{\mathrm{tr}(\Sigma)}
\Bigr),
\]
for a universal constant $C>0$.
\end{corollary}

\begin{proof}
The proof of Lemma~\ref{lem:score-perturb} never uses the dimension $d$ except through the estimate
\[
\mathrm{tr}(\Sigma)
=\sigma^2\sum_{i=1}^d(1-\rho_i^2)
\le 2\sigma^2\zeta d,
\]
which produces the remainder $L\sigma\sqrt{d\zeta}$.
Stopping at $\mathrm{tr}(\Sigma)$ yields the claim.
The later-time form with $\|\tilde s(u)\|$ in place of $\|s(u)\|$
follows exactly as in Lemma~\ref{lem:score-perturb}.
\end{proof}

\begin{lemma}[$\bLambda$-weighted score perturbation,
$\kappa$-dependent version]
\label{lem:trace_aware_weighted}
Let $Y\sim q$ on $\R^d$ with density $q=e^{-V}/Z$, where
$V\in C^2(\R^d)$, and assume that the score
$s(y):=\nabla\log q(y)$ is $L$-Lipschitz. Assume also that the differentiations
under the integral sign, integrations by parts, and conditional Brascamp-Lieb
steps used below are justified; for instance, this holds under standard
smoothness and tail-decay conditions ensuring that the required boundary terms
vanish.
Let $D=\diag(\rho_1,\dots,\rho_d)$ with
$\rho_i\in[1-\zeta,1)$, $\zeta\in(0,1/2]$.
Define
\[
\Sigma:=\sigma^2(I-D^2),
\qquad
Y':=DY,
\qquad
X:=Y'+\varepsilon,
\qquad
\varepsilon\sim\mathcal{N}(0,\Sigma)
\text{ independent of }Y.
\]
Let $\tilde{q}:=\law(X)$ and
$\tilde{s}(u):=\nabla\log\tilde{q}(u)$.
Let $\bLambda=\diag(\lambda_1,\dots,\lambda_d)$ with
all $\lambda_i>0$, and let
$\kappa:=\lambda_{\max}/\lambda_{\min}$.
If
\begin{equation}
16\sqrt{\kappa}\,L\sigma^2\zeta\le 1,
\label{eq:small-step-weighted}
\end{equation}
then for every $u\in\R^d$,
\begin{equation}
\|\tilde{s}(u)-s(u)\|_{\bLambda}
\le C(\kappa)\Bigl(
 L\zeta\|u\|_{\bLambda}
 +(1+L\sigma^2)\zeta\|s(u)\|_{\bLambda}
 +L\sqrt{\tr(\bLambda\Sigma)}
\Bigr),
\label{eq:weighted-perturb}
\end{equation}
where $C(\kappa)>0$ is a constant depending only on
$\kappa$ and universal constants. If moreover
\[
 \zeta\le\min\{1/8,\,1/(128\sqrt\kappa\,L\sigma^2)\},
\]
then, after increasing \(C(\kappa)\), one may replace
\(\|s(u)\|_\Lambda\) by the later-time score
\(\|\tilde s(u)\|_\Lambda\):
\begin{equation}
\|\tilde{s}(u)-s(u)\|_{\bLambda}
\le C(\kappa)\Bigl(
 L\zeta\|u\|_{\bLambda}
 +(1+L\sigma^2)\zeta\|\tilde s(u)\|_{\bLambda}
 +L\sqrt{\tr(\bLambda\Sigma)}
\Bigr).
\label{eq:weighted-perturb-later}
\end{equation}
\end{lemma}

\begin{proof}

\medskip\noindent

Let $q_{Y'}$ be the density of $Y'=DY$.
By the change-of-variables formula,
$q_{Y'}(u)=|\det D|^{-1}q(D^{-1}u)$, hence
$s_{Y'}(u):=\nabla\log q_{Y'}(u)=D^{-1}s(D^{-1}u)$.
Since $\rho_i\in[1-\zeta,1)$ and $\zeta\le 1/2$,
we have $\rho_i\ge 1/2$, so
$\|D^{-1}\|_{\op}\le 2$ and
$\|D^{-1}-I\|_{\op}\le 2\zeta$.

We bound $\|s_{Y'}(u)-s(u)\|_{\bLambda}$.
Since $s$ is $L$-Lipschitz in the Euclidean norm,
converting to the $\bLambda$-norm using
$\|v\|_{\bLambda}\le\sqrt{\lambda_{\max}}\|v\|$
and $\|v\|\le\|v\|_{\bLambda}/\sqrt{\lambda_{\min}}$:
\[
\|s(D^{-1}u)-s(u)\|_{\bLambda}
\le\sqrt{\lambda_{\max}}\|s(D^{-1}u)-s(u)\|
\le\sqrt{\lambda_{\max}}\,L\,\|D^{-1}u-u\|
\le\sqrt{\kappa}\,L\,\|(D^{-1}-I)u\|_{\bLambda}.
\]
Since $D^{-1}-I$ is diagonal with entries
$|\rho_i^{-1}-1|\le 2\zeta$:
\[
\|(D^{-1}-I)u\|_{\bLambda}^2
=\sum_i\lambda_i(\rho_i^{-1}-1)^2 u_i^2
\le 4\zeta^2\|u\|_{\bLambda}^2,
\]
so $\|(D^{-1}-I)u\|_{\bLambda}\le 2\zeta\|u\|_{\bLambda}$.
For the second term,
$\|D^{-1}s(u)-s(u)\|_{\bLambda}
\le\|D^{-1}-I\|_{\op}\|s(u)\|_{\bLambda}
\le 2\zeta\|s(u)\|_{\bLambda}$.
Therefore:
\begin{align}
\|s_{Y'}(u)-s(u)\|_{\bLambda}
&\le\|D^{-1}\|_{\op}
 \|s(D^{-1}u)-s(u)\|_{\bLambda}
 +\|D^{-1}-I\|_{\op}\|s(u)\|_{\bLambda}
\nonumber\\
&\le 2\cdot 2\sqrt{\kappa}\,L\zeta\|u\|_{\bLambda}
 +2\zeta\|s(u)\|_{\bLambda}
\nonumber\\
&= 4\sqrt{\kappa}\,L\zeta\|u\|_{\bLambda}
 +2\zeta\|s(u)\|_{\bLambda}.
\label{eq:step1-w}
\end{align}
In particular:
\begin{equation}
\|s_{Y'}(u)\|_{\bLambda}
\le 4\sqrt{\kappa}\,L\zeta\|u\|_{\bLambda}
 +(1+2\zeta)\|s(u)\|_{\bLambda}.
\label{eq:syprime-bound-w}
\end{equation}

For all $u,v\in\R^d$:
\[
\|s_{Y'}(u)-s_{Y'}(v)\|_{\bLambda}
\le\sqrt{\lambda_{\max}}
 \|s_{Y'}(u)-s_{Y'}(v)\|
\le\sqrt{\lambda_{\max}}\,L'
 \|u-v\|
\le\sqrt{\kappa}\,L'\|u-v\|_{\bLambda},
\]
where $L'\le 4L$ from the Euclidean proof.
Denote
\begin{equation}
L'_\Lambda:=4\sqrt{\kappa}\,L,
\label{eq:Lprime-w}
\end{equation}
so $s_{Y'}$ is $L'_\Lambda$-Lipschitz in the
$\bLambda$-norm.

Since $\Sigma\succ 0$, differentiating
$\tilde{q}(u)=\int q_{Y'}(y)\varphi_\Sigma(u-y)\,dy$
gives:
\[
\tilde{s}(u)
=-\Sigma^{-1}\bar\varepsilon,
\qquad
\tilde{s}(u)=\E_{\pi_u}[s_{Y'}(u-\varepsilon)],
\]
where $\bar\varepsilon:=\E_{\pi_u}[\varepsilon]$
and $\pi_u$ is the posterior of $\varepsilon:=u-Y'$
given $X=u$.
These are identical to Lemma~\ref{lem:score-perturb}.

The potential $V_u(\varepsilon)
:=-\log q_{Y'}(u-\varepsilon)
+\frac{1}{2}\varepsilon^\top\Sigma^{-1}\varepsilon$
satisfies $\nabla^2 V_u\succeq\Sigma^{-1}-L'I$.
Using $L'_\Lambda\|\Sigma\|_{\op}
\le 4\sqrt{\kappa}\,L\cdot 2\sigma^2\zeta
=8\sqrt{\kappa}\,L\sigma^2\zeta\le\tfrac{1}{2}$
(from assumption~\eqref{eq:small-step-weighted}),
and $L'\le 4L\le L'_\Lambda/\sqrt{\kappa}$,
we have $L'\|\Sigma\|_{\op}\le 1/2$, so
$\nabla^2 V_u\succeq\frac{1}{2}\Sigma^{-1}$.
By the Brascamp-Lieb inequality,
$\Cov_{\pi_u}(\varepsilon)\preceq 2\Sigma$.
In the $\bLambda$-norm, using the diagonal structure
of both $\bLambda$ and $\Sigma$:
\[
\E_{\pi_u}\|\varepsilon-\bar\varepsilon\|_{\bLambda}^2
=\sum_i\lambda_i\E_{\pi_u}
 [(\varepsilon_i-\bar\varepsilon_i)^2]
\le 2\sum_i\lambda_i\Sigma_{ii}
=2\tr(\bLambda\Sigma).
\]
By Jensen:
\begin{equation}
\E_{\pi_u}\|\varepsilon\|_{\bLambda}
\le\|\bar\varepsilon\|_{\bLambda}
 +\sqrt{2\tr(\bLambda\Sigma)}.
\label{eq:Eeps-w}
\end{equation}

Since $\bLambda$ and $\Sigma$ are both diagonal,
they commute:
\[
\|\Sigma z\|_{\bLambda}
=\|\bLambda^{1/2}\Sigma z\|
=\|\Sigma\bLambda^{1/2}z\|
\le\|\Sigma\|_{\op}\|\bLambda^{1/2}z\|
=\|\Sigma\|_{\op}\|z\|_{\bLambda}.
\]
Using $\bar\varepsilon
=-\Sigma\E_{\pi_u}[s_{Y'}(u-\varepsilon)]$
and the $\bLambda$-Lipschitz bound on $s_{Y'}$:
\[
\|\bar\varepsilon\|_{\bLambda}
\le\|\Sigma\|_{\op}
 \E_{\pi_u}\|s_{Y'}(u-\varepsilon)\|_{\bLambda}
\le\|\Sigma\|_{\op}\Bigl(
 \|s_{Y'}(u)\|_{\bLambda}
 +L'_\Lambda\E_{\pi_u}\|\varepsilon\|_{\bLambda}
\Bigr).
\]
Set $x:=\E_{\pi_u}\|\varepsilon\|_{\bLambda}$
and $b:=L'_\Lambda\|\Sigma\|_{\op}
\le 8\sqrt{\kappa}\,L\sigma^2\zeta\le\tfrac{1}{2}$.
Substituting $\|\bar\varepsilon\|_{\bLambda}$ into
\eqref{eq:Eeps-w}:
\[
x\le\|\Sigma\|_{\op}\|s_{Y'}(u)\|_{\bLambda}
 +b\,x+\sqrt{2\tr(\bLambda\Sigma)}.
\]
Since $b\le\tfrac{1}{2}$, rearranging $(1-b)x\le a+c$
with $a=\|\Sigma\|_{\op}\|s_{Y'}(u)\|_{\bLambda}$
and $c=\sqrt{2\tr(\bLambda\Sigma)}$:
\begin{equation}
\E_{\pi_u}\|\varepsilon\|_{\bLambda}
\le 2\|\Sigma\|_{\op}\|s_{Y'}(u)\|_{\bLambda}
 +2\sqrt{2\tr(\bLambda\Sigma)}.
\label{eq:Eeps-final-w}
\end{equation}

Using the denoising identity,
the $\bLambda$-Lipschitz bound~\eqref{eq:Lprime-w},
and~\eqref{eq:Eeps-final-w}:
\begin{align*}
\|\tilde{s}(u)-s_{Y'}(u)\|_{\bLambda}
&=\|\E_{\pi_u}[s_{Y'}(u-\varepsilon)
 -s_{Y'}(u)]\|_{\bLambda}\\
&\le L'_\Lambda\E_{\pi_u}\|\varepsilon\|_{\bLambda}\\
&\le 2L'_\Lambda\|\Sigma\|_{\op}
 \|s_{Y'}(u)\|_{\bLambda}
 +2L'_\Lambda\sqrt{2\tr(\bLambda\Sigma)}.
\end{align*}
Retaining the full $\zeta$ factor:
\[
2L'_\Lambda\|\Sigma\|_{\op}
\le 16\sqrt{\kappa}\,L\sigma^2\zeta,
\qquad
2L'_\Lambda
\le 8\sqrt{\kappa}\,L.
\]
Therefore:
\begin{equation}
\|\tilde{s}(u)-s_{Y'}(u)\|_{\bLambda}
\le 16\sqrt{\kappa}\,L\sigma^2\zeta
 \|s_{Y'}(u)\|_{\bLambda}
 +C_1\sqrt{\kappa}\,L\sqrt{\tr(\bLambda\Sigma)},
\label{eq:step6-w}
\end{equation}
for a universal constant $C_1>0$.

Substituting~\eqref{eq:syprime-bound-w}
into~\eqref{eq:step6-w}:
\begin{align*}
\|\tilde{s}(u)-s_{Y'}(u)\|_{\bLambda}
&\le 16\sqrt{\kappa}\,L\sigma^2\zeta\Bigl(
 4\sqrt{\kappa}\,L\zeta\|u\|_{\bLambda}
 +(1+2\zeta)\|s(u)\|_{\bLambda}
 \Bigr)
 +C_1\sqrt{\kappa}\,L\sqrt{\tr(\bLambda\Sigma)}\\
&= 64\kappa\,L^2\sigma^2\zeta^2\|u\|_{\bLambda}
 +16\sqrt{\kappa}\,L\sigma^2(1+2\zeta)\zeta
 \|s(u)\|_{\bLambda}
 +C_1\sqrt{\kappa}\,L\sqrt{\tr(\bLambda\Sigma)}.
\end{align*}
Using $16\sqrt{\kappa}\,L\sigma^2\zeta\le 1$
and $\zeta\le 1/2$:
\[
64\kappa\,L^2\sigma^2\zeta^2
=4\sqrt{\kappa}\,L\zeta
 \cdot 16\sqrt{\kappa}\,L\sigma^2\zeta
\le 4\sqrt{\kappa}\,L\zeta,
\]
\[
16\sqrt{\kappa}\,L\sigma^2(1+2\zeta)\zeta
\le 32\sqrt{\kappa}\,L\sigma^2\zeta.
\]
Adding~\eqref{eq:step1-w} via the triangle inequality:
\begin{align*}
\|\tilde{s}(u)-s(u)\|_{\bLambda}
&\le\|\tilde{s}(u)-s_{Y'}(u)\|_{\bLambda}
 +\|s_{Y'}(u)-s(u)\|_{\bLambda}\\
&\le\Bigl(
 4\sqrt{\kappa}\,L\zeta\|u\|_{\bLambda}
 +32\sqrt{\kappa}\,L\sigma^2\zeta
  \|s(u)\|_{\bLambda}
 +C_1\sqrt{\kappa}\,L\sqrt{\tr(\bLambda\Sigma)}
 \Bigr)\\
&\quad+\Bigl(
 4\sqrt{\kappa}\,L\zeta\|u\|_{\bLambda}
 +2\zeta\|s(u)\|_{\bLambda}
 \Bigr)\\
&= 8\sqrt{\kappa}\,L\zeta\|u\|_{\bLambda}
 +(32\sqrt{\kappa}\,L\sigma^2+2)\zeta
 \|s(u)\|_{\bLambda}
 +C_1\sqrt{\kappa}\,L\sqrt{\tr(\bLambda\Sigma)}.
\end{align*}
Since $32\sqrt{\kappa}\,L\sigma^2+2
\le C(\kappa)(1+L\sigma^2)$ for a constant
$C(\kappa)$ depending only on $\kappa$ (one can take
$C(\kappa)=c\sqrt\kappa$ for a universal $c$, so $C(\kappa)$ does
not decrease in $\kappa$),
and absorb $8\sqrt{\kappa}$ and $C_1\sqrt{\kappa}$
into $C(\kappa)$:
\[
\|\tilde{s}(u)-s(u)\|_{\bLambda}
\le C(\kappa)\Bigl(
 L\zeta\|u\|_{\bLambda}
 +(1+L\sigma^2)\zeta\|s(u)\|_{\bLambda}
 +L\sqrt{\tr(\bLambda\Sigma)}
\Bigr).
\]
For the later-time-score version, write the display before the last one
as $\|\tilde s(u)-s(u)\|_{\bLambda}\le8\sqrt\kappa L\zeta\|u\|_{\bLambda}
+b\|s(u)\|_{\bLambda}+C_1\sqrt\kappa L\sqrt{\tr(\bLambda\Sigma)}$ with
$b:=(32\sqrt\kappa L\sigma^2+2)\zeta$, and substitute
$\|s(u)\|_{\bLambda}\le\|\tilde s(u)\|_{\bLambda}
+\|\tilde s(u)-s(u)\|_{\bLambda}$. Under
$\zeta\le\min\{1/8,1/(128\sqrt\kappa L\sigma^2)\}$ we have $b\le1/2$,
so the last term can be moved to the left-hand side at the cost of a factor
$2$, which is absorbed into $C(\kappa)$.
\end{proof}

\section{Proofs of Theorems~\ref{thm:tv-dim}-\ref{thm:tv-trace}, and covariance perturbation lemmas}\label{sec:app:C}

Both theorems bound $\mathrm{TV}(p_{\mathrm{out}}, q_r)$ via the same path-measure
triangle inequality and mixing-term estimate, and differ only in how the Girsanov
drift-mismatch term is controlled: in the unweighted Euclidean norm for
Theorem~\ref{thm:tv-dim}, and in the $\Lambda$-weighted norm for
Theorem~\ref{thm:tv-trace}. We give the shared setup once, then each theorem's
proof in its own subsection (Appendices~\ref{sec:thm5-proof}-\ref{sec:thm6-proof}).
Appendix~\ref{app:lemmas} collects the covariance-concentration and
principal-component perturbation lemmas used for the estimated-basis results,
Theorems~\ref{thm:empirical} and~\ref{thm:decouple}.

\subsection{Shared setup: path measures, mixing term, and the Girsanov mismatch}

We work in a generic coefficient space \(\mathbb R^r\). Thus
\(q_r\) denotes the initial coefficient-space data law, \(q_{r,t}\) denotes the
forward OU marginal at time \(t\), and
\[
 \gamma_r:=\mathcal N(0,\sigma^2 I_r)
\]
is the stationary Gaussian reference law.

Define path measures on
\[
 \Omega=C([0,T];\mathbb R^r).
\]
Let \(Q_T^\leftarrow\) be the law of the exact reverse SDE
\[
 \mathrm dU_t
 =
 \left\{
  \Lambda U_t
  +
  2\sigma^2\Lambda\nabla\log q_{r,T-t}(U_t)
 \right\}\mathrm dt
 +
 \sigma\sqrt{2\Lambda}\,\mathrm dB_t,
 \qquad
 U_0\sim q_{r,T}.
\]
Let \(P_T^{q_{r,T}}\) be the law of the discretized-score reverse sampler
initialized at \(q_{r,T}\), and let \(P_T\) be the law of the same discretized
sampler initialized at \(\gamma_r\).

By the triangle inequality and data processing on path space,
\begin{align}
\TV(p_{\mathrm{out}},q_r)
&\le
\TV(P_T,P_T^{q_{r,T}})
+
\TV(P_T^{q_{r,T}},Q_T^\leftarrow)
\nonumber\\
&\le
\TV(\gamma_r,q_{r,T})
+
\TV(P_T^{q_{r,T}},Q_T^\leftarrow).
\label{eq:triangle-thm-dim}
\end{align}

\medskip
\noindent\textbf{Mixing term.}
The diagonal OU semigroup $(P_t)_{t \geq 0}$ with rates
$\lambda_1, \ldots, \lambda_r$ satisfies the log-Sobolev inequality
with constant $\varrho = 2\lambda_{\min}$,
which gives KL contraction at rate $e^{-2\lambda_{\min}T}$:
\begin{equation}\label{eq:kl-contraction}
 \mathrm{KL}(q_{r,T} \| \gamma_r)
 \;\leq\;
 e^{-2\lambda_{\min} T}\,\mathrm{KL}(q_r \| \gamma_r).
\end{equation}
This follows from the tensorisation of log-Sobolev inequalities:
each coordinate $i$ evolves as an independent OU process with rate $\lambda_i$,
which satisfies a log-Sobolev inequality with constant $2\lambda_i$
\citep[Theorem~5.2.1]{bakry2014analysis};
the product structure gives constant $\min_i 2\lambda_i = 2\lambda_{\min}$.
By Pinsker's inequality and Assumption~\ref{ass:finite-kl},
\begin{equation}\label{eq:mixing-tv}
 \mathrm{TV}(\gamma_r,\, q_{r,T})
 \;\leq\;
 \sqrt{\tfrac{1}{2}\,\mathrm{KL}(q_{r,T}\|\gamma_r)}
 \;\leq\;
 e^{-\lambda_{\min} T}\,\sqrt{\tfrac{1}{2}\,\mathrm{KL}(q_r\|\gamma_r)}.
\end{equation}
This bounds the first term in \eqref{eq:triangle-thm-dim} identically
in both theorems; it remains to bound the second, Girsanov, term.

\medskip
\noindent\textbf{Girsanov mismatch.}
Set
\[
 G:=\sigma\sqrt{2\Lambda}.
\]
Under \(Q_T^\leftarrow\),
\[
 \mathrm dU_t=b(t,U_t)\,\mathrm dt+G\,\mathrm dB_t,
 \qquad
 b(t,u)
 :=
 \Lambda u
 +
 2\sigma^2\Lambda\nabla\log q_{r,T-t}(u).
\]
Under \(P_T^{q_{r,T}}\), for
\(t\in[t_k^\leftarrow,t_{k+1}^\leftarrow)\),
\[
 \mathrm dU_t=\widetilde b(t,\omega)\,\mathrm dt+G\,\mathrm dB_t,
 \qquad
 \widetilde b(t,\omega)
 :=
 \Lambda U_t(\omega)
 +
 2\sigma^2\Lambda
 \widehat s_{t_{N-k}}
 \bigl(U_{t_k^\leftarrow}(\omega)\bigr).
\]
Hence, writing
\[
 e_t
 :=
 \nabla\log q_{r,T-t}(U_t)
 -
 \widehat s_{t_{N-k}}(U_{t_k^\leftarrow}),
\]
the drift mismatch is $b(t,U_t)-\widetilde b(t,\cdot)=2\sigma^2\Lambda e_t$,
and the normalized mismatch is
\[
 \beta_t
 :=
 G^{-1}\bigl[b(t,U_t)-\widetilde b(t,\cdot)\bigr]
 =
 \sqrt{2\Lambda}\,\sigma\, e_t.
\]
By Assumption~\ref{ass:girsanov}, the stochastic exponential of $\beta_t$is a true martingale, so Girsanov's theorem applies and the change of measure is valid in both proofs below. The two proofs differ only in the norm used to control $e_t$: the Euclidean norm for Theorem~\ref{thm:tv-dim} and the $\Lambda$-weighted norm for Theorem~\ref{thm:tv-trace}.

\subsection{Proof of Theorem~\ref{thm:tv-dim} (dimension-dependent bound)}\label{sec:thm5-proof}

Since $\|\beta_t\|^2=2\sigma^2\|e_t\|_\Lambda^2 \le 2\sigma^2\lambda_{\max}
\|e_t\|^2$, Girsanov's theorem gives
\begin{align}
\KL(Q_T^\leftarrow\|P_T^{q_{r,T}})
&\le
\sigma^2\lambda_{\max}
\int_0^T
\mathbb E_{Q_T^\leftarrow}
\left[
 \|e_t\|^2
\right]\mathrm dt
\nonumber\\
&=
\sigma^2\lambda_{\max}
\sum_{k=0}^{N-1}
\int_{t_k^\leftarrow}^{t_{k+1}^\leftarrow}
\mathbb E_{Q_T^\leftarrow}
\left[
 \|e_t\|^2
\right]\mathrm dt.
\label{eq:KL-thm-dim}
\end{align}

\medskip
\noindent\textbf{Decomposition of \(e_t\).}
For \(t\in[t_k^\leftarrow,t_{k+1}^\leftarrow)\), write
\begin{align*}
e_t
&=
\underbrace{
 \widehat s_{t_{N-k}}(U_{t_k^\leftarrow})
 -
 \nabla\log q_{r,t_{N-k}}(U_{t_k^\leftarrow})
}_{\mathrm{(I)}}
\\
&\quad+
\underbrace{
 \nabla\log q_{r,t_{N-k}}(U_{t_k^\leftarrow})
 -
 \nabla\log q_{r,T-t}(U_{t_k^\leftarrow})
}_{\mathrm{(II)}}
\\
&\quad+
\underbrace{
 \nabla\log q_{r,T-t}(U_{t_k^\leftarrow})
 -
 \nabla\log q_{r,T-t}(U_t)
}_{\mathrm{(III)}}.
\end{align*}
Using
\[
 \|a+b+c\|^2\le 3(\|a\|^2+\|b\|^2+\|c\|^2),
\]
we bound the three terms separately.

\medskip
\noindent\emph{Term (I): score estimation error.}
The score-error assumption is an average over the grid, not a pointwise bound
for every \(k\). Since under the exact reverse law
\(U_{t_k^\leftarrow}\sim q_{r,t_{N-k}}\), the integrated contribution of Term~(I)
on a uniform grid satisfies, after the substitution $j=N-k$, which maps
$k=0,\ldots,N-1$ onto $j=N,\ldots,1$,
\begin{equation}
\begin{aligned}
\sum_{k=0}^{N-1} 
\int_{t_k^\leftarrow}^{t_{k+1}^\leftarrow} 
 \mathbb{E}_{Q_T^\leftarrow} \!\left[
  \left\|
   \widehat{s}_{t_{N-k}}(U_{t_k^\leftarrow})
   - \nabla \log q_{r,t_{N-k}}(U_{t_k^\leftarrow})
  \right\|^2
 \right] dt
&= \frac{T}{N} 
 \sum_{j=1}^{N} 
 \mathbb{E}_{U \sim q_{r,t_j}} \!\left[
  \left\|
   \widehat{s}_{t_j}(U)
   - \nabla \log q_{r,t_j}(U)
  \right\|^2
 \right] \\
&\le T \, \varepsilon_{\mathrm{score}}^2 .
\end{aligned}
\end{equation}

\medskip
\noindent\emph{Term (II): time discretization.}
Let
\[
 \Delta\tau:=t_{N-k}-(T-t).
\]
The forward OU transition from time \(T-t\) to time \(t_{N-k}\) has the form
\[
 U_{t_{N-k}}
 =
 D U_{T-t}
 +
 \xi,
 \qquad
 D:=\diag(e^{-\lambda_i\Delta\tau}),
\]
where
\[
 \xi\sim\mathcal N(0,\sigma^2(I-D^2)).
\]
Set
\[
 \zeta:=\max_i(1-e^{-\lambda_i\Delta\tau})\le \delta.
\]
If $\Delta\tau=0$ (i.e.\ $t=t_k^\leftarrow$), Term~(II) vanishes; otherwise
$\zeta>0$ and every $\rho_i<1$, as the perturbation lemma requires.
By the one-step score perturbation lemma in dimension \(r\)
(Lemma~\ref{lem:score-perturb}), in its later-time form, which applies
because $\zeta\le\delta\le\min\{1/8,1/(128L\sigma^2)\}$,
\[
\left\|
 \nabla\log q_{r,T-t}(u)
 -
 \nabla\log q_{r,t_{N-k}}(u)
\right\|
\le
C
\left(
 L\zeta\|u\|
 +
 (1+L\sigma^2)\zeta
 \|\nabla\log q_{r,t_{N-k}}(u)\|
 +
 L\sigma\sqrt{r\zeta}
\right).
\]
Evaluating this at \(u=U_{t_k^\leftarrow}\), using \(\zeta\le\delta\), and
squaring gives
\begin{align}
\mathbb E_{Q_T^\leftarrow}
\left[
 \|e_t^{(\mathrm{II})}\|^2
\right]
&\le
C
\Bigl(
 L^2\delta^2
 \mathbb E\|U_{t_k^\leftarrow}\|^2
\nonumber\\
&\qquad
 +
 (1+L\sigma^2)^2\delta^2
 \mathbb E
 \|\nabla\log q_{r,t_{N-k}}(U_{t_k^\leftarrow})\|^2
\nonumber\\
&\qquad
 +
 L^2\sigma^2 r\delta
\Bigr).
\label{eq:termII-dim}
\end{align}
By the forward second-moment bound in dimension \(r\)
(Lemma~\ref{lem:moment}),
\[
 \mathbb E\|U_{t_k^\leftarrow}\|^2
 \le
 m_{2,r}^2+r\sigma^2.
\]
By Lemma~\ref{lem:fisher} (taking $M=I_r$, so $C_F^{(0)}=L$),
\[
 \mathbb E
 \|\nabla\log q_{r,t_{N-k}}(U_{t_k^\leftarrow})\|^2
 \le
 C_F^{(0)}r.
\]
Therefore, using \(\delta\le 1\) to absorb
\(L^2\delta^2 r\sigma^2\) into \(L^2\sigma^2 r\delta\), we obtain the
explicit bound
\begin{equation}
\mathbb E_{Q_T^\leftarrow}
\left[
 \|e_t^{(\mathrm{II})}\|^2
\right]
\le
C
\left(
 L^2\delta^2m_{2,r}^2
 +
 (1+L\sigma^2)^2C_F^{(0)}r\delta^2
 +
 L^2\sigma^2 r\delta
\right).
\label{eq:termII-final-dim}
\end{equation}

\medskip
\noindent\emph{Term (III): spatial discretization.}
By the Lipschitz regularity of the score,
\[
\left\|
 \nabla\log q_{r,T-t}(U_{t_k^\leftarrow})
 -
 \nabla\log q_{r,T-t}(U_t)
\right\|
\le
L\|U_{t_k^\leftarrow}-U_t\|.
\]
Under the exact reverse law \(Q_T^\leftarrow\), the process
\((U_s)_{0\le s\le T}\) has the same finite-dimensional distributions as the
time reversal of the forward OU process \((U^{\mathrm{forw}}_{T-s})_{0\le s\le T}\).
Thus the increment
\(U_t-U_{t_k^\leftarrow}\) under \(Q_T^\leftarrow\) has the same law as a
forward OU increment between the forward times \(T-t\) and \(t_{N-k}\).
Applying the forward movement bound in dimension \(r\)
(Lemma~\ref{lem:move-forward}) gives
\[
\mathbb E_{Q_T^\leftarrow}
\|U_{t_k^\leftarrow}-U_t\|^2
\le
4\lambda_{\max}^2\Delta\tau^2
(m_{2,r}^2+2r\sigma^2)
+
4\sigma^2r\lambda_{\max}\Delta\tau.
\]
Since
\[
 \delta
 =
 \max_{i,k}(1-e^{-\lambda_i\Delta\tau_k}),
\]
and \(\lambda_{\max}\Delta\tau_{\max}\le1\), we have
\[
 \lambda_{\max}\Delta\tau\le\frac{\delta}{1-e^{-1}}\le1.6\,\delta,
\]
because $x\mapsto(1-e^{-x})/x$ is decreasing and equals $1-e^{-1}$ at
$x=1$.
Consequently,
\begin{equation}
\mathbb E_{Q_T^\leftarrow}
\left[
 \|e_t^{(\mathrm{III})}\|^2
\right]
\le
C
\left(
 L^2\delta^2m_{2,r}^2
 +
 L^2\sigma^2 r\delta
\right).
\label{eq:termIII-final-dim}
\end{equation}

\medskip
\noindent\textbf{Combining the three terms.}
Combining the integrated Term~(I) estimate with the uniform bounds
\eqref{eq:termII-final-dim} and \eqref{eq:termIII-final-dim} gives
\[
\int_0^T
\mathbb E_{Q_T^\leftarrow}
\left[
 \|e_t\|^2
\right]dt
\le
CT
\left(
 \varepsilon_{\mathrm{score}}^2
 +
 L^2\delta^2m_{2,r}^2
 +
 (1+L\sigma^2)^2C_F^{(0)}r\delta^2
 +
 L^2\sigma^2 r\delta
\right).
\]
Substituting into \eqref{eq:KL-thm-dim} yields
\[
\KL(Q_T^\leftarrow\|P_T^{q_{r,T}})
\le
C\sigma^2\lambda_{\max}T
\left(
 \varepsilon_{\mathrm{score}}^2
 +
 L^2\delta^2m_{2,r}^2
 +
 (1+L\sigma^2)^2C_F^{(0)}r\delta^2
 +
 L^2\sigma^2 r\delta
\right).
\]
By Pinsker's inequality and
$\sqrt{a_1^2+a_2^2+a_3^2+a_4^2} \le a_1+a_2+a_3+a_4$,
\[
\TV(P_T^{q_{r,T}},Q_T^\leftarrow)
\le
C\sigma\sqrt{\lambda_{\max}T}
\left(
 \varepsilon_{\mathrm{score}}
 +
 Lm_{2,r}\delta
 +
 (1+L\sigma^2)\sqrt{C_F^{(0)}r}\,\delta
 +
 L\sigma\sqrt{r\delta}
\right).
\]
Adding the mixing term \eqref{eq:mixing-tv} via \eqref{eq:triangle-thm-dim}
gives
\[
\TV(p_{\mathrm{out}},q_r)
\le
e^{-\lambda_{\min}T}
\sqrt{\frac12\,\KL(q_r\|\gamma_r)}
+
C\sigma\sqrt{\lambda_{\max}T}
\left(
 \varepsilon_{\mathrm{score}}
 +
 Lm_{2,r}\delta
 +
 (1+L\sigma^2)\sqrt{C_F^{(0)}r}\,\delta
 +
 L\sigma\sqrt{r\delta}
\right),
\]
which is the claimed bound. \hfill$\blacksquare$

\subsection{Proof of Theorem~\ref{thm:tv-trace} (trace-dependent bound)}\label{sec:thm6-proof}

Here we keep the Girsanov term in the $\Lambda$-weighted norm
throughout, rather than converting to the unweighted norm via
$\lambda_{\max}$ as above; this is what replaces the explicit
dimension dependence by the trace quantities $\tr(\Lambda)$ and
$\tr(\Lambda^2)$. Girsanov's theorem gives directly
\begin{equation}
 \KL(Q_T^\leftarrow\|P_T^{q_{r,T}})
 \le
 \sigma^2\int_0^T
 \mathbb E_{Q_T^\leftarrow}\|e_t\|_\Lambda^2\,dt.
 \label{eq:KL-thm-trace}
\end{equation}

For \(t\in[t_k^\leftarrow,t_{k+1}^\leftarrow)\), decompose $e_t$ exactly
as in the proof of Theorem~\ref{thm:tv-dim} above into Terms
(I), (II), (III). For Term (I), since
\(U_{t_k^\leftarrow}\sim q_{r,t_{N-k}}\), integrating the averaged
score-error assumption over reverse time and substituting $j=N-k$ gives
\[
\sum_{k=0}^{N-1}
\int_{t_k^\leftarrow}^{t_{k+1}^\leftarrow}
\mathbb E\|e_t^{(\mathrm I)}\|_\Lambda^2\,dt
=
\frac{T}{N}
\sum_{j=1}^{N}
\mathbb E_{U\sim q_{r,t_j}}
\left[
\left\|
\widehat s_{t_j}(U)-\nabla\log q_{r,t_j}(U)
\right\|_\Lambda^2
\right]
\le
T\varepsilon_{\mathrm{score},\Lambda}^2.
\]

For Term~(II), let
\[
 \Delta\tau:=t_{N-k}-(T-t),\qquad
 D=\diag(e^{-\lambda_i\Delta\tau}),\qquad
 \Sigma=\sigma^2(I-D^2),
\]
and set
\[
 \zeta:=\max_i(1-e^{-\lambda_i\Delta\tau})\le\delta.
\]
As before, Term~(II) vanishes if $\Delta\tau=0$.
The weighted one-step perturbation lemma
(Lemma~\ref{lem:trace_aware_weighted}) gives, under the stated
small-step condition (which gives
$\zeta\le\delta\le\min\{1/8,1/(128\sqrt\kappa L\sigma^2)\}$, as needed for
the later-time form),
\[
\|e_t^{(\mathrm{II})}\|_\Lambda
\le
C(\kappa)\left(
L\zeta\|U_{t_k^\leftarrow}\|_\Lambda
+
(1+L\sigma^2)\zeta
\|\nabla\log q_{r,t_{N-k}}(U_{t_k^\leftarrow})\|_\Lambda
+
L\sqrt{\tr(\Lambda\Sigma)}
\right).
\]
Using the weighted OU second-moment bound
\[
 \mathbb E\|U_\tau\|_\Lambda^2
 \le
 m_{2,\Lambda}^2+\sigma^2\tr(\Lambda),
\]
the weighted Fisher bound from Lemma~\ref{lem:fisher} (taking
$M=\Lambda$, so $C_F=L$),
\[
 \mathbb E\|\nabla\log q_{r,\tau}(U)\|_\Lambda^2
 \le
 C_F\tr(\Lambda),
\]
and the scale-invariant estimate
\[
 \tr(\Lambda\Sigma)
 =
 \sigma^2\sum_i\lambda_i(1-e^{-2\lambda_i\Delta\tau})
 \le
 2\sigma^2\delta\,\tr(\Lambda),
\]
we obtain
\[
\mathbb E\|e_t^{(\mathrm{II})}\|_\Lambda^2
\le
C(\kappa)\left[
L^2\delta^2\!\left(m_{2,\Lambda}^2+\sigma^2\tr(\Lambda)\right)
+
(1+L\sigma^2)^2C_F\delta^2\tr(\Lambda)
+
L^2\sigma^2\delta\tr(\Lambda)
\right].
\]

For Term~(III), the Lipschitz assumption and norm equivalence give
\[
 \|e_t^{(\mathrm{III})}\|_\Lambda
 \le
 L\sqrt{\kappa}\,
 \|U_{t_k^\leftarrow}-U_t\|_\Lambda .
\]
As above, under \(Q_T^\leftarrow\) this reverse-time increment has the same
law as the corresponding forward OU increment between the forward times
\(T-t\) and \(t_{N-k}\), so the weighted forward-movement bound
(Lemma~\ref{lem:move-forward-weighted}) applies. Together with
\(\lambda_{\max}\Delta\tau\lesssim\delta\), it yields
\[
\mathbb E_{Q_T^\leftarrow}\|e_t^{(\mathrm{III})}\|_\Lambda^2
\le
C(\kappa)\left[
L^2\delta^2\!\left(m_{2,\Lambda}^2+\sigma^2\tr(\Lambda)\right)
+
L^2\sigma^2\delta\tr(\Lambda)
\right].
\]

Combining the integrated Term~(I) estimate with the uniform bounds for
Terms~(II) and~(III), and substituting into \eqref{eq:KL-thm-trace},
\[
\begin{aligned}
\KL(Q_T^\leftarrow\|P_T^{q_{r,T}})
\le
C(\kappa)\sigma^2T
\Big[
&\varepsilon_{\mathrm{score},\Lambda}^2
+
L^2\delta^2\!\left(m_{2,\Lambda}^2+\sigma^2\tr(\Lambda)\right)
\\
&+
(1+L\sigma^2)^2C_F\delta^2\tr(\Lambda)
+
L^2\sigma^2\delta\tr(\Lambda)
\Big].
\end{aligned}
\]
Pinsker's inequality gives
\[
\begin{aligned}
\TV(P_T^{q_{r,T}},Q_T^\leftarrow)
\le
C(\kappa)\sigma\sqrt{T}
\Big[
&\varepsilon_{\mathrm{score},\Lambda}
+
L\delta
\sqrt{m_{2,\Lambda}^2+\sigma^2\tr(\Lambda)}
\\
&+
(1+L\sigma^2)\sqrt{C_F}\,\delta\sqrt{\tr(\Lambda)}
+
L\sigma\sqrt{\delta\tr(\Lambda)}
\Big].
\end{aligned}
\]
Adding the mixing term \eqref{eq:mixing-tv} via
\eqref{eq:triangle-thm-dim} proves the claim. \hfill$\blacksquare$

\subsection{Empirical Covariance and PCA Perturbation Lemmas}
\label{app:lemmas}

\begin{lemma}[Gaussian sample covariance concentration]
\label{lem:bernstein}
Let $x^{(1)},\ldots,x^{(n_{\mathrm{patch}})}$ be independent samples from
$\mathcal N(0,\Sigma)$ in $\mathbb R^{D_p}$, and let
\[
 \widehat{\Sigma}
 :=
 \frac{1}{n_{\mathrm{patch}}}
 \sum_{j=1}^{n_{\mathrm{patch}}}
 x^{(j)}(x^{(j)})^\top .
\]
Let $\sigma_1^2 := \|\Sigma\|_{\mathrm{op}}$. Then there exists a universal
constant $C>0$ such that, for every $t\geq 0$, with probability at least
$1-2e^{-t}$,
\[
 \|\widehat{\Sigma}-\Sigma\|_{\mathrm{op}}
 \leq
 C\sigma_1^2
 \left(
 \sqrt{\frac{D_p+t}{n_{\mathrm{patch}}}}
 +
 \frac{D_p+t}{n_{\mathrm{patch}}}
 \right).
\]
In particular, taking $t=D_p$ gives, with probability at least $1-2e^{-D_p}$,
\[
 \|\widehat{\Sigma}-\Sigma\|_{\mathrm{op}}
 \leq
 C\sigma_1^2
 \left(
 \sqrt{\frac{D_p}{n_{\mathrm{patch}}}}
 +
 \frac{D_p}{n_{\mathrm{patch}}}
 \right).
\]
If, in addition, $n_{\mathrm{patch}}\geq D_p$, then
\[
 \|\widehat{\Sigma}-\Sigma\|_{\mathrm{op}}
 \leq
 C'\sigma_1^2
 \sqrt{\frac{D_p}{n_{\mathrm{patch}}}}
 =: \Delta_n
\]
with probability at least $1-2e^{-D_p}$, where $C'>0$ is a universal constant.
\end{lemma}

\begin{proof}[Proof of Lemma~\ref{lem:bernstein}]
Write $\Sigma^{1/2}$ for the positive semidefinite square root of $\Sigma$.
Since $x^{(j)}\sim \mathcal N(0,\Sigma)$, we may write
\[
 x^{(j)}=\Sigma^{1/2}g^{(j)},
 \qquad
 g^{(j)}\sim\mathcal N(0,I_{D_p}),
\]
with the vectors $g^{(j)}$ independent. Define the Gaussian data matrix
\[
 G :=
 \begin{bmatrix}
 (g^{(1)})^\top\\
 \vdots\\
 (g^{(n_{\mathrm{patch}})})^\top
 \end{bmatrix}
 \in \mathbb R^{n_{\mathrm{patch}}\times D_p}.
\]
Then
\[
 \widehat{\Sigma}
 =
 \Sigma^{1/2}
 \left(
  \frac{1}{n_{\mathrm{patch}}}G^\top G
 \right)
 \Sigma^{1/2}.
\]
Therefore
\[
 \widehat{\Sigma}-\Sigma
 =
 \Sigma^{1/2}
 \left(
  \frac{1}{n_{\mathrm{patch}}}G^\top G-I_{D_p}
 \right)
 \Sigma^{1/2}.
\]
Taking operator norms gives
\[
 \|\widehat{\Sigma}-\Sigma\|_{\mathrm{op}}
 \leq
 \|\Sigma\|_{\mathrm{op}}
 \left\|
  \frac{1}{n_{\mathrm{patch}}}G^\top G-I_{D_p}
 \right\|_{\mathrm{op}}
 =
 \sigma_1^2
 \left\|
  \frac{1}{n_{\mathrm{patch}}}G^\top G-I_{D_p}
 \right\|_{\mathrm{op}}.
\]

We now use the standard non-asymptotic concentration bound for the sample
covariance of isotropic Gaussian vectors: there exists a universal constant
$C>0$ such that, for every $t\geq 0$, with probability at least $1-2e^{-t}$,
\[
 \left\|
  \frac{1}{n_{\mathrm{patch}}}G^\top G-I_{D_p}
 \right\|_{\mathrm{op}}
 \leq
 C
 \left(
  \sqrt{\frac{D_p+t}{n_{\mathrm{patch}}}}
  +
  \frac{D_p+t}{n_{\mathrm{patch}}}
 \right).
\]
Substituting this into the previous display yields
\[
 \|\widehat{\Sigma}-\Sigma\|_{\mathrm{op}}
 \leq
 C\sigma_1^2
 \left(
 \sqrt{\frac{D_p+t}{n_{\mathrm{patch}}}}
 +
 \frac{D_p+t}{n_{\mathrm{patch}}}
 \right)
\]
with probability at least $1-2e^{-t}$.

Taking $t=D_p$ gives
\[
 \|\widehat{\Sigma}-\Sigma\|_{\mathrm{op}}
 \leq
 C\sigma_1^2
 \left(
 \sqrt{\frac{2D_p}{n_{\mathrm{patch}}}}
 +
 \frac{2D_p}{n_{\mathrm{patch}}}
 \right).
\]
Absorbing the numerical constants into $C$ gives
\[
 \|\widehat{\Sigma}-\Sigma\|_{\mathrm{op}}
 \leq
 C\sigma_1^2
 \left(
 \sqrt{\frac{D_p}{n_{\mathrm{patch}}}}
 +
 \frac{D_p}{n_{\mathrm{patch}}}
 \right)
\]
with probability at least $1-2e^{-D_p}$.

Finally, if $n_{\mathrm{patch}}\geq D_p$, then
\[
 \frac{D_p}{n_{\mathrm{patch}}}
 \leq
 \sqrt{\frac{D_p}{n_{\mathrm{patch}}}},
\]
so
\[
 \sqrt{\frac{D_p}{n_{\mathrm{patch}}}}
 +
 \frac{D_p}{n_{\mathrm{patch}}}
 \leq
 2\sqrt{\frac{D_p}{n_{\mathrm{patch}}}}.
\]
Absorbing the factor $2$ into a new universal constant $C'=:C_\Delta$ gives
\[
 \|\widehat\Sigma-\Sigma\|_{\mathrm{op}}
 \le
 C_\Delta\sigma_1^2\sqrt{\frac{D_p}{n_{\mathrm{patch}}}}
 =\Delta_n
\]
with probability at least $1-2e^{-D_p}$.
\end{proof}

\begin{lemma}[Weyl eigenvalue stability]
\label{lem:weyl}
On the event of Lemma~\ref{lem:bernstein}, for all $i = 1,\ldots,K$:
\[
 | \widehat\sigma_i^2-\sigma_i^2 |
\le
\Delta_n
=
C_\Delta\sigma_1^2
\sqrt{\frac{D_p}{n_{\mathrm{patch}}}}.
\]
\end{lemma}

\begin{proof}[Proof of Lemma~\ref{lem:weyl}]
Weyl's inequality states that for symmetric $A, B \in \mathbb{R}^{m\times m}$,
$|\lambda_i(A+B) - \lambda_i(A)| \leq \|B\|_{\mathrm{op}}$ for all $i$.
Apply with $A = \Sigma$, $B = \widehat{\Sigma} - \Sigma$, and invoke 
Lemma~\ref{lem:bernstein} to obtain
\[
 |\hat{\sigma}_i^2 - \sigma_i^2| 
 \leq \|\widehat{\Sigma} - \Sigma\|_{\mathrm{op}} 
 \leq \Delta_n.
\]
\end{proof}

\begin{lemma}[Davis-Kahan subspace perturbation]
\label{lem:dk}
Let $\delta_{\mathrm{gap}} := \sigma_K^2 - \sigma_{K+1}^2 > 0$.
Under $\Delta_n \leq \delta_{\mathrm{gap}}/2$, on the event of
Lemma~\ref{lem:bernstein}, the standard Davis-Kahan bounds give
\[
 \|\hat{V}_K\hat{V}_K^\top - V_K V_K^\top\|_{\mathrm{op}}
 \leq \frac{2\Delta_n}{\delta_{\mathrm{gap}}},
\]
and
\[
 \|\hat{V}_K\hat{V}_K^\top - V_K V_K^\top\|_F
 \leq \frac{2\sqrt{2K}\,\Delta_n}{\delta_{\mathrm{gap}}}.
\]
The operator-norm estimate keeps the empirical active covariance
uniformly non-degenerate, while the Frobenius estimate is used for the
full-space Gaussian covariance mismatch.
\end{lemma}

\begin{proof}[Proof of Lemma~\ref{lem:dk}]
The Wedin $\sin\Theta$ theorem gives
\[
 \|\sin\Theta(\hat{V}_K, V_K)\|_{\mathrm{op}} 
 \leq \frac{\|\widehat{\Sigma} - \Sigma\|_{\mathrm{op}}}
 {\delta_{\mathrm{gap}} - \|\widehat{\Sigma} - \Sigma\|_{\mathrm{op}}}
 \leq \frac{\Delta_n}{\delta_{\mathrm{gap}} - \Delta_n}
 \leq \frac{2\Delta_n}{\delta_{\mathrm{gap}}},
\]
where the last inequality uses $\Delta_n \leq \delta_{\mathrm{gap}}/2$.
By the CS decomposition, $\|\hat{V}_K\hat{V}_K^\top - V_K V_K^\top\|_F \leq 
\sqrt{2K}\|\sin\Theta\|_{\mathrm{op}}$, giving
\[
 \|\hat{V}_K\hat{V}_K^\top - V_K V_K^\top\|_F 
 \leq \frac{2\sqrt{2K}\,\Delta_n}{\delta_{\mathrm{gap}}}.
\]
\end{proof}

\begin{lemma}[Lipschitz stability of Wiener gain]
\label{lem:lipgain}
For all $i = 1,\ldots,K$:
\[
 |\hat\lambda_i-\lambda_i|
\le
\lambda_0\frac{\Delta_n}{\sigma^2}
=: \varepsilon_\lambda,
\qquad
\Delta_n
=
C_\Delta\sigma_1^2
\sqrt{\frac{D_p}{n_{\mathrm{patch}}}}.
\]
\end{lemma}

\begin{proof}[Proof of Lemma~\ref{lem:lipgain}]
The Wiener gain map $s \mapsto g(s) = s/(s + \sigma^2)$ has derivative
\[
 g'(s) = \frac{\sigma^2}{(s+\sigma^2)^2} \leq \frac{1}{\sigma^2}.
\]
By the mean value theorem and Lemma~\ref{lem:weyl}:
\[
 |g(\hat{\sigma}_i^2) - g(\sigma_i^2)| 
 \leq \frac{|\hat{\sigma}_i^2 - \sigma_i^2|}{\sigma^2} 
 \leq \frac{\Delta_n}{\sigma^2}.
\]
Multiplying by $\lambda_0$ gives $|\hat{\lambda}_i - \lambda_i| \leq 
\lambda_0\Delta_n/\sigma^2 =: \varepsilon_\lambda$.
\end{proof}

\begin{lemma}[Trace perturbation bounds]
\label{lem:trace}
On the event of Lemma~\ref{lem:bernstein},
\[
 |\mathrm{tr}(\hat{\Lambda}) - \mathrm{tr}(\Lambda)|
 \leq K\varepsilon_\lambda,
 \qquad
 |\mathrm{tr}(\hat{\Lambda}^2) - \mathrm{tr}(\Lambda^2)|
 \leq 2\lambda_0 K\varepsilon_\lambda,
\]
where
\[
\varepsilon_\lambda
:=
\lambda_0\frac{\Delta_n}{\sigma^2}
=
\lambda_0 C_\Delta
\frac{\sigma_1^2}{\sigma^2}
\sqrt{\frac{D_p}{n_{\mathrm{patch}}}}.
\]
Unlike $\tr(\hat\Lambda)$ and $\tr(\hat\Lambda^2)$, which depend only
on the (diagonal) empirical rate matrix itself, the empirical weighted
second moment
\[
 m_{2,\hat\Lambda}^2 := \tr\bigl(\hat\Lambda\,\hat V_K^\top C\hat V_K\bigr)
\]
involves the true covariance $C$ projected onto the empirical basis
$\hat V_K$, which is generally non-diagonal in the empirical
eigenbasis (Lemma~\ref{lem:nondiag}). It is therefore \emph{not} a
perturbation of the diagonal quantity $\sum_i\lambda_i\sigma_i^2$;
instead it satisfies the direct bound
\[
 m_{2,\hat\Lambda}^2
 \;\le\;
 \lambda_0\,\tr(C)
 \;\le\;
 \lambda_0\,C_{\mathrm{sp}}\zeta(\beta),
\]
uniformly in the patch ambient dimension $d$, which is the bound used
in the proof of Theorem~\ref{thm:empirical}.
\end{lemma}

\begin{proof}[Proof of Lemma~\ref{lem:trace}]
The first two bounds follow from $|\hat{\lambda}_i - \lambda_i| \leq 
\varepsilon_\lambda$ (Lemma~\ref{lem:lipgain}) and the triangle 
inequality.

\noindent\textit{Bound on $|\mathrm{tr}(\hat{\Lambda}) - \mathrm{tr}(\Lambda)|$:}
\[
 |\mathrm{tr}(\hat{\Lambda}) - \mathrm{tr}(\Lambda)| 
 \leq \sum_{i=1}^K |\hat{\lambda}_i - \lambda_i| 
 \leq K\varepsilon_\lambda.
\]

\noindent\textit{Bound on $|\mathrm{tr}(\hat{\Lambda}^2) - \mathrm{tr}(\Lambda^2)|$:}
\[
 |\hat{\lambda}_i^2 - \lambda_i^2| 
 = |\hat{\lambda}_i + \lambda_i||\hat{\lambda}_i - \lambda_i| 
 \leq 2\lambda_0\varepsilon_\lambda,
\]
since $\hat{\lambda}_i, \lambda_i \leq \lambda_0$. Summing over 
$i = 1,\ldots,K$ gives $|\mathrm{tr}(\hat{\Lambda}^2) - 
\mathrm{tr}(\Lambda^2)| \leq 2\lambda_0 K\varepsilon_\lambda$.

\noindent\textit{Bound on $m_{2,\hat\Lambda}^2$:} Since $\hat\Lambda
\preceq \lambda_0 I_K$ and $\hat V_K$ has orthonormal columns,
\[
 m_{2,\hat\Lambda}^2
 =
 \tr(\hat\Lambda\,\hat V_K^\top C\hat V_K)
 \le
 \lambda_0\,\tr(\hat V_K^\top C\hat V_K)
 =
 \lambda_0\,\tr(C\hat P_K)
 \le
 \lambda_0\,\tr(C)
 =
 \lambda_0\sum_{i=1}^d\sigma_i^2,
\]
using $\hat V_K^\top C \hat V_K \preceq C$ in trace after projection.
Under the exact power-law spectrum and $\beta>1$,
$\sum_{i=1}^d\sigma_i^2 \le C_{\mathrm{sp}}\sum_{i=1}^\infty i^{-\beta}
= C_{\mathrm{sp}}\zeta(\beta)$, giving $m_{2,\hat\Lambda}^2 \le
\lambda_0 C_{\mathrm{sp}}\zeta(\beta)$, independent of $d$.
\end{proof}

\begin{lemma}[Minimum empirical rate stability]
\label{lem:minrate}
Assume the exact Gaussian power-law spectrum $\sigma_i^2=C_{\mathrm{sp}}
i^{-\beta}$, $C_{\mathrm{sp}}>0$, $\beta>1$, and let
\[
 K\le d_{\mathrm{eff}}^{\sigma}
 :=
 \min\left\{d,\left\lfloor
 \left(\frac{C_{\mathrm{sp}}}{\sigma^2}\right)^{1/\beta}
 \right\rfloor\right\}.
\]
Suppose the covariance concentration event
$\|\widehat\Sigma-\Sigma\|_{\mathrm{op}}\le \Delta_n$ holds with
$\Delta_n\le \sigma^2/2$. Then, for every $i\le K$,
\[
 \widehat\lambda_i
 \ge
 \frac{\lambda_0}{3},
 \qquad
 \widehat\kappa
 :=
 \frac{\widehat\lambda_{\max}}{\widehat\lambda_{\min}}
 \le
 3.
\]
\end{lemma}

\begin{proof}[Proof of Lemma~\ref{lem:minrate}]
Since
\[
 K\le d_{\mathrm{eff}}^{\sigma}
 =
 \left\lfloor
 \left(\frac{C_{\mathrm{sp}}}{\sigma^2}\right)^{1/\beta}
 \right\rfloor,
\]
we have, for every \(i\le K\),
\[
 i
 \le
 \left(\frac{C_{\mathrm{sp}}}{\sigma^2}\right)^{1/\beta}.
\]
Therefore
\[
 i^{-\beta}
 \ge
 \frac{\sigma^2}{C_{\mathrm{sp}}},
\]
and hence, by the exact power-law spectrum,
\[
 \sigma_i^2
 =
 C_{\mathrm{sp}}i^{-\beta}
 \ge
 \sigma^2.
\]

By Weyl's inequality and the covariance concentration event,
\[
 |\widehat\sigma_i^2-\sigma_i^2|
 \le
 \|\widehat\Sigma-\Sigma\|_{\mathrm{op}}
 \le
 \Delta_n.
\]
Thus, using \(\Delta_n\le\sigma^2/2\),
\[
 \widehat\sigma_i^2
 \ge
 \sigma_i^2-\Delta_n
 \ge
 \sigma^2-\frac{\sigma^2}{2}
 =
 \frac{\sigma^2}{2}.
\]

The empirical Wiener-gain rate is
\[
 \widehat\lambda_i
 =
 \lambda_0
 \frac{\widehat\sigma_i^2}
 {\widehat\sigma_i^2+\sigma^2}.
\]
Since the map
\[
 s\mapsto \frac{s}{s+\sigma^2}
\]
is increasing for \(s\ge0\), we obtain
\[
 \widehat\lambda_i
 \ge
 \lambda_0
 \frac{\sigma^2/2}{\sigma^2/2+\sigma^2}
 =
 \frac{\lambda_0}{3}.
\]
Also,
\[
 \widehat\lambda_i\le \lambda_0
\]
for all \(i\), because
\[
 \frac{\widehat\sigma_i^2}
 {\widehat\sigma_i^2+\sigma^2}
 \le 1.
\]
Therefore
\[
 \widehat\kappa
 =
 \frac{\widehat\lambda_{\max}}{\widehat\lambda_{\min}}
 \le
 \frac{\lambda_0}{\lambda_0/3}
 =
 3.
\]
\end{proof}

\begin{lemma}[Empirical-coordinate covariance perturbation]
\label{lem:nondiag}
On the event of Lemma~\ref{lem:bernstein},
\[
 \bigl\|\hat{V}_K^\top C\,\hat{V}_K - \hat{\Sigma}_K\bigr\|_F
 \leq \sqrt{K}\,\Delta_n,
\]
where $\hat{\Sigma}_K := \operatorname{diag}(\hat{\sigma}_1^2,\ldots,\hat{\sigma}_K^2)$.
This bounds the discrepancy between the true covariance in empirical
coordinates and the diagonal covariance obtained from the empirical
eigenvalues; it does not say the true empirical-coordinate covariance
is diagonal.
\end{lemma}

\begin{proof}[Proof of Lemma~\ref{lem:nondiag}]
Since the columns of $\hat{V}_K$ are eigenvectors of $\hat{\Sigma}$, we have
$\hat{V}_K^\top\hat{\Sigma}\hat{V}_K = \hat{\Sigma}_K$, and therefore
\[
 \hat{V}_K^\top C\,\hat{V}_K - \hat{\Sigma}_K
 \;=\; \hat{V}_K^\top(C - \hat{\Sigma})\hat{V}_K.
\]
By submultiplicativity, $\|\hat{V}_K\|_{\mathrm{op}}=1$ and $\|\hat{V}_K\|_F=\sqrt{K}$,
so
\[
 \bigl\|\hat{V}_K^\top(C-\hat{\Sigma})\hat{V}_K\bigr\|_F
 \;\leq\; \|C-\hat{\Sigma}\|_{\mathrm{op}}\,\|\hat{V}_K\|_F
 \;\leq\; \Delta_n\sqrt{K}.
\]
\end{proof}

\begin{proof}[Proof of the diagonal-approximation bound]
Set $\Delta_C := \hat{V}_K^\top(C-\hat{\Sigma})\hat{V}_K$ and
$E := \hat{\Sigma}_K^{-1/2}\Delta_C\,\hat{\Sigma}_K^{-1/2}$, so that
$\hat{\Sigma}_K^{-1/2}(\hat{V}_K^\top C\hat{V}_K)\hat{\Sigma}_K^{-1/2} = I+E$.
Since $\hat{\sigma}_K^2$ is the smallest diagonal entry of $\hat{\Sigma}_K$,
\[
 \|E\|_{\mathrm{op}}
 \;\leq\; \hat{\sigma}_K^{-2}\|\Delta_C\|_{\mathrm{op}}
 \;\leq\; \frac{\Delta_n}{\hat{\sigma}_K^2}
 \;\leq\; \frac{1}{2},
\]
so all eigenvalues $\mu_i$ of $I+E$ lie in $[1/2,3/2]$. For $s\in[1/2,3/2]$
the bound $s-1-\log s \leq (s-1)^2$ gives
\[
 2\,\mathrm{KL}\!\left(\mathcal{N}(0,\hat{V}_K^\top C\hat{V}_K)\,\big\|\,
 \mathcal{N}(0,\hat{\Sigma}_K)\right)
 \;=\; \sum_i(\mu_i - 1 - \log\mu_i)
 \;\leq\; \|E\|_F^2.
\]
By Lemma~\ref{lem:nondiag} and $\|\hat{\Sigma}_K^{-1}\|_{\mathrm{op}}\leq\hat{\sigma}_K^{-2}$,
\[
 \|E\|_F
 \;\leq\; \hat{\sigma}_K^{-2}\|\Delta_C\|_F
 \;\leq\; \frac{\sqrt{K}\,\Delta_n}{\hat{\sigma}_K^2},
\]
hence $\mathrm{KL} \leq K\Delta_n^2/(2\hat{\sigma}_K^4)$.
Pinsker's inequality yields the stated bound.
\end{proof}

\section{Proofs of the Empirical and Decoupling Results}
\label{app:proofs-empirical}

\subsection{Proof of Theorem~\ref{thm:empirical}}

The proof controls sampling in the estimated coordinates and then compares the empirical and population subspaces. Keeping these two comparisons separate isolates the covariance-estimation error.
\label{app:proof_empirical}

\begin{proof}
Define the population smoothed truncated Gaussian law
\[
 \widetilde q_K^\eta
 :=
 \mathcal N(0, P_K\mathcal C P_K+\eta^2I_d),
\]
and the empirical smoothed truncated Gaussian law
\[
 \widetilde q_{\widehat K}^\eta
 :=
 \mathcal N(0,\widehat P_K\mathcal C\widehat P_K+\eta^2I_d).
\]
The empirical sampler output law is denoted by
$\widetilde p_{\mathrm{emp}}^\eta$. By the triangle inequality,
\begin{equation}\label{eq:empirical-triangle}
 \TV(\widetilde p_{\mathrm{emp}}^\eta,q^\eta)
 \le
 \TV(\widetilde p_{\mathrm{emp}}^\eta,\widetilde q_{\widehat K}^\eta)
 +
 \TV(\widetilde q_{\widehat K}^\eta,\widetilde q_K^\eta)
 +
 \TV(\widetilde q_K^\eta,q^\eta).
\end{equation}
We bound the three terms separately.

The empirical active-space process has rate matrix
\[
 \widehat\Lambda
 =
 \diag(\widehat\lambda_1,\ldots,\widehat\lambda_K),
 \qquad
 \widehat\lambda_i
 =
 \lambda_0
 \frac{\widehat\sigma_i^2}{\widehat\sigma_i^2+\sigma^2}.
\]
On the covariance concentration event, Weyl's inequality gives
\[
 |\widehat\sigma_i^2-\sigma_i^2|
 \le
 \Delta_n,
 \qquad
 i=1,\ldots,K.
\]
Since $K\le d_{\mathrm{eff}}^{\sigma}$ and the spectrum is exact power-law,
\[
 \sigma_i^2\ge \sigma^2,
 \qquad
 i\le K.
\]
Using $\Delta_n\le c_0\sigma^2$ with $c_0\le 1/2$, we get
\[
 \widehat\sigma_i^2
 \ge
 \sigma_i^2-\Delta_n
 \ge
 \frac{\sigma^2}{2}.
\]
Therefore
\[
 \widehat\lambda_i
 =
 \lambda_0
 \frac{\widehat\sigma_i^2}
 {\widehat\sigma_i^2+\sigma^2}
 \ge
 \lambda_0
 \frac{\sigma^2/2}{\sigma^2/2+\sigma^2}
 =
 \frac{\lambda_0}{3}.
\]
Also $\widehat\lambda_i\le\lambda_0$, and hence
\[
 \widehat\kappa
 :=
 \frac{\widehat\lambda_{\max}}{\widehat\lambda_{\min}}
 \le 3.
\]
Thus the empirical rate matrix is uniformly conditioned with constants
independent of $K$ and $d$.

It remains to verify that the constants in the top-$K$ sampling theorem remain
independent of the patch ambient dimension when applied with the empirical rate matrix $\widehat\Lambda$.
Since $0<\widehat\lambda_i\le \lambda_0$, we have
\[
 \tr(\widehat\Lambda)
 =
 \sum_{i=1}^K \widehat\lambda_i
 \le
 \lambda_0K,
\]
and
\[
 \tr(\widehat\Lambda^2)
 =
 \sum_{i=1}^K \widehat\lambda_i^2
 \le
 \lambda_0^2K.
\]
For the empirical weighted second moment
$m_{2,\widehat\Lambda}^2 := \tr(\widehat\Lambda\,\widehat V_K^\top\mathcal
C\widehat V_K)$, Lemma~\ref{lem:trace} gives directly
\[
 m_{2,\widehat\Lambda}^2
 \le
 \lambda_0 C_{\mathrm{sp}}\zeta(\beta),
\]
independent of the patch ambient dimension $d$.

We also verify the Gaussian score regularity and weighted Fisher bound in the
empirical coordinates. Let
\[
 \widehat C_K:=\widehat V_K^\top\mathcal C\widehat V_K.
\]
For any unit vector $a\in\mathbb R^K$, set $u=\widehat V_Ka$. Since
$\widehat P_Ku=u$,
\[
 \|(I-P_K)u\|
 =\|(I-P_K)\widehat P_Ku\|
 \le \|(I-P_K)\widehat P_K\|_{\mathrm{op}}
 \le \|\widehat P_K-P_K\|_{\mathrm{op}}.
\]
By the operator-norm Davis-Kahan bound in Lemma~\ref{lem:dk} and the
assumption $\Delta_n\le c_0\delta_{\mathrm{gap}}$, we have
\[
 \|\widehat P_K-P_K\|_{\mathrm{op}}
 \le
 2c_0.
\]
Choose the universal constant $c_0>0$ small enough that
\((2c_0)^2\le1/2\); for instance, $c_0\le1/(2\sqrt2)$ suffices under the
constant convention of Lemma~\ref{lem:dk}. Then
\[
 \|P_Ku\|^2
 =1-\|(I-P_K)u\|^2
 \ge \frac12.
\]
Since $P_K$ commutes with $\mathcal C$ and the eigenvalues of
$\mathcal C$ on $\operatorname{col}(P_K)$ are at least $\sigma_K^2$, we get
\[
\begin{aligned}
 a^\top\widehat C_Ka
 &=u^\top\mathcal C u \\
 &\ge \sigma_K^2\|P_Ku\|^2 \\
 &=\sigma_K^2\bigl(1-\|(I-P_K)u\|^2\bigr)
 \ge \frac12\sigma_K^2.
\end{aligned}
\]
Because $K\le d_{\mathrm{eff}}^{\sigma}$, the exact power law gives
$\sigma_K^2\ge\sigma^2$, and therefore
\[
 \widehat C_K\succeq \frac12\sigma^2 I_K.
\]
For the empirical OU covariance
\[
 \widehat D_t=e^{-\widehat\Lambda t}\widehat C_K e^{-\widehat\Lambda t}
 +\sigma^2(I_K-e^{-2\widehat\Lambda t}),
\]
we consequently have, for every unit vector $a$,
\[
\begin{aligned}
 a^\top\widehat D_t a
 &\ge \frac{\sigma^2}{2}\|e^{-\widehat\Lambda t}a\|^2
  +\sigma^2\langle a,(I_K-e^{-2\widehat\Lambda t})a\rangle \\
 &=\sigma^2\sum_{i=1}^K
  \left(1-\frac12 e^{-2\widehat\lambda_i t}\right)a_i^2
 \ge \frac{\sigma^2}{2}.
\end{aligned}
\]
Thus $\widehat D_t\succeq (\sigma^2/2)I_K$ for all $t$, and the empirical score
$\nabla\log\widehat q_{K,t}(u)=-\widehat D_t^{-1}u$ is uniformly Lipschitz with
constant at most $2/\sigma^2$. Moreover,
\[
\begin{aligned}
 \mathbb E_{U\sim\widehat q_{K,t}}
 \|\nabla\log\widehat q_{K,t}(U)\|_{\widehat\Lambda}^2
 &=\tr(\widehat\Lambda\widehat D_t^{-1}) \\
 &\le \frac{2}{\sigma^2}\tr(\widehat\Lambda).
\end{aligned}
\]
Hence the coefficient-space Lipschitz and Fisher assumptions hold for the
empirical Gaussian OU process with constants independent of $d$.

Thus all quantities entering the top-$K$ sampling theorem with
$\widehat\Lambda$ satisfy
\[
 \widehat\kappa\le 3,
 \qquad
 \tr(\widehat\Lambda)=O(K),
 \qquad
 \tr(\widehat\Lambda^2)=O(K),
 \qquad
 m_{2,\widehat\Lambda}^2=O(1),
\]
with constants independent of $d$.

It remains to bound the mixing term of the top-$K$ theorem, which involves
$\KL(\N(0,\widehat C_K)\,\|\,\gamma_K)$ with $\gamma_K=\N(0,\sigma^2I_K)$.
Let $\mu_1,\ldots,\mu_K$ be the eigenvalues of $\widehat C_K/\sigma^2$. Since
$\widehat C_K\succeq\frac12\sigma^2I_K$, every $\mu_i\ge1/2$, and for
$\mu\ge1/2$ we have $-1-\log\mu\le\log2-1<0$, hence
$\mu-1-\log\mu\le\mu$. Therefore
\[
 \KL\bigl(\N(0,\widehat C_K)\,\|\,\gamma_K\bigr)
 =\frac12\sum_{i=1}^K(\mu_i-1-\log\mu_i)
 \le\frac{\tr(\widehat C_K)}{2\sigma^2}
 =\frac{\tr(\mathcal C\widehat P_K)}{2\sigma^2}
 \le\frac{\tr(\mathcal C)}{2\sigma^2}
 \le\frac{C_{\mathrm{sp}}\zeta(\beta)}{2\sigma^2},
\]
the same dimension-free bound as in the population case. Since
$\widehat\lambda_{\min}\ge\lambda_0/3$, the mixing term is at most
$e^{-\lambda_0T/3}\sqrt{C_{\mathrm{sp}}\zeta(\beta)}/(2\sigma)$, which a
horizon $T=O(\lambda_0^{-1}\log(1/\varepsilon))$ makes negligible, with a
constant independent of $K$ and $d$. Finally, the oracle scores are linear,
so Assumption~\ref{ass:girsanov} holds by the linear-growth criterion stated
after it.

By the oracle-score assumption,
\[
 \widehat s_t(u)
 =
 \nabla\log\widehat q_{K,t}(u),
\]
the empirical score error vanishes:
\[
 \varepsilon_{\mathrm{score},\widehat\Lambda}=0.
\]
Applying the top-$K$ sampling theorem in the empirical active space therefore
gives
\begin{equation}\label{eq:empirical-sampling-error}
 \TV(\widetilde p_{\mathrm{emp}}^\eta,\widetilde q_{\widehat K}^\eta)
 \le
 \widetilde O\!\left(\sqrt{\frac K N}\right),
\end{equation}
with constants independent of the patch ambient dimension.

We compare
\[
 \widetilde q_{\widehat K}^\eta
 =
 \mathcal N(0,\widehat P_K\mathcal C\widehat P_K+\eta^2I_d)
\]
and
\[
 \widetilde q_K^\eta
 =
 \mathcal N(0,P_K\mathcal CP_K+\eta^2I_d).
\]
Let
\[
 A
 :=
 \widehat P_K\mathcal C\widehat P_K+\eta^2I_d,
 \qquad
 B
 :=
 P_K\mathcal CP_K+\eta^2I_d.
\]
Then
\[
 A\succeq \eta^2I_d,
 \qquad
 B\succeq \eta^2I_d.
\]
Moreover,
\[
 A-B
 =
 \widehat P_K\mathcal C\widehat P_K-P_K\mathcal CP_K.
\]
Using the identity
\[
 \widehat P_K\mathcal C\widehat P_K-P_K\mathcal CP_K
 =
 (\widehat P_K-P_K)\mathcal C\widehat P_K
 +
 P_K\mathcal C(\widehat P_K-P_K),
\]
we obtain
\[
\begin{aligned}
 \|A-B\|_F
 &\le
 \|(\widehat P_K-P_K)\mathcal C\widehat P_K\|_F
 +
 \|P_K\mathcal C(\widehat P_K-P_K)\|_F \\
 &\le
 2\|\mathcal C\|_{\mathrm{op}}
 \|\widehat P_K-P_K\|_F.
\end{aligned}
\]
Since $\|\mathcal C\|_{\mathrm{op}}=\sigma_1^2$,
\[
 \|A-B\|_F
 \le
 2\sigma_1^2\|\widehat P_K-P_K\|_F.
\]
By the Davis-Kahan theorem, using
$\Delta_n\le c_0\delta_{\mathrm{gap}}$ with $c_0\le1/2$,
\[
 \|\widehat P_K-P_K\|_F
 \le
 \frac{2\sqrt{2K}\Delta_n}{\delta_{\mathrm{gap}}}.
\]
Therefore
\begin{equation}\label{eq:cov-perturb-bound}
 \|A-B\|_F
 \le
 C
 \frac{\sigma_1^2\sqrt K}{\delta_{\mathrm{gap}}}
 \Delta_n.
\end{equation}

The additional assumption
\[
 \Delta_n
 \le
 c_0
 \frac{\eta^2\delta_{\mathrm{gap}}}{\sigma_1^2\sqrt K}
\]
with $c_0>0$ sufficiently small ensures
\[
 \|B^{-1/2}(A-B)B^{-1/2}\|_{\mathrm{op}}
 \le
 \frac12.
\]
Indeed, since $B\succeq \eta^2I_d$,
\[
\begin{aligned}
 \|B^{-1/2}(A-B)B^{-1/2}\|_{\mathrm{op}}
 &\le
 \frac{\|A-B\|_{\mathrm{op}}}{\eta^2} \\
 &\le
 \frac{\|A-B\|_F}{\eta^2} \\
 &\le
 \frac{C\sigma_1^2\sqrt K}{\eta^2\delta_{\mathrm{gap}}}\Delta_n
 \le
 \frac12.
\end{aligned}
\]

We now bound the Gaussian TV distance. For centered Gaussians,
\[
 \KL(\mathcal N(0,A)\,\|\,\mathcal N(0,B))
 =
 \frac12
 \left[
  \tr(B^{-1}A)-d-\log\det(B^{-1}A)
 \right].
\]
Let
\[
 H
 :=
 B^{-1/2}(A-B)B^{-1/2}.
\]
Then
\[
 B^{-1/2}AB^{-1/2}
 =
 I+H.
\]
Since $\|H\|_{\mathrm{op}}\le1/2$, all eigenvalues of $I+H$ lie in
$[1/2,3/2]$. The scalar inequality
\[
 s-1-\log s
 \le
 C(s-1)^2,
 \qquad
 s\in[1/2,3/2],
\]
implies
\[
 \KL(\mathcal N(0,A)\,\|\,\mathcal N(0,B))
 \le
 C\|H\|_F^2.
\]
Furthermore,
\[
 \|H\|_F
 =
 \|B^{-1/2}(A-B)B^{-1/2}\|_F
 \le
 \frac{\|A-B\|_F}{\eta^2}.
\]
Hence
\[
 \KL(\mathcal N(0,A)\,\|\,\mathcal N(0,B))
 \le
 C
 \frac{\|A-B\|_F^2}{\eta^4}.
\]
By Pinsker's inequality,
\[
 \TV(\mathcal N(0,A),\mathcal N(0,B))
 \le
 C
 \frac{\|A-B\|_F}{\eta^2}.
\]
Using \eqref{eq:cov-perturb-bound}, we obtain
\begin{equation}\label{eq:empirical-basis-error}
 \TV(\widetilde q_{\widehat K}^\eta,\widetilde q_K^\eta)
 \le
 C_1
 \frac{\sigma_1^2\sqrt K}{\eta^2\delta_{\mathrm{gap}}}
 \Delta_n.
\end{equation}

The third term of \eqref{eq:empirical-triangle},
$\TV(\widetilde q_K^{\,\eta},q^\eta)$, is the truncation error of
Section~\ref{sec:sharp}; we keep it as it is. Theorem~\ref{thm:num-directions}(a) bounds it by
$C_{\mathrm{sp}}K^{-(2\beta-1)/2}/(\eta^2\sqrt{8(2\beta-1)})$.

\paragraph{Assembly.}
Substituting
\eqref{eq:empirical-sampling-error}
and
\eqref{eq:empirical-basis-error}
into
\eqref{eq:empirical-triangle}
gives
\[
 \TV(\widetilde p_{\mathrm{emp}}^\eta,q^\eta)
 \le
 \widetilde O\!\left(\sqrt{\frac K N}\right)
 +
 C_1
 \frac{\sigma_1^2\sqrt K}{\eta^2\delta_{\mathrm{gap}}}
 \Delta_n
 +
 \TV\bigl(\widetilde q_K^{\,\eta},q^\eta\bigr).
\]
The bound depends on $K,N,\eta,\beta,\sigma_1,\delta_{\mathrm{gap}}$ and
$\Delta_n$, but not on the ambient dimension $d$.
\end{proof}

\subsection{Proof of Theorem~\ref{thm:decouple}}

We verify that the training-data condition controls the perturbation uniformly over all admissible retained dimensions. The population choices then follow after reserving the stated error budget for basis estimation.
\label{app:proof_decouple}

\begin{proof}
We first show that the empirical perturbation assumptions and the empirical
basis/covariance mismatch bound hold uniformly for every $K\le K_{\max}$.
Under the exact power-law spectrum,
\[
 \sigma_K^2=C_{\mathrm{sp}}K^{-\beta},
 \qquad
 \delta_{\mathrm{gap}}(K)
 =C_{\mathrm{sp}}\bigl(K^{-\beta}-(K+1)^{-\beta}\bigr).
\]
The quantities \(\sigma_K^2\), \(\delta_{\mathrm{gap}}(K)\), and
\(\eta^2\delta_{\mathrm{gap}}(K)/(\sigma_1^2\sqrt K)\) are decreasing in
\(K\), so the small-perturbation threshold \(M_K\) is minimized at
\(K_{\max}\). We also need uniform control of the empirical mismatch factor.
By the mean-value theorem, there exist constants
\(0<c_\beta\le C_\beta<\infty\), depending only on \(\beta\) and
\(C_{\mathrm{sp}}\), such that for every \(K\ge1\),
\[
 c_\beta K^{-\beta-1}
 \le
 \delta_{\mathrm{gap}}(K)
 \le
 C_\beta K^{-\beta-1}.
\]
Consequently, for every \(K\le K_{\max}\),
\[
 \frac{\sqrt K}{\delta_{\mathrm{gap}}(K)}
 \le
 C_\beta'
 \frac{\sqrt{K_{\max}}}{\delta_{\mathrm{gap}}(K_{\max})},
\]
where \(C_\beta'>0\) is another constant depending only on the spectral
power-law constants. Therefore the most restrictive smallness and
basis-mismatch requirements over \(K\le K_{\max}\) are controlled by their
values at \(K_{\max}\), up to fixed constants.

By Assumption~\ref{ass:sufficient_data} and
\[
 \Delta_n=C_\Delta\sigma_1^2\sqrt{\frac{D_p}{n_{\mathrm{patch}}}},
\]
the first term in the maximum defining the sample-size condition ensures
\[
 \Delta_n\le c_0 M_{K_{\max}},
\]
after adjusting the hidden universal constant. Since $M_K\ge M_{K_{\max}}$
for all $K\le K_{\max}$ under the monotonicity above, the small-perturbation
conditions in Theorem~\ref{thm:empirical} hold uniformly for all such $K$.

The second term in Assumption~\ref{ass:sufficient_data} gives, again after
absorbing universal constants,
\[
 \Delta_n
 \le
 c\,\frac{\varepsilon\eta^2\delta_{\mathrm{gap}}(K_{\max})}
 {\sigma_1^2\sqrt{K_{\max}}}.
\]
Consequently, for every $K\le K_{\max}$,
\[
\begin{aligned}
 C_1\frac{\sigma_1^2\sqrt K}{\eta^2\delta_{\mathrm{gap}}(K)}\Delta_n
 &\le
 C_1\frac{\sigma_1^2\sqrt{K_{\max}}}
 {\eta^2\delta_{\mathrm{gap}}(K_{\max})}\Delta_n \\
 &\le \frac{\varepsilon}{4},
\end{aligned}
\]
by choosing the hidden constant in the sample-size lower bound sufficiently
large. This proves the claimed uniform control of the empirical
basis/covariance mismatch term.

Substituting this uniform estimate into Theorem~\ref{thm:empirical} gives
\eqref{eq:reduced}; the remaining two terms are exactly the population
sampling and truncation terms, which proves the last claim.
\end{proof}

\section{Empirical Results}\label{app:exp}
\subsection{Empirical check of assumption \ref{ass:powerlaw}}\label{app:e1}

Assumption \ref{ass:powerlaw} requires a two-sided power-law \emph{envelope}
$c_{\mathrm{sp}}\, i^{-\beta}\le \sigma_i^2\le C_{\mathrm{sp}}\, i^{-\beta}$ with $\beta>1$; it does not require the
eigenvalues to follow one exact power law. This appendix checks the envelope, and the sign condition on $\beta$, on the
covariance spectra of natural-image patches. We ask three questions: (i) is the decay of power-law type rather than
exponential (low-rank-like); (ii) how large is the envelope ratio $C_{\mathrm{sp}}/c_{\mathrm{sp}}$, and does it hold on
data not used to estimate $\beta$; and (iii) is the departure from a single power law, which the data clearly show,
compatible with the way Assumption \ref{ass:powerlaw} is used in Sections~3-4. The short answers are: yes; $C_{\mathrm{sp}}/c_{\mathrm{sp}}\le
2.42$ on the fitted range and $\le 4.11$ on the full spectral bulk, out of sample; and yes, because the
departure is a \emph{broken} power law, which satisfies the envelope by construction
(Proposition~\ref{prop:broken-envelope}), followed by a high-frequency roll-off that only affects the side of the
envelope the analysis does not need there (Section~\ref{app:a1-compat}).

\subsubsection{Data and Protocol}
\label{app:a1-protocol}

\paragraph{Patches.} We use CIFAR-10 and CIFAR-100(\cite{krizhevsky2009learning}), DIV2K (\cite{Agustsson_2017_CVPR_Workshops}) and BSDS500 (\cite{MartinFTM01}). A patch is a random $p\times p$ RGB crop with values in $[0,1]$, vectorised to $d=3p^2$ coordinates, for
$p\in\{8,16\}$ (CIFAR) and $p\in\{8,16,32\}$ (DIV2K, BSDS500). Train, validation and test patches come from disjoint
images (Table~\ref{tab:a1-data}): CIFAR uses one patch per image; DIV2K and BSDS500 use $16$ patches from each of $100$
images per split ($80$ for the DIV2K validation split). BSDS500 uses its official split; for DIV2K the original training
images are split $90/10$ into train/validation and the original validation images form the test set.

\paragraph{Spectra and fitting window.} On each split we compute the exact eigenvalues
$\hat\sigma_1^2\ge\dots\ge\hat\sigma_d^2$ of the sample covariance (global mean removed). Exponents are fitted by least
squares on $\log\hat\sigma_i^2$ at $60$ log-spaced ranks in a window $[\ell,h]$ with $\ell=8$ and
$h=\min\{d/4,\;(n_{\min}-1)/4,\;r_{\mathrm{num}}/2\}$, where $n_{\min}$ is the smallest split size and $r_{\mathrm{num}}$
the numerically resolved rank. The cap at $d/4$ and $n/4$ keeps the window where sample eigenvalues estimate population
eigenvalues reliably; ranks below $8$ (mean luminance and colour directions) are examined separately.

\paragraph{Models.} On the common window we compare a power law $\log\sigma_i^2=a-\beta\log i$, an exponential law
$\log\sigma_i^2=a-\lambda i$, and the continuous broken power law of Definition~\ref{def:broken}. The broken law's
breakpoint is chosen among $9$ log-spaced candidates by validation error, and all three models are scored by mean squared
error of $\log\hat\sigma_i^2$ on the test split.

\paragraph{Envelope ratio.} With $\hat\beta$ the power-law exponent fitted on the training split, we define the
\emph{compensated spectrum} $\hat\sigma_i^2\, i^{\hat\beta}$ and report, on each split separately,
\[
\frac{C}{c}\;=\;\frac{\max_{i\in\mathcal I}\hat\sigma_i^2\, i^{\hat\beta}}{\min_{i\in\mathcal I}\hat\sigma_i^2\, i^{\hat\beta}},
\]
the smallest possible ratio $C_{\mathrm{sp}}/c_{\mathrm{sp}}$ of Assumption \ref{ass:powerlaw} on a rank set $\mathcal I$ with exponent
$\hat\beta$. A value of $1$ is an exact power law. Because $\hat\beta$ is estimated on training patches only, the
validation and test ratios are out-of-sample checks. We use two rank sets: the fitting window $\mathcal I=[\ell,h]$, and
the \emph{spectral bulk} $\mathcal I=[8,m]$, where the bulk end $m$ is the largest rank such that
$\hat\sigma_i^2 i^{\hat\beta}\ge c/2$ for all $8\le i\le m$, with $c$ the training-window minimum.

\begin{table*}[t]
\centering
\caption{Data used for the check. Vectors and independent sources (images; artists for FMA(\cite{fma_dataset})) per train/val/test split;
fitting window $[\ell,h]$; fitted breakpoint relative to dimension, $\hat b/d$. $^\dagger$Audio, outside the scope of the
paper.}
\label{tab:a1-data}
\small
\begin{tabular}{llrcccc}
\toprule
Dataset & Patch & $d$ & Vectors & Sources & Window $[\ell,h]$ & $\hat b/d$ \\
\midrule
CIFAR-10 & 8$\times$8 & 192 & 6000/2000/2000 & 6000/2000/2000 & $[8,48]$ & 0.09 \\
CIFAR-10 & 16$\times$16 & 768 & 6000/2000/2000 & 6000/2000/2000 & $[8,192]$ & 0.10 \\
CIFAR-100 & 8$\times$8 & 192 & 6000/2000/2000 & 6000/2000/2000 & $[8,48]$ & 0.17 \\
CIFAR-100 & 16$\times$16 & 768 & 6000/2000/2000 & 6000/2000/2000 & $[8,192]$ & 0.10 \\
DIV2K & 8$\times$8 & 192 & 1600/1280/1600 & 100/80/100 & $[8,48]$ & 0.17 \\
DIV2K & 16$\times$16 & 768 & 1600/1280/1600 & 100/80/100 & $[8,192]$ & 0.10 \\
DIV2K & 32$\times$32 & 3072 & 1600/1280/1600 & 100/80/100 & $[8,319]$ & 0.05 \\
BSDS500 & 8$\times$8 & 192 & 1600/1600/1600 & 100/100/100 & $[8,48]$ & 0.17 \\
BSDS500 & 16$\times$16 & 768 & 1600/1600/1600 & 100/100/100 & $[8,192]$ & 0.13 \\
BSDS500 & 32$\times$32 & 3072 & 1600/1600/1600 & 100/100/100 & $[8,399]$ & 0.09 \\
\midrule
FMA$^\dagger$ & 512 & 512 & 1600/1600/1600 & 96/66/52 & $[8,128]$ & 0.08 \\
FMA$^\dagger$ & 2048 & 2048 & 1600/1600/1600 & 96/66/52 & $[8,399]$ & 0.04 \\
\bottomrule
\end{tabular}
\end{table*}

\subsubsection{Results}
\label{app:a1-results}

\paragraph{The decay is of power-law type.} Figure~\ref{fig:a1-spectra} shows the spectra on a log-log scale. Across
the ten image cases, the exponential model has $2.4$-$20$ times the test error of the power law
(Table~\ref{tab:a1-fits}), so the spectra are not low-rank-like: every direction carries variance and the decay is
polynomial, the regime Assumption \ref{ass:powerlaw} describes (see also Appendix~A.2).

\paragraph{The exponent satisfies $\beta>1$.} The fitted exponents lie in $\hat\beta\in[1.24,1.85]$ across
all image cases and splits, and the three independent splits agree to within $0.11$
(Table~\ref{tab:a1-fits}), so $\beta>1$, needed for the dimension-free moment bound of Section~4, holds throughout; the
weaker condition $\beta>1/2$ of the truncation analysis holds with a wide margin. $\hat\beta$ decreases with patch size
(Figure~\ref{fig:a1-summary}a), from about $1.65$-$1.84$ at $8\times8$ to $1.24$-$1.27$ at $32\times32$, approaching
the value $\beta\approx1$ suggested by the Fourier argument of Appendix~B.5. Assumption \ref{ass:powerlaw} is stated for a fixed
dimension $d$, so we report its constants per patch size.

\paragraph{The envelope holds with small constants, out of sample.} On the fitting window the envelope ratio is at most
$2.42$ for every image case and split (Table~\ref{tab:a1-envelope}): over a range of up to $50$ in rank, every
eigenvalue lies within a factor of about two of a single power law. Figure~\ref{fig:a1-compensated} shows the compensated
spectra; they are flat to within the two horizontal lines on the window and well beyond it, and the validation and test
curves follow the training curve although $\hat\beta$ was not fitted on them. Extending to the spectral bulk $[8,m]$,
which covers ranks up to $m/d=0.33$-$0.50$ of the dimension for $8\times8$ and $16\times16$
patches, the ratio is at most $4.11$.

\paragraph{Low and high ranks.} The two ends of the spectrum behave differently, and both are visible in
Figure~\ref{fig:a1-compensated}. The first few eigenvalues (mean luminance and the colour-opponent directions) lie
\emph{above} the bulk envelope; including all resolved ranks raises the upper constant $C$ by at most a factor
$6.9$ (last column of Table~\ref{tab:a1-envelope}), while the lower side holds there with margin
($\hat\sigma_i^2 i^{\hat\beta}\ge 0.60\, c$ for $i<8$). Beyond the bulk end $m$, the spectrum rolls off faster than
any power law with these exponents. This is the high-spatial-frequency end of the patch, where optical blur,
anti-aliasing, demosaicing and JPEG quantisation remove energy; it lies \emph{below} the envelope, so it violates only
the lower side.

\begin{table*}[t]
\centering
\caption{Model comparison on the common window. $\hat\beta$: power-law exponent fitted separately on each split. Test
log-MSE of the power, exponential and broken power laws (smallest in bold). $\hat\beta_1,\hat\beta_2,\hat b$:
broken-law head and tail exponents and breakpoint (training split). }
\label{tab:a1-fits}
\small
\begin{tabular}{llcccccccccc}
\toprule
 & & \multicolumn{3}{c}{$\hat\beta$} & \multicolumn{3}{c}{Test log-MSE} & \multicolumn{3}{c}{Broken power law} \\
\cmidrule(lr){3-5}\cmidrule(lr){6-8}\cmidrule(lr){9-11}
Dataset & Patch & train & val & test & power & exp. & broken & $\hat\beta_1$ & $\hat\beta_2$ & $\hat b$ \\
\midrule
CIFAR-10 & 8$\times$8 & 1.82 & 1.85 & 1.83 & 0.009 & 0.022 & \textbf{0.004} & 1.28 & 2.04 & 16 \\
CIFAR-10 & 16$\times$16 & 1.57 & 1.60 & 1.60 & 0.018 & 0.161 & \textbf{0.003} & 1.42 & 2.05 & 74 \\
CIFAR-100 & 8$\times$8 & 1.84 & 1.83 & 1.83 & 0.009 & 0.031 & \textbf{0.005} & 1.74 & 2.43 & 34 \\
CIFAR-100 & 16$\times$16 & 1.56 & 1.57 & 1.58 & 0.016 & 0.167 & \textbf{0.004} & 1.43 & 1.96 & 74 \\
DIV2K & 8$\times$8 & 1.65 & 1.68 & 1.70 & 0.011 & 0.029 & \textbf{0.010} & 1.51 & 2.43 & 34 \\
DIV2K & 16$\times$16 & 1.49 & 1.55 & 1.52 & 0.031 & 0.135 & \textbf{0.009} & 1.29 & 2.12 & 74 \\
DIV2K & 32$\times$32 & 1.27 & 1.38 & 1.34 & 0.015 & 0.279 & \textbf{0.010} & 1.19 & 1.87 & 153 \\
BSDS500 & 8$\times$8 & 1.64 & 1.57 & 1.63 & 0.010 & 0.023 & \textbf{0.006} & 1.52 & 2.40 & 34 \\
BSDS500 & 16$\times$16 & 1.40 & 1.38 & 1.43 & 0.025 & 0.123 & \textbf{0.006} & 1.26 & 2.41 & 102 \\
BSDS500 & 32$\times$32 & 1.24 & 1.26 & 1.27 & 0.015 & 0.304 & \textbf{0.008} & 1.20 & 2.54 & 270 \\
\midrule
FMA$^\dagger$ & 512 & 1.51 & 1.55 & 1.42 & 0.184 & 0.207 & \textbf{0.164} & 1.28 & 1.85 & 42 \\
FMA$^\dagger$ & 2048 & 1.24 & 1.19 & 1.21 & 0.177 & 0.207 & \textbf{0.106} & 0.93 & 1.77 & 84 \\
\bottomrule
\end{tabular}
\end{table*}

\begin{table*}[t]
\centering
\caption{Envelope ratios $C/c$ of Assumption \ref{ass:powerlaw} with the training exponent $\hat\beta$. Window: observed ratio per split,
and the ratio implied by the fitted broken power law alone (Proposition~\ref{prop:broken-envelope}). Bulk: end $m$ on
the test split, and the ratio on $[8,m]$ per split. $U$: factor by which including every resolved rank (in particular
$i<8$) raises the upper constant, maximum over splits.}
\label{tab:a1-envelope}
\small
\begin{tabular}{llccccrcccccc}
\toprule
 & & \multicolumn{4}{c}{Window $[\ell,h]$} & \multicolumn{5}{c}{Bulk $[8,m]$} & \\
\cmidrule(lr){3-6}\cmidrule(lr){7-11}
Dataset & Patch & train & val & test & broken law & $m$ & $m/d$ & train & val & test & $U$ \\
\midrule
CIFAR-10 & 8$\times$8 & 1.66 & 1.68 & 1.55 & 1.47 & 89 & 0.46 & 3.21 & 3.20 & 3.28 & 1.0 \\
CIFAR-10 & 16$\times$16 & 1.63 & 1.70 & 1.79 & 1.57 & 295 & 0.38 & 3.25 & 3.18 & 3.30 & 1.0 \\
CIFAR-100 & 8$\times$8 & 1.62 & 1.64 & 1.53 & 1.23 & 96 & 0.50 & 3.11 & 3.29 & 3.05 & 1.0 \\
CIFAR-100 & 16$\times$16 & 1.55 & 1.61 & 1.68 & 1.46 & 306 & 0.40 & 3.12 & 3.05 & 3.21 & 1.1 \\
DIV2K & 8$\times$8 & 1.51 & 1.37 & 1.39 & 1.32 & 76 & 0.40 & 2.93 & 2.21 & 2.57 & 5.1 \\
DIV2K & 16$\times$16 & 2.01 & 2.42 & 2.04 & 1.82 & 265 & 0.35 & 3.97 & 3.22 & 3.73 & 5.0 \\
DIV2K & 32$\times$32 & 1.85 & 2.42 & 1.92 & 1.56 & 484 & 0.16 & 3.69 & 3.00 & 3.48 & 6.9 \\
BSDS500 & 8$\times$8 & 1.35 & 1.31 & 1.39 & 1.31 & 66 & 0.34 & 2.58 & 2.66 & 2.34 & 3.3 \\
BSDS500 & 16$\times$16 & 1.95 & 1.76 & 2.11 & 1.90 & 251 & 0.33 & 3.86 & 4.11 & 3.98 & 4.6 \\
BSDS500 & 32$\times$32 & 1.97 & 1.84 & 1.99 & 1.66 & 604 & 0.20 & 3.93 & 4.05 & 3.97 & 4.7 \\
\midrule
FMA$^\dagger$ & 512 & 1.68 & 1.92 & 2.09 & 1.45 & 283 & 0.55 & 3.36 & 3.42 & 5.34 & 1.6 \\
FMA$^\dagger$ & 2048 & 2.41 & 2.95 & 2.68 & 2.26 & 648 & 0.32 & 4.83 & 4.68 & 6.70 & 1.4 \\
\bottomrule
\end{tabular}
\end{table*}

\begin{figure*}[t]
\centering
\includegraphics[width=\textwidth]{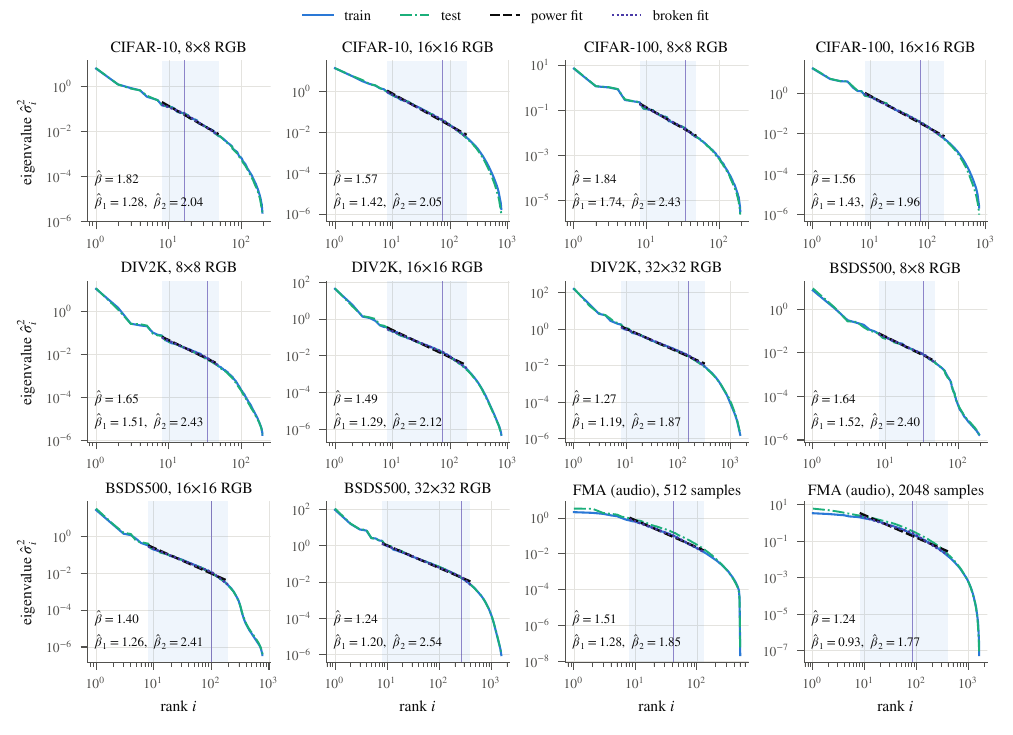}
\caption{Sample covariance spectra of image patches (train and test splits; FMA audio in the last two panels). Shading:
fitting window $[\ell,h]$. Dashed: power-law fit; dotted: broken power-law fit; thin vertical line: fitted breakpoint
$\hat b$. Ranks beyond the sample rank are not shown.}
\label{fig:a1-spectra}
\end{figure*}

\begin{figure*}[t]
\centering
\includegraphics[width=\textwidth]{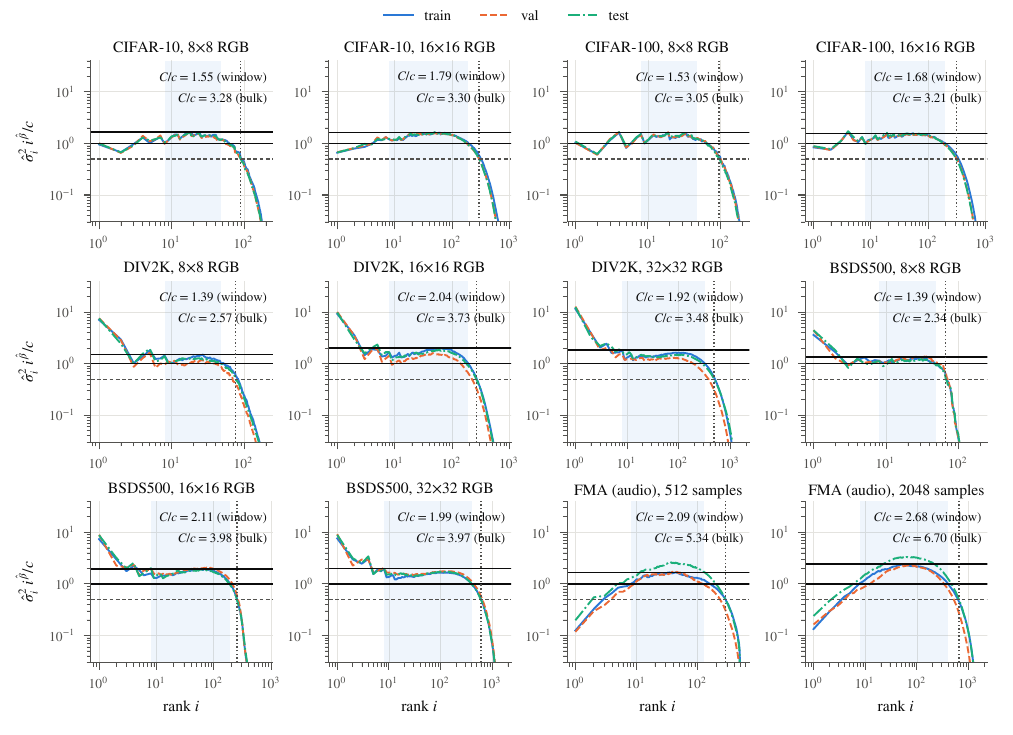}
\caption{Compensated spectra $\hat\sigma_i^2\, i^{\hat\beta}$, normalised by the training-window minimum $c$, for the
three disjoint splits; $\hat\beta$ is fitted on the training split only. An exact power law is a horizontal line.
Solid horizontal lines: training envelope $[c,C]$ on the window (shaded). Dashed line: $c/2$, which defines the bulk end
$m$ (dotted vertical line, test split). Values below $0.03$ are clipped.}
\label{fig:a1-compensated}
\end{figure*}

\begin{figure*}[t]
\centering
\includegraphics[width=\textwidth]{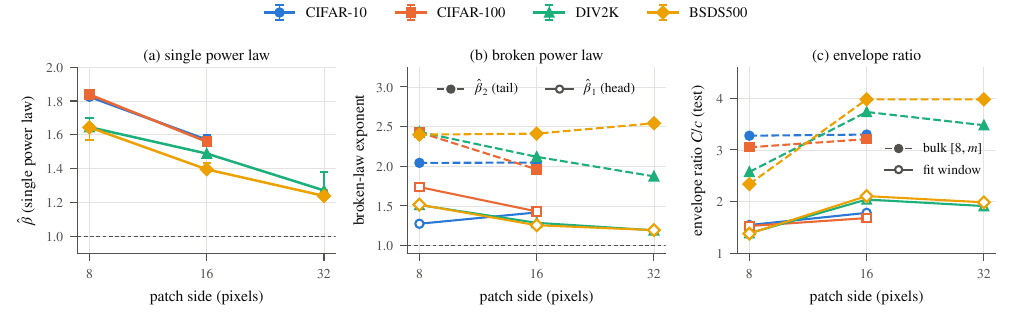}
\caption{Dependence on patch size (image datasets). (a) Single power-law exponent; bars span the estimates on the
three disjoint splits. (b) Broken-law head and tail exponents. (c) Envelope ratio on the test split, over the fitting
window and over the bulk $[8,m]$. Dashed horizontal lines in (a,b): $\beta=1$.}
\label{fig:a1-summary}
\end{figure*}

\subsubsection{The Broken Power Law and Why It Is Compatible with Assumption \ref{ass:powerlaw}}
\label{app:a1-broken}

In every case the broken power law has the smallest validation error, and on the test split it improves on the single
power law by a factor $1.1$-$6.8$ in log-MSE (Table~\ref{tab:a1-fits}). The spectra therefore steepen
gradually within the window; a single exponent is an idealisation. This does not conflict with Assumption \ref{ass:powerlaw}, which only
asks for a two-sided envelope.

\begin{definition}[Continuous broken power law]
\label{def:broken}
For $A>0$, a breakpoint $b\ge1$ and exponents $\beta_1\le\beta_2$, the continuous broken power law is
\[
\sigma_i^2 \;=\;
\begin{cases}
A\, i^{-\beta_1}, & i\le b,\\[2pt]
A\, b^{\beta_2-\beta_1}\, i^{-\beta_2}, & i> b,
\end{cases}
\]
or equivalently
$\log\sigma_i^2=\log A-\beta_1\log i-(\beta_2-\beta_1)(\log i-\log b)_+$.
It decays like $i^{-\beta_1}$ before the breakpoint and like $i^{-\beta_2}$ after it, and is continuous at $i=b$. The
case $\beta_1=\beta_2$ is the exact power law~(2).
\end{definition}

\begin{proposition}[Envelope of a broken power law]
\label{prop:broken-envelope}
Let $\sigma_i^2$ be a continuous broken power law and $1\le\ell\le b\le h$. For every $\beta\in[\beta_1,\beta_2]$,
$c\, i^{-\beta}\le\sigma_i^2\le C\, i^{-\beta}$ for $\ell\le i\le h$, where
\[
\frac{C}{c}=\max\Bigl\{\Bigl(\frac{b}{\ell}\Bigr)^{\beta-\beta_1},\;\Bigl(\frac{h}{b}\Bigr)^{\beta_2-\beta}\Bigr\}
\]
and $C=A\,b^{\beta-\beta_1}$. In particular, with $\beta=\beta_2$ the ratio is $(b/\ell)^{\beta_2-\beta_1}$, which
does not depend on $h$: a broken power law with a fixed breakpoint satisfies Assumption \ref{ass:powerlaw} with exponent $\beta_2$ and
constants independent of the dimension.
\end{proposition}

\begin{proof}
Let $f(i)=\sigma_i^2\, i^{\beta}$. For $i\le b$, $f(i)=A\,i^{\beta-\beta_1}$ is nondecreasing because
$\beta\ge\beta_1$; for $i\ge b$, $f(i)=A\,b^{\beta_2-\beta_1} i^{\beta-\beta_2}$ is nonincreasing because
$\beta\le\beta_2$. Hence $\max_{[\ell,h]}f=f(b)=A\,b^{\beta-\beta_1}$ and $\min_{[\ell,h]}f=\min\{f(\ell),f(h)\}$, and
the ratio $f(b)/\min\{f(\ell),f(h)\}$ is the stated maximum.
\end{proof}

\paragraph{What the data show.} Every fitted broken law satisfies $\hat\beta_1\le\hat\beta\le\hat\beta_2$, so
Proposition~\ref{prop:broken-envelope} applies with the single-law exponent. The ratio it predicts from the bend alone
(column ``broken law'' in Table~\ref{tab:a1-envelope}) accounts for $44$-$96\%$ of the observed
$\log(C/c)$; the remainder is the small scatter of individual eigenvalues around the smooth law. The head exponents
satisfy $\hat\beta_1\ge1.19 > 1$ and the tail exponents $\hat\beta_2\in[1.87,2.54]$, so both pieces of
the broken law individually satisfy the exponent condition of Assumption \ref{ass:powerlaw}. The breakpoint sits at a roughly constant
fraction of the dimension, $\hat b/d\in[0.05,0.17]$ (Table~\ref{tab:a1-data}), rather than at a fixed rank: the
bend is the onset of the high-frequency roll-off described in Section~\ref{app:a1-results}, not a change of regime among
low frequencies. The breakpoint is selected from a coarse grid of $9$ candidates and should be read as a location, not a
precise estimate.

\subsubsection{Which Side of the Envelope the Analysis Uses}
\label{app:a1-compat}

The departures from an exact power law found above are of two kinds: a bounded multiplicative wiggle and bend in the
bulk, and a roll-off below the envelope at the highest frequencies. The first is absorbed by the envelope constants
(Proposition~\ref{prop:broken-envelope}, Table~\ref{tab:a1-envelope}). For the second, note that the two sides of~(1)
enter the analysis in different places.
\begin{itemize}
\item The \emph{upper} side $\sigma_i^2\le C_{\mathrm{sp}}i^{-\beta}$ is used for all ranks: it bounds the total
variance by $C_{\mathrm{sp}}\zeta(\beta)$ (Section~4.1, Lemma~2) and the tail sums $\sum_{i>K}x_i^2$ in the
sufficiency part of Theorem~2, whose proof uses no other property of the spectrum. A spectrum that falls faster than
$i^{-\beta}$ only decreases these quantities, so the high-frequency roll-off keeps every upper bound valid and makes the
sufficient number of retained directions conservative. Empirically, the upper side holds on every resolved rank once
$C$ is enlarged by the factor $U\le6.9$ of Table~\ref{tab:a1-envelope}.
\item The \emph{lower} side $\sigma_i^2\ge c_{\mathrm{sp}}i^{-\beta}$ is used only for the retained directions
$i\le K\le d^{\sigma}_{\mathrm{eff}}$, where it keeps the rates of Theorem~3 bounded below, and on the block
$K<i\le2K$, where it gives the necessity part of Theorem~2 and Corollary~1. Empirically, it holds with $c/2$ on all ranks
$i\le m$, including $i<8$.
\end{itemize}
Consequently the data support Assumption \ref{ass:powerlaw} in the form in which the analysis uses it: an upper envelope on all resolved
ranks, and a lower envelope on the spectral bulk, which suffices whenever $2K\le m$. The sharp statements of Section~3
(Theorem~2, Corollary~1) and Theorem~4 are formulated for the exact power law~(2); in the light of
Table~\ref{tab:a1-fits}, they should be read as describing the bulk exponent, with the constants of the envelope in
place of $C_{\mathrm{sp}}$. Theorem~1 needs no spectral model at all and can be evaluated directly on the empirical
eigenvalues.
\subsection{Empirical check of the truncation theory}
\label{app:e2}

This appendix checks the results of Sections~3-4 on the model for which they are stated: centered Gaussian data
with the exact power law~(2). For this model every law involved is Gaussian, so total variation can be estimated
directly from the densities, without a learned discriminator. We test, in order, the two-sided characterisation of the
truncation error (Theorem~1), the number of retained directions and its exponents (Theorem~2), the cut-off rule
(Theorem~2(c)), the failure of truncating at the noise level (Corollary~1), the absence of a dependence on $d$, the
lower bound for imperfect samplers (Theorem~1(b)) and the decomposition~(10), and the exact-score sampler of
Theorems~3-4. In short: every bound holds at every tested point; the truncation error equals
$0.28$-$0.30$ times the Wiener-gain tail norm; the fitted exponents agree with $2/(2\beta-1)$ to within $1\%$; and
the explicit sufficient choice $K_\varepsilon(\eta)$ is within a factor $1.5$ of the true minimum.

\subsubsection{Setup}
\label{app:e2-setup}

\paragraph{Model.} We take $C_{\mathrm{sp}}=1$, $\sigma_i^2=i^{-\beta}$ for $i\le d$, with $\beta\in\{1.25, 1.5, 2\}$ and
$d=2^{18}$. The evaluation level is set through the number of directions above it,
$k_\eta=\lfloor(C_{\mathrm{sp}}/\eta^2)^{1/\beta}\rfloor$: we use $\eta^2=k_\eta^{-\beta}$ with
$k_\eta\in\{4, 8, 16, 32, 64\}$, and accuracies $\varepsilon\in[0.001,0.3]$ (nine values).

\paragraph{Estimating TV.} For two laws $P,Q$ with densities,
$\mathrm{TV}(P,Q)=\mathbb E_{Y\sim(P+Q)/2}\bigl|\tanh\bigl(\tfrac12\log\tfrac{dP}{dQ}(Y)\bigr)\bigr|$, which we
average over $10{,}000$ draws. For the ideal truncated sampler, $\log(d\tilde q^\eta_K/dq^\eta)$ is a sum over the
discarded coordinates, so a single pass with common random numbers gives the error at every $K$; each of the
$15$ curves is evaluated at about $953$ values of $K$ (all $K\le600$ and a fine logarithmic grid above). The
relative Monte Carlo standard error is at most $0.8\%$. The estimator reproduces the closed-form TV between
$\mathcal N(0,1)$ and $\mathcal N(0,4)$.

\paragraph{Required dimension.} $K^\star(\varepsilon)$ is the first crossing of $\mathrm{TV}=\varepsilon$, interpolated
log-linearly between grid points. A point is called \emph{asymptotic} if $4k_\eta\le K^\star\le d/8$: well past the
noise level and well inside the ambient dimension, which is the regime of Theorem~2 ($2K\le d$).

\subsubsection{Theorem 1: the Truncation Error Is a Function of the Wiener-Gain Tail}
\label{app:e2-thm1}

Figure~\ref{fig:e2-thm1}(a) shows the truncation error against $K$ for $\beta=1.5$; panel~(b) plots all
$14261$ evaluated points, for every $\beta$ and $\eta$, against $\|(g_i)_{i>K}\|_2$. They fall on a single
curve. Below saturation, the ratio $\mathrm{TV}/\|(g_i)_{i>K}\|_2$ lies in $[0.276,0.302]$
(Table~\ref{tab:e2-thm1}), inside the interval $[1/100,3/2]$ of Theorem~1(a); no point violates either bound or the
explicit Pinsker bound $(\tfrac18\sum_{i>K}x_i^2/(1+x_i))^{1/2}$, which overestimates the error by a factor of only
1.22-1.30. The lower constant $1/100$, inherited from the worst case over all Gaussian pairs, is
therefore conservative by a factor of about $28$ for the truncation structure studied here; the dependence on the
Wiener gains is exact.

\begin{table}[t]
\centering
\caption{Theorem~1 on all truncation curves. Points: evaluated $(\eta,K)$ pairs; ratio range over the points with
$10^{-5}<\|(g_i)_{i>K}\|_2<\tfrac12$ and a Monte Carlo signal-to-noise ratio above $5$. No point violates the
bounds $\tfrac1{100}G_K\le\mathrm{TV}\le\tfrac32G_K$ or the Pinsker bound.}
\label{tab:e2-thm1}
\small
\begin{tabular}{rrrcc}
\toprule
$\beta$ & points & used & $\mathrm{TV}/\|(g_i)_{i>K}\|_2$ & TV/Pinsker \\
\midrule
1.25 & 4779 & 3242 & 0.276-0.290 & 0.78-0.82 \\
1.5 & 4757 & 3763 & 0.276-0.290 & 0.77-0.81 \\
2 & 4725 & 3838 & 0.276-0.302 & 0.77-0.81 \\
\bottomrule
\end{tabular}
\end{table}

\begin{figure*}[t]
\centering
\includegraphics[width=\textwidth]{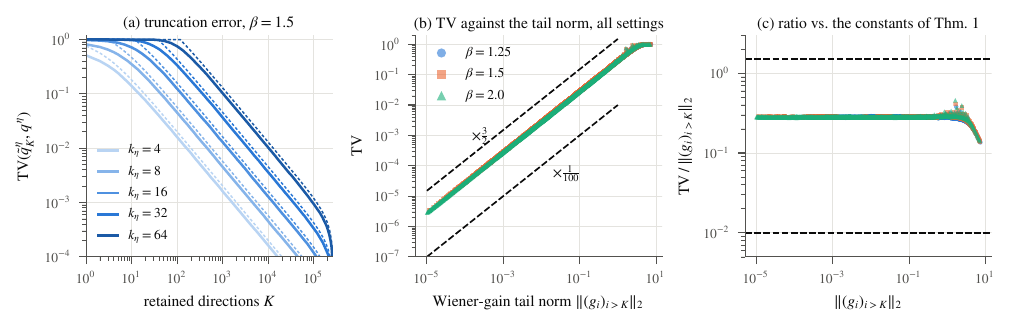}
\caption{Theorem~1. (a) Truncation error $\mathrm{TV}(\tilde q^\eta_K,q^\eta)$ against $K$ for $\beta=1.5$ and five
resolutions (solid), with the explicit Pinsker bound of Theorem~1(a) (dashed). (b) All points for all $\beta$ and
$\eta$ against the Wiener-gain tail norm; dashed lines are the constants $3/2$ and $1/100$ of Theorem~1(a).
(c) Their ratio, which is constant until the error saturates at one.}
\label{fig:e2-thm1}
\end{figure*}

\subsubsection{Theorem 2: How Many Directions Must Be Kept}
\label{app:e2-thm2}

Theorem~2 predicts $K^\star=\Theta_\beta\bigl((C_{\mathrm{sp}}/(\eta^2\varepsilon))^{2/(2\beta-1)}\bigr)$. For each
$\beta$ we regress $\log K^\star$ on $\log(1/\varepsilon)$ and $\log(C_{\mathrm{sp}}/\eta^2)$ jointly over the
$97$ asymptotic points. Both fitted exponents match $2/(2\beta-1)$ (Table~\ref{tab:e2-thm2},
Figure~\ref{fig:e2-thm2}(c)). The largest deviation, about $1\%$ at $\beta=1.25$, is pre-asymptotic: restricted to
$K^\star\ge32k_\eta$ the exponents become $1.334$ and $1.332$, against $1.333$. Plotted against
$(C_{\mathrm{sp}}/(\eta^2\varepsilon))^{2/(2\beta-1)}$, the required dimensions for all $\beta$, $\eta$ and $\varepsilon$
lie on parallel lines of slope one (Figure~\ref{fig:e2-thm2}(b)).

The two parts of Theorem~2 differ in sharpness. The sufficient choice of Theorem~2(a), the smallest $K$ for which the
explicit bound~(11) is at most $\varepsilon$ (that is, $K_{2\varepsilon}(\eta)$ in the notation of Theorem~2), exceeds $K^\star$ by a factor of only 1.15-1.51
(Figure~\ref{fig:e2-thm2}(a), dashed). The necessary bound of Theorem~2(b), stated for $\varepsilon<1/200$, holds at
every point but lies a factor 26-528 below $K^\star$ (dotted). The gap comes from the constants $1/100$ (Theorem~1) and $2^{-\beta}$ used in its proof, consistent with the
remark after Theorem~2 that only the dependence on $\eta$ and $\varepsilon$ matches.

\begin{table*}[t]
\centering
\caption{Theorem~2. Fitted exponents of $1/\varepsilon$ and of $C_{\mathrm{sp}}/\eta^2$ in $K^\star$ (joint
least squares, $95\%$ intervals), over asymptotic points ($4k_\eta\le K^\star\le d/8$) and over
$K^\star\ge16k_\eta$; sharpness of the sufficient choice $K_\varepsilon$ of Theorem~2(a) and of the necessary bound
$K_{\rm nec}$ of Theorem~2(b) ($\varepsilon<1/200$); and the signal-to-noise ratio of the first discarded direction at
the required cut, $x_{K^\star+1}\sqrt{K^\star}/\varepsilon$ (Theorem~2(c)).}
\label{tab:e2-thm2}
\small
\begin{tabular}{rcccrcccccc}
\toprule
 & & \multicolumn{3}{c}{Asymptotic points} & \multicolumn{2}{c}{$K^\star\ge16k_\eta$} & & & \\
\cmidrule(lr){3-5}\cmidrule(lr){6-7}
$\beta$ & $\frac{2}{2\beta-1}$ & $1/\varepsilon$ & $C_{\mathrm{sp}}/\eta^2$ & $n$ & $1/\varepsilon$ & $C_{\mathrm{sp}}/\eta^2$ &
$K_\varepsilon/K^\star$ & $K^\star/K_{\rm nec}$ & $x_{K^\star+1}\sqrt{K^\star}/\varepsilon$ \\
\midrule
1.25 & 1.333 & $1.347\pm0.005$ & $1.340\pm0.007$ & 32 & 1.337 & 1.333 & 1.32-1.51 & 521-528 & 4.24-4.57 \\
1.5 & 1.000 & $1.005\pm0.003$ & $1.001\pm0.003$ & 34 & 0.997 & 0.999 & 1.24-1.33 & 110-114 & 4.77-5.21 \\
2 & 0.667 & $0.670\pm0.002$ & $0.669\pm0.001$ & 31 & 0.668 & 0.668 & 1.15-1.21 & 26-26 & 5.60-6.18 \\
\bottomrule
\end{tabular}
\end{table*}

\begin{figure*}[t]
\centering
\includegraphics[width=\textwidth]{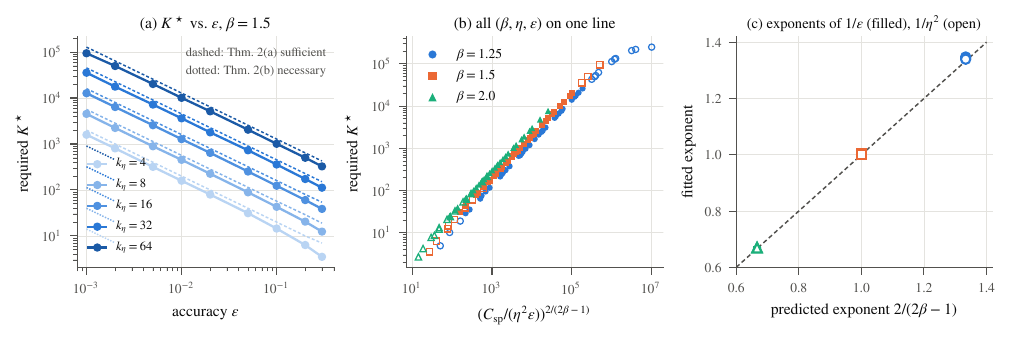}
\caption{Theorem~2. (a) Required dimension against accuracy for $\beta=1.5$ and five resolutions (markers), with the
sufficient choice of Theorem~2(a) (dashed) and the necessary bound of Theorem~2(b) where it applies,
$\varepsilon<1/200$ (dotted). (b) Required dimension for all settings against the predicted variable; open markers are
outside the asymptotic range. (c) Fitted exponents against the prediction $2/(2\beta-1)$.}
\label{fig:e2-thm2}
\end{figure*}

\subsubsection{Where to Cut, and Why the Noise Level Is the Wrong Place}
\label{app:e2-cut}

Theorem~2(c) says that the first discarded direction should have signal-to-noise ratio of order
$\varepsilon/\sqrt{K}$. At the required cut, $x_{K^\star+1}\sqrt{K^\star}/\varepsilon$ is constant to within about
$\pm5\%$ across a $300$-fold range of $\varepsilon$ and a $16$-fold range of $k_\eta$, with a value that depends only on
$\beta$ (Figure~\ref{fig:e2-cut}(a), last column of Table~\ref{tab:e2-thm2}). The cut therefore sits far below the
evaluation level: $x_{K^\star+1}\approx c_\beta\,\varepsilon/\sqrt{K^\star}$ with $c_\beta$ between 4.2 and 6.2.

Truncating at the evaluation level instead ($K=k_\eta$, Corollary~1) leaves an error that grows with resolution
(Figure~\ref{fig:e2-cut}(b), $1\le k_\eta\le 1024$, $d=2^{16}$): for $\beta=1.5$ it is $0.11$ at $k_\eta=1$,
$0.57$ at $k_\eta=16$, and at least $0.99$ for every $\beta$ at $k_\eta=256$, however accurately the
retained directions are sampled. Corollary~1 guarantees at least $1/100$; the observed error is close to one, as
predicted by the $\sqrt{k_\eta}$ growth of the Wiener-gain tail.

\paragraph{Independence of the ambient dimension.} For $\beta=1.5$ and $k_\eta=16$ we repeat the computation with
$d\in\{2^{12},2^{14},2^{16},2^{18}\}$ (Figure~\ref{fig:e2-cut}(c)). Wherever $2K^\star\le d$, the required dimension
differs from its value at $d=2^{18}$ by at most $6.7\%$: it is set by $\eta$ and $\varepsilon$, not by $d$.
Only when $K^\star$ approaches $d$ does the finite spectrum lower it.

\begin{figure*}[t]
\centering
\includegraphics[width=\textwidth]{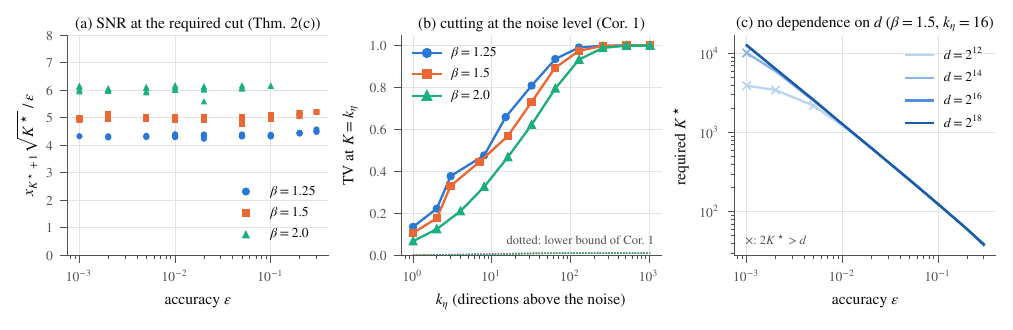}
\caption{(a) Signal-to-noise ratio of the first discarded direction at the required cut, scaled by
$\sqrt{K^\star}/\varepsilon$ (Theorem~2(c)), for all asymptotic points. (b) Error of truncating exactly at the noise
level, $K=k_\eta$, against $k_\eta$ (Corollary~1); dotted lines, the guaranteed lower bound. (c) Required dimension
for four ambient dimensions; crosses mark $2K^\star>d$.}
\label{fig:e2-cut}
\end{figure*}

\subsubsection{Practical Selection Rules}
\label{app:e2-select}

Theorem~2 recommends choosing $K$ from an aggregate spectral-tail budget. We compare this with the rules that
are common in practice, over all 135 combinations of $\beta$, $k_\eta$ and $\varepsilon$
(Table~\ref{tab:e2-select}, Figure~\ref{fig:e2-select}). Both rules based on theory, the explicit Pinsker bound of
Theorem~1(a) and the closed form of Theorem~2(a), meet the target in all cases while retaining at most $1.6$ and $2$
times the minimum. Truncating at the noise level meets it in one case of $135$. Explained-variance thresholds do
not depend on $\eta$ or $\varepsilon$, so they are simultaneously too aggressive for fine resolutions and wasteful for
coarse ones: the $95\%$ rule meets the target in $36\%$ of cases, and the $99.9\%$ rule, although usually
accurate, retains up to $4.7\times10^{4}$ times more directions than needed.

\begin{table}[t]
\centering
\caption{Selection rules over all $(\beta,k_\eta,\varepsilon)$: share of cases meeting $\mathrm{TV}\le\varepsilon$ (upper
end of the Monte Carlo interval), and retained dimension relative to the minimum $K^\star$. Variance rules keep the
smallest $K$ explaining the stated share of $\sum_i\sigma_i^2$; the Theorem~2(a) rule is the smallest $K$ whose explicit
bound~(11) is at most $\varepsilon$; the tail-budget rule is the smallest $K$ whose Pinsker bound in Theorem~1(a) is
at most $\varepsilon$.}
\label{tab:e2-select}
\small
\setlength{\tabcolsep}{4pt}
\begin{tabular}{lccc}
\toprule
Rule & Meets & \multicolumn{2}{c}{$K/K^\star$: median, range} \\
\midrule
Noise level, $K=k_\eta$ & 1\% & 0.034 & 0.000-1.33 \\
95\% variance & 36\% & 0.35 & 0.002-2{,}047 \\
99\% variance & 64\% & 4.10 & 0.008-21{,}479 \\
99.9\% variance & 90\% & 30 & 0.078-47{,}499 \\
Theorem~2(a), $K_{2\varepsilon}(\eta)$ & 100\% & 1.26 & 1.04-2.00 \\
Tail budget, Theorem~1(a) & 100\% & 1.25 & 1.00-1.60 \\
\bottomrule
\end{tabular}
\end{table}

\begin{figure*}[t]
\centering
\includegraphics[width=\textwidth]{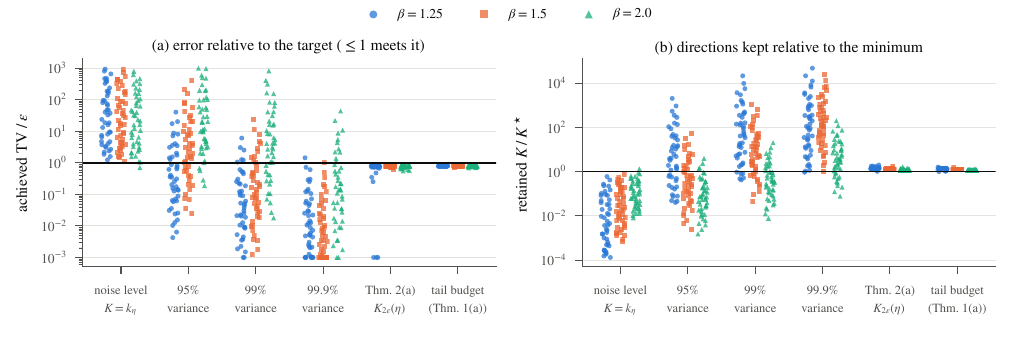}
\caption{Selection rules over all $(\beta,k_\eta,\varepsilon)$. (a) Achieved truncation error relative to the target
(values below $10^{-3}$ are drawn at $10^{-3}$). (b) Retained dimension relative to the minimum.}
\label{fig:e2-select}
\end{figure*}

\subsubsection{Imperfect Samplers and the Exact-Score Sampler}
\label{app:e2-sampler}

\paragraph{Theorem 1(b) and decomposition (10).} We perturb the retained variances of the ideal sampler,
$p=\mathcal N(0,\mathrm{diag}(\sigma_i^2(1+\delta s_i))_{i\le K})$ with fixed random signs $s_i=\pm1$, for
$\beta=1.5$, $k_\eta=16$, $K\in\{16,64,256\}$ and $\delta\in[0,0.5]$ (Figure~\ref{fig:e2-samplers}(a)). In all
18 cases the total error stays above the truncation error (Theorem~1(b)) and below the sum in~(10), within Monte
Carlo error. The sum is loose: (10) measures the sampling error before smoothing, $\mathrm{TV}(p,q_K)$, while the
target smooths every direction at level $\eta$, which hides errors in retained directions with
$\sigma_i^2\ll\eta^2$.

\paragraph{End-to-end sampler (Theorems 3-4).} With the exact Gaussian score, the exponential-integrator
update~(14) maps Gaussians to Gaussians, so the output law of the sampler~(7) is computed exactly by a
variance recursion. We use the settings of Theorem~4 for $\beta=1.5$, $k_\eta=16$, $\lambda_0=1$:
$K=K_\varepsilon(\eta)$, $\sigma^2=C_{\mathrm{sp}}K^{-\beta}$, the stated horizon $T$ and a uniform grid
(Table~\ref{tab:e2-e2e}, Figure~\ref{fig:e2-samplers}(b)). The total error falls to the truncation floor and stays
there: no number of steps removes it, as Theorem~1(b) requires. The target $\mathrm{TV}\le\varepsilon$ is reached
after $32$-$128$ steps, far fewer than the step condition of Theorem~4 requires. At those step counts the
unsmoothed sampling error $\mathrm{TV}(p,q_K)$ is still 0.57-0.78; it is the smoothing at level $\eta$, not
accounted for in the bound, that makes the sampler accurate so early.

Finally, Figure~\ref{fig:e2-samplers}(c) shows the smallest $N$ with unsmoothed sampling error
$\mathrm{TV}(p,q_K)\le\varepsilon/2$ ($\varepsilon=0.05$), the quantity Theorem~3 bounds by
$\tilde O(K/\varepsilon^2)$, as $K$ grows with $\sigma^2=C_{\mathrm{sp}}K^{-\beta}$. It grows like $K^{0.58}$, more
slowly than the linear dependence of the bound. Wiener-gain rates and equal rates on the same subspace have the same
exponent ($0.58$ and $0.58$) and differ by a constant factor ($0.59$-$0.60$), as
Proposition~3 states: the choice of retained directions, not the rates, sets the order of the cost.

\begin{table*}[t]
\centering
\caption{Exact-score sampler in the settings of Theorem~4 ($\beta=1.5$, $k_\eta=16$, Wiener-gain
rates). Trunc.: truncation error $\mathrm{TV}(\tilde q^\eta_K,q^\eta)$; $N$: fewest grid steps with total
$\mathrm{TV}\le\varepsilon$; total: $\mathrm{TV}(\tilde p^\eta,q^\eta)$ at that $N$; $\mathrm{TV}(p,q_K)$: unsmoothed
sampling error at that $N$.}
\label{tab:e2-e2e}
\small
\begin{tabular}{rrcccccc}
\toprule
$\varepsilon$ & $K$ & $\sigma^2$ & $T$ & trunc. & $N$ & total & $\mathrm{TV}(p,q_K)$ \\
\midrule
0.1 & 320 & $1.7\times10^{-4}$ & 17.0 & 0.039 & 32 & 0.074 & 0.78 \\
0.05 & 640 & $6.2\times10^{-5}$ & 19.4 & 0.020 & 64 & 0.027 & 0.63 \\
0.02 & 1600 & $1.6\times10^{-5}$ & 22.6 & 0.008 & 128 & 0.010 & 0.57 \\
\bottomrule
\end{tabular}
\end{table*}

\begin{figure*}[t]
\centering
\includegraphics[width=\textwidth]{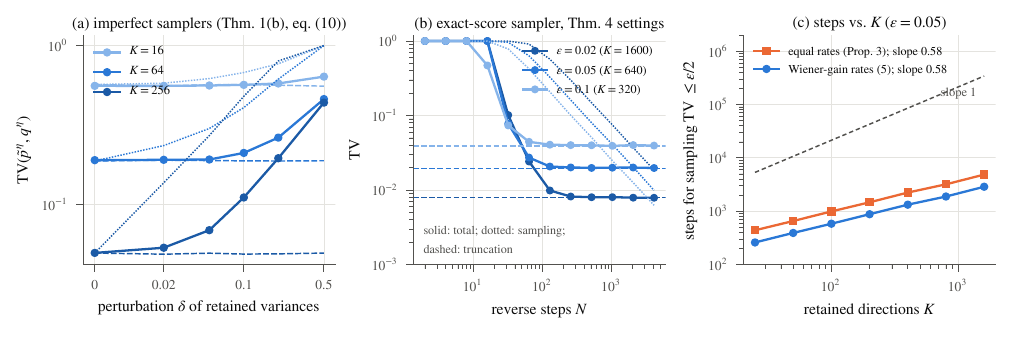}
\caption{(a) Samplers with perturbed retained variances: total error (markers), truncation error (dashed), and the
upper bound of~(10) (dotted). (b) Exact-score sampler in the settings of Theorem~4: total error (solid), truncation
error (dashed), unsmoothed sampling error (dotted). (c) Steps needed for unsmoothed sampling error $\le\varepsilon/2$
against $K$, for Wiener-gain and equal rates; the dashed line has the slope of the bound of Theorem~3.}
\label{fig:e2-samplers}
\end{figure*}

\subsection{Numerical study on image patches}
\label{app:experiments}

Theorems~1 and~2 are proved for Gaussian data with an exact power law,
and Section~5 gives only upper bounds for mixtures. This appendix asks how
well the Gaussian predictions describe a standard model of natural-image
patches, for which every truncation error can be evaluated exactly. No
sampler is trained: the retained coordinates are drawn from the data law
itself, so all errors below are truncation errors
$\mathrm{TV}(\widetilde q^{\,\eta}_K, q^\eta)$ of the ideal truncated
sampler (Definition~1 with $p = q_K$). Following
Zoran and Weiss (2012), we model patches by a Gaussian mixture. All
statements in this appendix are about the fitted mixture, not about the
empirical distribution of image patches. All experiments were run on a single Apple M1 Pro MacBook, without a GPU.

\subsubsection{Data and mixture fit}
\label{app:exp-data}

\paragraph{Patches.}
For each of CIFAR-10 and CIFAR-100 we hold out $5{,}000$ of the
$50{,}000$ training images for validation and fit on the remaining
$45{,}000$; the $10{,}000$ test images are used only for evaluation. From
each image we draw two $8\times8\times3$ patches at random positions,
giving $90{,}000$ training, $10{,}000$ validation and $20{,}000$ test
patches of dimension $d=192$. Pixel values are dequantised by adding
uniform noise on each quantisation bin, and the mean training patch is
subtracted. 

\paragraph{Mixtures.}
We fit zero-mean Gaussian mixtures
$q=\sum_{m=1}^{M}\pi_m\mathcal N(0,C_m)$ with $M\in\{1,2,4,8,16,32\}$
components by EM. EM is initialised by $8$ preliminary iterations and run
for at most $100$ iterations, or until the relative change of the training
log-likelihood falls below $10^{-5}$. A ridge of relative size $10^{-5}$
is added to every covariance, and components that collapse are reseeded.
Table~\ref{tab:exp-gmm} reports the held-out log-likelihood gain over
$M=1$ in bits per dimension. With $M=32$ the held-out log-likelihood
exceeds that of a single Gaussian by $0.57$ (CIFAR-10) and $0.64$
(CIFAR-100) bits per dimension. Validation and test gains agree to within
$0.01$ bits per dimension for every $M$, so the larger mixtures do not
overfit. The gain is still increasing at $M=32$, but with diminishing
returns ($+0.02$ and $+0.03$ bits per dimension from $M=16$ to $32$). For
$M\ge8$ on CIFAR-10 and $M\ge16$ on CIFAR-100, EM reached the iteration
cap before meeting the tolerance. We use the $M=32$ mixture as the data
law throughout. Slower convergence affects how well the mixture
approximates the data, not the validity of the comparisons below, which
all treat the fitted mixture as the data law.

\begin{table}[t]
 \centering
 \caption{Held-out log-likelihood gain over a single Gaussian (bits per
 dimension) and number of EM iterations; $^\ddagger$~iteration cap
 reached, (\,$\cdot$\,)~number of reseeded components.}
 \label{tab:exp-gmm}
 \small
 \setlength{\tabcolsep}{4pt}
\begin{tabular}{@{}rrrrrrr@{}}
\toprule
 & \multicolumn{3}{c}{CIFAR-10} & \multicolumn{3}{c}{CIFAR-100}\\
\cmidrule(lr){2-4}\cmidrule(lr){5-7}
$M$ & val. & test & EM iter. & val. & test & EM iter.\\
\midrule
1 & 0.000 & 0.000 & - & 0.000 & 0.000 & -\\
2 & 0.294 & 0.293 & 16 & 0.333 & 0.333 & 17\\
4 & 0.416 & 0.420 & 58 & 0.453 & 0.452 & 50\\
8 & 0.488 & 0.493 & 100$^{\ddagger}$ & 0.550 & 0.549 & 72\\
16 & 0.549 & 0.558 & 100$^{\ddagger}$ & 0.611 & 0.611 & 100$^{\ddagger}$\\
32 & 0.567 & 0.573 & 100$^{\ddagger}$\,(3) & 0.641 & 0.642 & 100$^{\ddagger}$\,(7)\\
\bottomrule
\end{tabular}
\end{table}

\paragraph{Spectra and Assumption~1.}
For each component, and for the pooled covariance
$\Sigma=\sum_m\pi_mC_m$, we fit $\log\sigma_i^2=\log C-\beta\log i$ by
least squares on the ranks $8\le i\le48=d/4$. We also record the envelope
ratio $\hat C_{\mathrm{sp}}/\hat c_{\mathrm{sp}}$ of the smallest
two-sided bound~(1) with exponent $\hat\beta$ on the same ranks.
The lowest ranks are excluded because the few leading directions deviate
from the power law, and the highest ranks because the spectrum bends
downwards there (Figure~\ref{fig:exp-model}b,c).
Table~\ref{tab:exp-spectra} summarises the fits.

Every component satisfies Assumption~1 on this window. Its fitted
exponent lies in $[1.56,2.57]$ (weighted medians $2.04$ and $2.00$), its
envelope ratio is at most $1.89$, and all mixture weight lies on
components with $\hat\beta_m>1$, the condition used by the convergence
results of Section~4. The pooled covariance has $\hat\beta=1.82$
(CIFAR-10) and $1.84$ (CIFAR-100). It is flatter than the typical
component, as expected when components with different principal
directions are mixed. These exponents exceed the value $\beta\approx1$
suggested in Appendix~B.5 for a single luminance channel with a
$|f|^{-2}$ power spectrum. They are, however, specific to the rank window,
the colour patches and the resolution of CIFAR images, so the exponent
should still be estimated for each application.

\begin{table}[t]
 \centering
 \caption{Power-law fits on ranks $8$-$48$ for the $M=32$ mixtures.
 Weighted quantiles use the mixture weights $\pi_m$.}
 \label{tab:exp-spectra}
 \small
\begin{tabular}{@{}lrr@{}}
\toprule
 & CIFAR-10 & CIFAR-100\\
\midrule
pooled $\hat\beta$ & 1.82 & 1.84\\
pooled $\hat C_{\mathrm{sp}}/\hat c_{\mathrm{sp}}$ & 1.60 & 1.65\\
components: $\hat\beta_m$, weighted median & 2.04 & 2.00\\
\quad weighted 10\%-90\% quant. & 1.70-2.27 & 1.72-2.25\\
\quad range over components & 1.57-2.57 & 1.56-2.57\\
\quad weight with $\hat\beta_m>1$ & 1.00 & 1.00\\
components: $\hat C_{\mathrm{sp}}/\hat c_{\mathrm{sp}}$, w.\ median & 1.44 & 1.44\\
\quad weighted 90\% quant. & 1.78 & 1.65\\
$2/(2\hat\beta-1)$ (pooled) & 0.757 & 0.749\\
\bottomrule
\end{tabular}
\end{table}

\begin{figure*}[t]
 \centering
 \includegraphics[width=\textwidth]{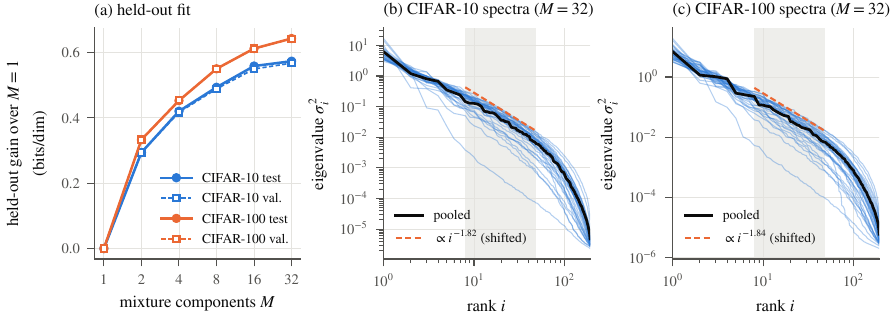}
 \caption{\textbf{(a)}~Held-out log-likelihood gain of the patch mixtures
 over a single Gaussian. \textbf{(b,c)}~Eigenvalues of the $32$ component
 covariances (blue) and of the pooled covariance $\Sigma$ (black). The
 shaded band is the rank window $[8,48]$ used for the power-law fits; the
 dashed line has the fitted pooled slope, shifted vertically for
 visibility.}
 \label{fig:exp-model}
\end{figure*}

\subsubsection{Resolutions, samplers and error estimates}
\label{app:exp-truncation}

\paragraph{Resolution.}
Let $\sigma_1^2\ge\dots\ge\sigma_d^2$ be the eigenvalues of $\Sigma$ and
$v_1,\dots,v_d$ its eigenvectors. For $k_\eta\in\{2,4,8,16,32\}$ we set
$\eta^2=\sigma_{k_\eta}^2$, so that $k_\eta$ is the number of pooled
directions with signal-to-noise ratio at least one, as in Section~3.2. On
CIFAR-10 this gives $\eta^2/\sigma_1^2$ between $0.19$ ($k_\eta=2$) and
$0.0025$ ($k_\eta=32$). Tables~\ref{tab:exp-requiredK}
and~\ref{tab:exp-requiredK-c100} list all values.

\paragraph{Shared-basis truncation.}
Unless stated otherwise, samplers use the shared principal-component
basis of $\Sigma$. For $K=0,\dots,d$ the ideal truncated sampler retains
$U=V_K^\top X$ with $X\sim q$, adds $\mathcal N(0,\eta^2I_K)$, and fills
the remaining $d-K$ coordinates with $\mathcal N(0,\eta^2)$
(Definition~1 with $p=q_K$ and $V$ the eigenbasis of $\Sigma$). Its law
$\widetilde q^{\,\eta}_K$ is a Gaussian mixture with component
covariances
$V\bigl(\begin{smallmatrix}V_K^\top C_mV_K+\eta^2I_K&0\\0&\eta^2I_{d-K}\end{smallmatrix}\bigr)V^\top$.
The target $q^\eta$ has component covariances $C_m+\eta^2I_d$. As a
benchmark we apply the same construction to the moment-matched Gaussian
$\mathcal N(0,\Sigma)$. For it, $\widetilde q^{\,\eta}_K$ and $q^\eta$
are the two Gaussians of Theorem~1, so the benchmark is what Theorem~1
predicts from the covariance alone.

\paragraph{Total variation.}
Both laws have closed-form densities. We estimate
$\mathrm{TV}(\widetilde q^{\,\eta}_K,q^\eta)
=\mathbb E_{Y\sim q^\eta}\bigl[(1-\widetilde q^{\,\eta}_K(Y)/q^\eta(Y))_+\bigr]$
by Monte Carlo with $6{,}000$ samples per configuration.
The standard error is below $3.3\%$ of the estimate in all $1{,}930$
configurations. The same estimator gives the non-Gaussianity
$\mathrm{TV}_{\mathrm{NG}}=\mathrm{TV}(q^\eta,\mathcal N(0,\Sigma+\eta^2I))$
reported in Table~\ref{tab:exp-requiredK}. The required numbers of
directions are
$K^\star=\min\{K:\mathrm{TV}(\widetilde q^{\,\eta}_K,q^\eta)\le\varepsilon\}$
for the mixture and $K_{\mathrm G}$, defined in the same way, for
$\mathcal N(0,\Sigma)$. Using the upper end of a $95\%$ confidence interval
instead of the point estimate increases $K^\star$ by at most two.

\subsubsection{Truncation error and the Gaussian prediction}
\label{app:exp-gaussian}

\begin{figure*}[t]
 \centering
 \includegraphics[width=\textwidth]{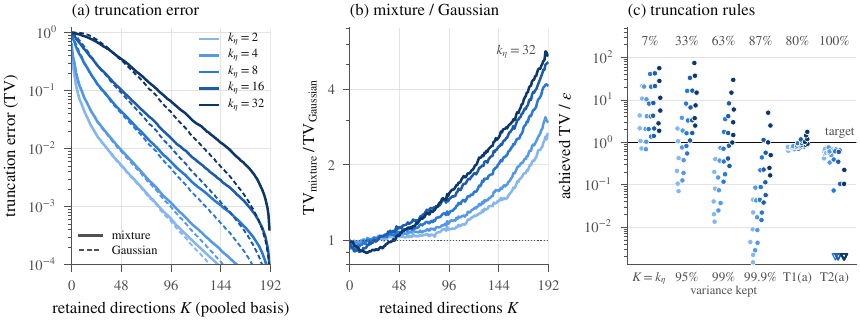}
 \caption{Truncation error on CIFAR-10 patches ($M=32$ mixture, shared
 basis); lighter colours are coarser resolutions $k_\eta$.
 \textbf{(a)}~Mixture error (solid) and the error predicted by
 Theorem~1 for the moment-matched Gaussian (dashed).
 \textbf{(b)}~Ratio of the two.
 \textbf{(c)}~Error achieved by six truncation rules
 (Appendix~\ref{app:exp-rules}), relative to the target
 $\varepsilon\in\{0.3,0.2,0.1,0.05,0.02,0.01\}$, with the percentage of
 the 30 (resolution, $\varepsilon$) pairs that meet the target:
 $K=k_\eta$; the smallest $K$ keeping $95$, $99$ or $99.9\%$ of the
 variance; the smallest $K$ at which the explicit bound of Theorem~1(a)
 for $\mathcal N(0,\Sigma)$ is at most $\varepsilon$ (T1(a)); and the
 smallest $K$ at which bound~(11) with fitted
 $(\hat\beta,\hat C_{\mathrm{sp}})$ is at most $\varepsilon$ (T2(a)).
 Triangles mark T2(a) choices with $K=d$, i.e.\ no truncation.
 CIFAR-100 is in Figure~\ref{fig:exp-truncation-c100}.}
 \label{fig:exp-truncation}
\end{figure*}

\begin{table*}[t]
 \centering
 \caption{Retained directions needed on CIFAR-10 patches.
 $\eta^2/\sigma_1^2$: evaluation level relative to the top pooled
 eigenvalue. $\mathrm{TV}_{\mathrm{NG}}$: distance
 $\mathrm{TV}(q^\eta,\mathcal N(0,\Sigma+\eta^2 I))$ between the smoothed
 mixture and its moment-matched Gaussian. $\mathrm{TV}(k_\eta)$: error
 when exactly the directions above the evaluation level are retained.
 $K^\star$ and $K_{\mathrm G}$: smallest $K$ with error at most
 $\varepsilon$ for the mixture and for $\mathcal N(0,\Sigma)$.
 $K_{\mathrm{T1}}$, $K_{\mathrm{T2}}$: rules T1(a) and T2(a) of
 Appendix~\ref{app:exp-rules}; $^\dagger$~misses the target,
 $^\ast$~no truncation ($K=d$). $\bar K_{\mathrm{comp}}$: expected number
 of retained directions, $\sum_m\pi_mK_m$, when each component uses its
 own basis with per-component budget $\varepsilon$, which guarantees error
 at most $\varepsilon$ by Proposition~4
 (Appendix~\ref{app:exp-components}). CIFAR-100 is in
 Table~\ref{tab:exp-requiredK-c100}.}
 \label{tab:exp-requiredK}
 \small
 \setlength{\tabcolsep}{4.2pt}
\begin{tabular}{@{}rrrr|rrrrr|rrrrr@{}}
\toprule
 & & & & \multicolumn{5}{c}{$\varepsilon=0.1$} & \multicolumn{5}{c}{$\varepsilon=0.01$} \\
$k_\eta$ & $\eta^2/\sigma_1^2$ & $\mathrm{TV}_{\mathrm{NG}}$ & $\mathrm{TV}(k_\eta)$ & $K^\star$ & $K_{\mathrm{G}}$ & $K_{\mathrm{T1}}$ & $K_{\mathrm{T2}}$ & $\bar K_{\mathrm{comp}}$ & $K^\star$ & $K_{\mathrm{G}}$ & $K_{\mathrm{T1}}$ & $K_{\mathrm{T2}}$ & $\bar K_{\mathrm{comp}}$ \\
\midrule
2 & 0.1922 & 0.039 & 0.22 & 6 & 6 & 7 & 9 & 7.3 & 35 & 34 & 39 & 46 & 32.3 \\
4 & 0.1081 & 0.064 & 0.21 & 10 & 10 & 12 & 13 & 11.3 & 47 & 45 & 51 & 71 & 41.4 \\
8 & 0.0234 & 0.199 & 0.41 & 31 & 30 & 35 & 40 & 28.8 & 84 & 79 & 84 & 192$^{\ast}$ & 67.4 \\
16 & 0.0102 & 0.330 & 0.43 & 49 & 46 & 52 & 75 & 41.8 & 108 & 97 & 102$^{\dagger}$ & 192$^{\ast}$ & 82.1 \\
32 & 0.0025 & 0.602 & 0.56 & 84 & 78 & 83$^{\dagger}$ & 192$^{\ast}$ & 66.1 & 152 & 128 & 134$^{\dagger}$ & 192$^{\ast}$ & 107.5 \\
\bottomrule
\end{tabular}
\end{table*}

For the Gaussian benchmark the error is very close to proportional to the
Wiener-gain tail. The ratio $\mathrm{TV}/\|(g_i)_{i>K}\|_2$ has median
$0.281$ over all settings, and stays in $[0.272,0.288]$ whenever
$10^{-3}\le\|(g_i)_{i>K}\|_2\le0.1$. This is the small-gain constant
$1/(2\sqrt\pi)\approx0.282$ (Appendix~\ref{app:exp-constant}), far
inside the universal constants $1/100$ and $3/2$ of Theorem~1(a).

The mixture follows the same curve when the resolution is coarse
(Figure~\ref{fig:exp-truncation}a,b). Over the range where the Gaussian
error lies in $(10^{-3},0.3)$, the ratio of mixture to Gaussian error has
median $1.03$ (CIFAR-10) and $1.05$ (CIFAR-100) at $k_\eta=2$. It grows
with resolution, to a median of $1.80$ and a maximum of $4.0$ at
$k_\eta=32$, and reaches $5.7$ when only a few directions are discarded.
The excess tracks how non-Gaussian the smoothed data are:
$\mathrm{TV}_{\mathrm{NG}}$ rises from $0.04$ to $0.60$ over the same
range of resolutions (Table~\ref{tab:exp-requiredK}). The mixture error
falls below the Gaussian one only when few directions are retained
($K\le36$, errors above $0.07$), and then by at most $13\%$.

Consistent with Proposition~1, the explicit Gaussian upper bound of
Theorem~1(a) does not carry over to the mixture: it is exceeded for
between $69$ and $127$ of the $192$ values of $K$ at each resolution. The
two-sided bound with universal constants,
$G_K/100\le\mathrm{TV}\le\tfrac32G_K$ with
$G_K=\min\{1,\|(g_i)_{i>K}\|_2\}$ computed from $\Sigma$, nevertheless holds for the mixture
in every configuration, although the largest ratio, $1.49\,G_K$ at
$k_\eta=32$ with $K$ close to $d$, leaves little margin.

\subsubsection{How many directions}
\label{app:exp-directions}

Despite the excess error, the number of directions predicted from the
covariance is accurate. The ratio $K^\star/K_{\mathrm G}$ lies in
$[0.98,1.06]$ for all $\varepsilon\le0.1$ and $k_\eta\le8$, and in
$[0.88,1.19]$ overall, the largest gap being at $k_\eta=32$ and
$\varepsilon=0.01$ (Tables~\ref{tab:exp-requiredK}
and~\ref{tab:exp-requiredK-c100}). Because the error decays steeply in
$K$, the excess error of the mixture translates into at most $19\%$ more
directions.

At $\varepsilon=0.01$, the mixture needs between $4.8$ and $18$ times more
directions than the $k_\eta$ that lie above the evaluation level.
Retaining only those $k_\eta$ directions leaves an error between $0.16$
and $0.56$, which increases with $k_\eta$ as Corollary~1 predicts.
Formally, the hypotheses of Corollary~1 and of Theorem~2(b) are not met
here: the spectrum is not an exact power law, Corollary~1 requires
$k_\eta\ge4^{\beta+1}\approx50$, and Theorem~2(b) requires
$\varepsilon<1/200$. The observations are therefore consistent with these
results rather than a test of them.

At coarse resolution the dependence on accuracy matches Theorem~2. Over
$\varepsilon\le0.1$, the fitted slope of $\log K^\star$ against
$\log(1/\varepsilon)$ is $0.80$ (CIFAR-10) and $0.79$ (CIFAR-100) at
$k_\eta=2$, against $2/(2\hat\beta-1)=0.757$ and $0.749$. At finer
resolution the slope decreases, to $0.26$ at $k_\eta=32$. This is because
$K^\star$ then reaches ranks beyond about $100$, where the spectrum falls
faster than any power law (Figure~\ref{fig:exp-model}b,c), and the finite
dimension $d=192$ matters.

\subsubsection{Truncation rules}
\label{app:exp-rules}

We compare rules that use only $\Sigma$, $\eta$ and $\varepsilon$, all of
which are available before training.
\begin{itemize}
 \item \emph{Noise level}: $K=k_\eta$.
 \item \emph{Variance fraction $\rho\in\{0.95,0.99,0.999\}$}: the smallest
 $K$ with $\sum_{i\le K}\sigma_i^2\ge\rho\sum_i\sigma_i^2$. This gives
 $K=21,53,103$ on CIFAR-10 and $20,53,107$ on CIFAR-100, independently of
 $\eta$ and $\varepsilon$.
 \item \emph{T1(a)}: the smallest $K$ with
 $\bigl(\tfrac18\sum_{i>K}x_i^2/(1+x_i)\bigr)^{1/2}\le\varepsilon$, where
 $x_i=\sigma_i^2/\eta^2$ (the explicit bound of Theorem~1(a)).
 \item \emph{T2(a)}: the smallest $K<d$ at which the right-hand side
 of~(11), evaluated with the pooled fit $(\hat\beta,\hat C_{\mathrm{sp}})$,
 is at most $\varepsilon$, and $K=d$ if there is none. Because~(11) bounds
 the error by $\varepsilon$ rather than $\varepsilon/2$, this is
 $K_{2\varepsilon}(\eta)$ in the notation of Theorem~2(a).
\end{itemize}

Figure~\ref{fig:exp-truncation}c shows the outcome.
\begin{itemize}
 \item \emph{Noise level.} It meets the target in $6$ of $60$ cases, all
 with $\varepsilon\ge0.2$.
 \item \emph{Variance fraction.} These rules ignore $\eta$. They miss the
 target at fine resolution and retain far too many directions at coarse
 resolution, with achieved errors down to $0.0015\varepsilon$.
 \item \emph{T1(a).} This is the most economical rule. It meets the target
 in $47$ of $60$ cases, with a median of $1.11K^\star$ directions. Its
 misses occur at $k_\eta\ge16$ and $\varepsilon\le0.1$, exactly where the
 mixture exceeds the Gaussian bound, and the achieved error is then at
 most $1.78\varepsilon$.
 \item \emph{T2(a).} It always meets the target but is conservative, with
 a median of $1.5K^\star$ directions and up to $4K^\star$. In $12$ of
 $60$ cases it retains all $d$ directions, because extrapolating the
 fitted power law overstates the tail beyond rank $100$.
\end{itemize}
The last two rules illustrate the two ways in which the Gaussian theory
can be used. T1(a) evaluates the bound on the measured spectrum. It is
nearly tight, but it inherits the gap between mixture and Gaussian at fine
resolution. T2(a) replaces the spectrum by a fitted power law, which is
simpler but conservative whenever the true spectrum decays faster than
the fit.

\subsubsection{Component bases}
\label{app:exp-components}

The component-basis sampler of Proposition~4 draws $m$ with probability
$\pi_m$ and runs the ideal truncated sampler for $\mathcal N(0,C_m)$ in
the eigenbasis of $C_m$. It retains the smallest $K_m$ for which the
explicit bound of Theorem~1(a) for that component is at most a common
budget $b$. Proposition~4 then bounds the mixture error by
$\sum_m\pi_m\varepsilon_m\le b$, where $\varepsilon_m$ is the truncation
error of component~$m$. We use $24$ log-spaced budgets
$b\in[10^{-3},0.5]$, together with the values of $\varepsilon$.

Proposition~4 is never violated. Since $\sum_m\pi_m\varepsilon_m\le b$
for every budget, every budget meets its guarantee. The mixture error is
between $0.58$ and $0.88$ times the bound
(Figure~\ref{fig:exp-components}a), and the bound is tightest at the
finest resolution. Choosing $b=\varepsilon$ certifies the target without
any Monte Carlo evaluation of the mixture, and achieves errors between
$0.44\varepsilon$ and $0.64\varepsilon$.

Figure~\ref{fig:exp-components}b compares the expected number of
retained directions with the shared-basis requirement $K^\star$. At
$\varepsilon=0.01$, the budget $b=\varepsilon$ uses $8$-$31\%$ fewer
directions per sample than the shared basis needs
($\bar K_{\mathrm{comp}}$ against $K^\star$ in
Tables~\ref{tab:exp-requiredK} and~\ref{tab:exp-requiredK-c100}). Tuning
the budget to the largest value that still meets $\varepsilon$ gives
median savings of $31\%$ (CIFAR-10) and $33\%$ (CIFAR-100). Component
bases help most at small errors and fine resolution, where the shared
basis pays for the mismatch between components. At coarse resolution and
large $\varepsilon$ the two approaches are comparable. These savings
require sampling the component label and storing $M$ bases. The number of
directions also varies widely across components, from $4$ to $63$ at
$k_\eta=2$ and $\varepsilon=0.01$ on CIFAR-10.

\begin{figure*}[t]
 \centering
 \includegraphics[width=\textwidth]{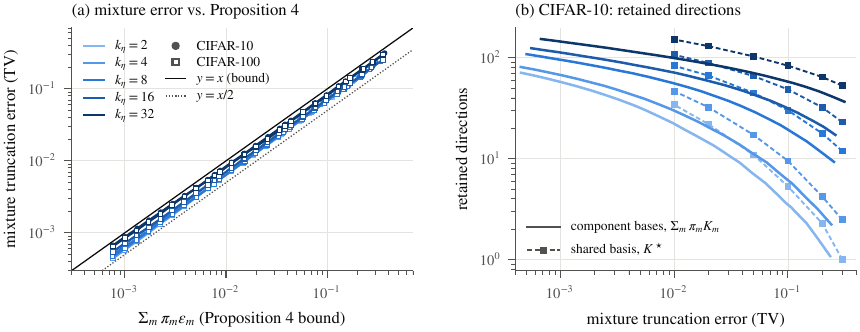}
 \caption{Component-basis truncation. \textbf{(a)}~Mixture truncation
 error against the bound of Proposition~4 for all budgets, resolutions
 and both datasets. \textbf{(b)}~CIFAR-10: expected number of retained
 directions $\sum_m\pi_mK_m$ for component bases (solid), and the number
 $K^\star$ needed with the shared basis (dashed), as functions of the
 achieved mixture error.}
 \label{fig:exp-components}
\end{figure*}

\subsubsection{CIFAR-100}
\label{app:exp-c100}

Figure~\ref{fig:exp-truncation-c100} and
Table~\ref{tab:exp-requiredK-c100} repeat Figure~\ref{fig:exp-truncation}
and Table~\ref{tab:exp-requiredK} for CIFAR-100, and all conclusions above
hold with nearly identical numbers. The non-Gaussianity at each
resolution, the mixture-to-Gaussian error ratios, the ratio
$K^\star/K_{\mathrm G}$ and the success rates of the truncation rules agree
with CIFAR-10 to within the variation across resolutions. The only
differences are in individual cells: at $k_\eta=4$ CIFAR-100 needs fewer
directions, and the noise-level rule meets the target in $4$ of $30$
cases instead of $2$.

\begin{figure*}[t]
 \centering
 \includegraphics[width=\textwidth]{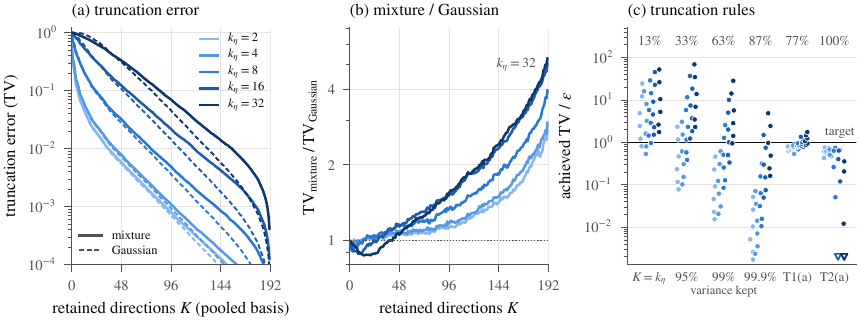}
 \caption{As Figure~\ref{fig:exp-truncation}, for CIFAR-100 patches.}
 \label{fig:exp-truncation-c100}
\end{figure*}

\begin{table*}[t]
 \centering
 \caption{As Table~\ref{tab:exp-requiredK}, for CIFAR-100 patches.}
 \label{tab:exp-requiredK-c100}
 \small
 \setlength{\tabcolsep}{4.2pt}
\begin{tabular}{@{}rrrr|rrrrr|rrrrr@{}}
\toprule
 & & & & \multicolumn{5}{c}{$\varepsilon=0.1$} & \multicolumn{5}{c}{$\varepsilon=0.01$} \\
$k_\eta$ & $\eta^2/\sigma_1^2$ & $\mathrm{TV}_{\mathrm{NG}}$ & $\mathrm{TV}(k_\eta)$ & $K^\star$ & $K_{\mathrm{G}}$ & $K_{\mathrm{T1}}$ & $K_{\mathrm{T2}}$ & $\bar K_{\mathrm{comp}}$ & $K^\star$ & $K_{\mathrm{G}}$ & $K_{\mathrm{T1}}$ & $K_{\mathrm{T2}}$ & $\bar K_{\mathrm{comp}}$ \\
\midrule
2 & 0.1656 & 0.050 & 0.24 & 6 & 6 & 8 & 9 & 7.5 & 36 & 35 & 40 & 47 & 32.6 \\
4 & 0.1255 & 0.063 & 0.16 & 8 & 8 & 10 & 11 & 9.1 & 42 & 41 & 46 & 58 & 37.0 \\
8 & 0.0324 & 0.169 & 0.29 & 23 & 23 & 27 & 29 & 22.5 & 77 & 73 & 78 & 158 & 61.0 \\
16 & 0.0082 & 0.385 & 0.46 & 54 & 50 & 56 & 79 & 43.6 & 119 & 105 & 110$^{\dagger}$ & 192$^{\ast}$ & 86.2 \\
32 & 0.0026 & 0.610 & 0.53 & 85 & 78 & 83$^{\dagger}$ & 188 & 64.6 & 155 & 132 & 137$^{\dagger}$ & 192$^{\ast}$ & 107.6 \\
\bottomrule
\end{tabular}
\end{table*}

\subsubsection{The small-gain constant}
\label{app:exp-constant}

The ratio $\mathrm{TV}/\|(g_i)_{i>K}\|_2\approx0.282$ observed for the
Gaussian benchmark has a simple explanation. Under $q^\eta$, write the
discarded coordinates as $\sqrt{\sigma_i^2+\eta^2}\,z_i$ with
$z_i\sim\mathcal N(0,1)$ independent. The log-likelihood ratio of
$\widetilde q^{\,\eta}_K$ to $q^\eta$ is then
\begin{align*}
 L&=\tfrac12\sum_{i>K}\Bigl[-\log(1-g_i)-\tfrac{g_i}{1-g_i}\,z_i^2\Bigr]\\
  &=-\tfrac12\sum_{i>K}g_i\,(z_i^2-1)+R,
\end{align*}
where the remainder $R$ is of order $\sum_{i>K}g_i^2(1+z_i^2)$.
When every discarded gain is small, the central limit theorem makes $L$
approximately $\mathcal N(0,\|(g_i)_{i>K}\|_2^2/2)$. Since
$\mathrm{TV}=\tfrac12\mathbb E|e^L-1|\approx\tfrac12\mathbb E|L|$ for small
$L$, this gives
\[
 \mathrm{TV}(\widetilde q^{\,\eta}_K,q^\eta)\approx
 \frac{\|(g_i)_{i>K}\|_2}{2\sqrt\pi}\approx0.282\,\|(g_i)_{i>K}\|_2 .
\]
Making this rigorous, with a Berry-Esseen bound on $L$, would sharpen the
constants of Theorem~1(a) to $1/(2\sqrt\pi)+o(1)$ in the regime
$\max_{i>K}g_i\to0$, $\|(g_i)_{i>K}\|_2\to0$. Under a power law with
$K\ge k_\eta$, this is the regime relevant for small $\varepsilon$.




\end{document}

%% file: roadmap_figure.tex
\begin{figure}[!p]
\centering
\begin{tikzpicture}[
  font=\scriptsize,
  >={Stealth[length=2.2mm]},
  box/.style={draw, rounded corners=2pt, align=center, inner sep=3pt,
              text width=#1, fill=white},
  box/.default=7.4cm,
  main/.style={box=#1, very thick, fill=blue!6},
  main/.default=7.4cm,
  tool/.style={box=#1, fill=gray!8},
  tool/.default=7.4cm,
  group/.style={draw, dashed, rounded corners=4pt, inner sep=6pt},
  grouplabel/.style={font=\scriptsize\bfseries, fill=white, inner sep=1pt},
  arr/.style={->, thick},
]
\node[box=16.4cm, fill=yellow!10] (setup) at (8.4,0) {%
  \textbf{Setup (Section~\ref{sec:setup})}:
  power-law spectrum (Assumption~\ref{ass:powerlaw}),
  truncated samplers (Definition~\ref{def:truncated-class}),
  error decomposition \eqref{eq:error-decomposition}\\[2pt]
  $\TV(\widetilde p^\eta,q^\eta)\;\le\;
   \underbrace{\TV(p,q_K)}_{\text{sampling}}
   \;+\;
   \underbrace{\TV(\widetilde q_K^{\,\eta},q^\eta)}_{\text{truncation}}$};

\node[tool] (L1) at (4.2,-2.1) {%
  Exact divergence and self-contained bounds
  (Lemma~\ref{lem:sharp-residual});\\
  two-sided Gaussian TV bound \citep{devroye2018total}};
\node[main] (L2) [below=5mm of L1] {%
  \textbf{Theorem~\ref{thm:truncation-cost}} (cost of truncation)\\
  $\TV\asymp\min\{1,\|(g_i)_{i>K}\|_2\}$ in the Wiener gains;\\
  lower bound holds for \emph{every} truncated sampler};
\node[main] (L3) [below=5mm of L2] {%
  \textbf{Theorem~\ref{thm:num-directions}} (number of retained directions)\\
  (a) sufficient: $K_\varepsilon(\eta)$ directions\\
  (b) necessary: $K\gtrsim_\beta\bigl(C_{\mathrm{sp}}/(\eta^2\varepsilon)\bigr)^{2/(2\beta-1)}$\\
  (c) cut where the signal-to-noise ratio is $\asymp\varepsilon/\sqrt K$};
\node[box] (L4) [below=5mm of L3] {%
  \textbf{Corollary~\ref{cor:noise-floor}}:
  cutting at the noise level leaves $\TV\ge1/100$};

\node[tool] (R1) at (12.6,-2.1) {%
  Assumptions~\ref{ass:lipschitz}--\ref{ass:finite-kl};
  Girsanov and score-perturbation lemmas
  (Appendix~\ref{sec:app:lemmas})};
\node[tool] (R2) [below=5mm of R1] {%
  Theorems~\ref{thm:tv-dim}--\ref{thm:tv-trace} (Appendix~\ref{app:proofs-convergence}):
  TV bounds for arbitrary diagonal rates,\\
  dimension enters via $\tr(\Lambda)$ and $\E\|U\|_\Lambda^2$};
\node[main] (R3) [below=5mm of R2] {%
  \textbf{Theorem~\ref{thm:convergence}} (convergence on retained coordinates)\\
  Wiener-gain rates and $K\le d_{\mathrm{eff}}^{\sigma}$ give\\
  $N=\widetilde O(K/\varepsilon^2)$ steps, independent of $d$\\
  (equal rates give the same order, Proposition~\ref{prop:isotropic})};
\node[tool] (R4) [below=5mm of R3] {%
  Lemma~\ref{lem:active-sigma}: sampling error with explicit
  dependence on $\sigma$ (Gaussian data)};

\begin{scope}[on background layer]
  \node[group, draw=blue!60, fit=(L1)(L4), inner ysep=9pt] (GL) {};
  \node[group, draw=red!60, fit=(R1)(R4), inner ysep=9pt] (GR) {};
\end{scope}
\node[grouplabel, text=blue!70!black, anchor=north west]
  at ($(GL.north west)+(0.12,-0.06)$)
  {Truncation error, Gaussian data (Section~\ref{sec:sharp})};
\node[grouplabel, text=red!70!black, anchor=north west]
  at ($(GR.north west)+(0.12,-0.06)$)
  {Sampling error, any data (Section~\ref{sec:convergence})};

\path let \p1=(GL.south), \p2=(GR.south) in
  node[main=10.5cm, fill=green!8, anchor=north] (E2E)
  at (8.4,{min(\y1,\y2)-0.75cm}) {%
  \textbf{Theorem~\ref{thm:end-to-end}} (end-to-end complexity, exact scores)\\
  $K=K_\varepsilon(\eta)$,\quad $\sigma^2=C_{\mathrm{sp}}K^{-\beta}$
  \;$\Longrightarrow$\;
  $\TV(\widetilde p^\eta,q^\eta)\le\varepsilon$ with
  $N=\widetilde O\bigl(\varepsilon^{-4\beta/(2\beta-1)}\bigr)$};

\node[box=5.1cm, anchor=north] (X2) at ($(E2E.south)+(0,-1.05)$) {%
  \textbf{Theorems~\ref{thm:empirical}--\ref{thm:decouple}} (estimated basis)\\
  extra error $\propto\Delta_n/(\eta^2\delta_{\mathrm{gap}})$;
  oracle score; lemmas in Appendix~\ref{app:lemmas}};
\node[box=5.1cm, anchor=north] (X1) at ($(X2.north)+(-5.45,0)$) {%
  \textbf{Proposition~\ref{prop:full-bound}} (non-Gaussian data)\\
  truncation error $\varepsilon_{\mathrm{PCA}}\le\sqrt{I/2}$;
  upper bound only};
\node[box=5.1cm, anchor=north] (X3) at ($(X2.north)+(5.45,0)$) {%
  \textbf{Proposition~\ref{prop:gmm-lift}} (patch Gaussian mixtures)\\
  $\TV\le\sum_m\pi_m\varepsilon_m$, componentwise;
  upper bound only};
\begin{scope}[on background layer]
  \node[group, draw=gray!70, fit=(X1)(X2)(X3), inner ysep=9pt] (GX) {};
\end{scope}
\node[grouplabel, text=black!70, anchor=north west]
  at ($(GX.north west)+(0.12,-0.06)$)
  {Extensions (Section~\ref{sec:extensions})};

\node[box=16.4cm, fill=orange!8, anchor=north] (take) at ($(GX.south)+(0,-0.7)$) {%
  \textbf{Takeaway (Section~\ref{sec:conclusion})}:
  set the truncation level by the target accuracy, not by the noise level;
  run the diffusion below the evaluation level;
  a better score network reduces only the sampling error.};

\draw[arr] ($(setup.south -| L1.north east)+(-0.9,0)$) --
  node[left, font=\tiny]{truncation term}
  ($(L1.north east)+(-0.9,0)$);
\draw[arr] ($(setup.south -| R1.north east)+(-0.9,0)$) --
  node[left, font=\tiny]{sampling term}
  ($(R1.north east)+(-0.9,0)$);
\draw[arr] (L1) -- (L2);  \draw[arr] (L2) -- (L3);  \draw[arr] (L3) -- (L4);
\draw[arr] (R1) -- (R2);  \draw[arr] (R2) -- (R3);  \draw[arr] (R3) -- (R4);
\draw[arr] (L3.east) -- ($(GL.east |- L3.east)+(0.2,0)$)
  |- node[pos=0.75, above, font=\tiny]{$K=K_\varepsilon(\eta)$ from (a)} ($(E2E.north west)+(1.2,0.35)$)
  -- ($(E2E.north west)+(1.2,0)$);
\draw[arr] (R4.south) -- node[right, font=\tiny]{sampling error}
  (R4.south |- E2E.north);
\draw[arr] (E2E.south) -- (GX.north -| E2E.south);
\draw[arr] (GX.south -| take.north) -- (take.north);
\end{tikzpicture}
\caption{Roadmap of the results.  The error decomposition of
Section~\ref{sec:setup} splits the problem into a truncation error, computed
exactly for Gaussian data in Section~\ref{sec:sharp}, and a sampling error on
the retained coordinates, bounded for general data in
Section~\ref{sec:convergence}; the two combine in
Theorem~\ref{thm:end-to-end}.  Grey boxes are supporting results in the
appendix; bold boxes are the main theorems.}
\label{fig:roadmap}
\end{figure}